\documentclass[hidelinks,onefignum,onetabnum]{siamart250211}

\usepackage{tikz}
\usepackage{mathtools}
\usepackage{graphicx}
\usepackage[caption=false,font=footnotesize,labelfont=bf]{subfig}
\usepackage{placeins}
\usepackage{needspace}
\usepackage{enumitem}

\usepackage{amsfonts}
\usepackage{graphicx}
\graphicspath{{./}}
\usepackage{epstopdf}
\usepackage{setspace} 
\usepackage{bm}
\usepackage{algpseudocode}
\usepackage{mathtools}
\usepackage{mathabx}
\usepackage{multirow}

\ifpdf
  \DeclareGraphicsExtensions{.eps,.pdf,.png,.jpg}
\else
  \DeclareGraphicsExtensions{.eps}
\fi

\newsiamremark{remark}{Remark}
\newsiamremark{hypothesis}{Hypothesis}
\crefname{hypothesis}{Hypothesis}{Hypotheses}
\newsiamthm{claim}{Claim}

\headers{Fast algorithm for computing pseudospectra}{K. Deng and K. Xu}

\title{Recycling singular and projection subspaces for pseudospectra computation}

\author{Kuan Deng\thanks{School of Mathematical Sciences, University of Science and Technology of China, 96 Jinzhai Road, Hefei 230026, Anhui, China (\email{dengkuan@mail.ustc.edu.cn}, \email{kuanxu@ustc.edu.cn}).} \and Kuan Xu\footnotemark[2]}
\usepackage{amsopn}

\def\ema{\epsilon_{\mathrm{m}}}

\makeatletter
\newcommand{\AlgoOutputRule}{%
  \par\nointerlineskip\vspace{0.7ex}%
  \Statex\hspace*{-\algorithmicindent}%
  \rule{\dimexpr\linewidth+\algorithmicindent\relax}{0.4pt}%
  \par\nointerlineskip\vspace{0.4ex}%
}
\makeatother

\ifpdf
\hypersetup{
pdftitle={Recycling singular and projection subspaces for pseudospectra computation},
pdfauthor={D. Doe, P. T. Frank, and J. E. Smith}
}
\fi

\definecolor{psePrelim}{rgb}{0.0000,0.4470,0.7410}
\definecolor{pseInvLanc}{rgb}{0.8500,0.3250,0.0980}
\definecolor{pseRecySing}{rgb}{0.4940,0.1840,0.5560}
\definecolor{pseLobpcg}{rgb}{0.4660,0.6740,0.1880}
\definecolor{pseUpdateRecy}{rgb}{0.3010,0.7450,0.9330}
\definecolor{pseQr}{rgb}{0.9290,0.6940,0.1250}
\colorlet{psePredPts}{pseInvLanc}
\colorlet{psePrecond}{psePrelim}
\newcommand{\pselegendcenter}[1]{%
  \raisebox{\dimexpr(\ht\strutbox-\dp\strutbox-\height+\depth)/2\relax}{#1}%
}
\DeclareRobustCommand{\pselegendbox}[1]{%
  \pselegendcenter{\textcolor{#1}{\rule{2.0ex}{1.5ex}}}%
}
\DeclareRobustCommand{\pselegenddot}[1]{%
  \pselegendcenter{\tikz{\path[draw=#1,fill=#1] (0,0) circle[radius=0.41ex];}}%
}
\DeclareRobustCommand{\pselegendcross}[1]{%
  \pselegendcenter{\tikz{%
    \draw[#1,line width=1.02pt,line cap=round] (-0.56ex,-0.56ex) -- (0.56ex,0.56ex);
    \draw[#1,line width=1.02pt,line cap=round] (-0.56ex,0.56ex) -- (0.56ex,-0.56ex);
  }}%
}
\makeatletter
\renewcommand{\AlgoOutputRule}{%
  \par\nointerlineskip\vspace{0.7ex}% 紧贴上一行；需要更贴就改为 -0.7ex/-0.8ex
  \Statex\hspace*{-\algorithmicindent}%
  \rule{\dimexpr\linewidth+\algorithmicindent\relax}{0.4pt}%
  \par\nointerlineskip\vspace{0.4ex}
}

\def\mG{\mathcal{G}}
\def\mH{\mathcal{H}}
\def\mI{\mathcal{I}}

\def\mL{\mathcal{L}}

\def\mN{\mathcal{N}}
\def\mS{\mathcal{S}}
\def\mP{\mathcal{P}}

\def\mO{\mathcal{O}}

\def\mV{\mathcal{V}}

\def\ema{\epsilon_{\mathrm{m}}}
\begin{document}

\maketitle

% REQUIRED
\begin{abstract}
Computing matrix pseudospectra over a prescribed region requires evaluating the smallest singular value of $C-zI$ at a large number of grid points, which can be prohibitively expensive for large-scale matrices. We develop a recycling-based framework for accelerating such computations for both dense and sparse matrices. The main idea is to exploit the correlation between singular value problems at neighboring grid points by adaptively recycling singular subspaces computed at previously visited points by an iterative SVD solver. We develop fast Rayleigh--Ritz-SVD procedures for extracting Ritz singular pairs from the recycled singular subspaces, together with fast residual evaluation procedures, with an overall cost that scales linearly with the number of recycled samples. When the iterative SVD solver admits preconditioning, we propose using a two-level preconditioner whose projection subspaces are recycled. Numerical experiments demonstrate that the proposed recycling strategies yield substantial speedups over existing methods while maintaining the accuracy of the computed pseudospectra.
\end{abstract}

% REQUIRED
\begin{keywords}
pseudospectra, singular value decomposition, eigenvalue decomposition, subspace recycling, LOBPCG, preconditioning
\end{keywords}

% REQUIRED
\begin{MSCcodes}
65F15, % Numerical computation of eigenvalues and eigenvectors of matrices
65F08, % Preconditioners for iterative methods
15A18, % Eigenvalues, singular values, and eigenvectors
47A10  % Spectrum, resolvent 
\end{MSCcodes}

\section{Introduction}\label{sec:intro}
The pseudospectrum provides a more complete characterization of the spectral behavior of a nonnormal matrix than the spectrum alone. It has been widely used in applications such as stability analysis of dynamical systems, control theory, and numerical analysis \cite{tre4,tre2}.

For $C \in \mathbb{C}^{n\times n}$, the $\epsilon$-pseudospectrum is defined by
\begin{align}
\sigma_\epsilon(C) = \{z \in \mathbb{C} : \lVert(C-zI)^{-1}\rVert > \epsilon^{-1}\}, \label{defSquare}
\end{align}
with the convention that $\lVert(C-zI)^{-1}\rVert=\infty$ whenever $z$ is an eigenvalue of $C$. In the $2$-norm, \eqref{defSquare} is equivalently characterized in terms of the smallest singular value:
\begin{align}
\sigma_\epsilon(C) = \{z \in \mathbb{C} : \sigma_{\min}(C-zI) < \epsilon\}. \label{defSquareSVD}
\end{align}

For $C \in \mathbb{C}^{m \times n}$ with $m \geq n$, the definition in \cref{defSquareSVD} also applies, where $I$ is interpreted as the $m \times n$ rectangular identity matrix \cite{wri2}. One important application of the pseudospectra of rectangular matrices is the computation of the spectra of infinite-dimensional operators via the pseudospectra of finite-dimensional \emph{matrix foldings}, i.e., rectangular truncations of an infinite-dimensional matrix \cite{han,col1,col2}.

In view of \cref{defSquareSVD}, the numerical computation of $\epsilon$-pseudospectra for $C \in \mathbb{C}^{m \times n}$ over a prescribed region $\Omega \subset \mathbb{C}$ reduces to evaluating the smallest singular value of
\begin{align*}
C(z) := C-zI,
\end{align*}
at each point $z$ on a grid $\mathcal{G}$ covering $\Omega$. This task is computationally demanding for two reasons. First, the grid $\mathcal{G}$ is often large, so even moderately expensive computations at individual grid points can result in a substantial overall computational cost. Second, the matrices $C(z)$ are often ill-conditioned, and the smallest singular values may be poorly separated. As a consequence, Lanczos- or Krylov-type iterative methods may converge very slowly or exhibit stagnation.

For a square dense matrix $C$, the core \textsc{EigTool} method \cite{tre1,tre2} nevertheless computes the largest eigenvalue of $N(z)^{-1}$, where
\begin{align*}
N(z) := (C-zI)^*(C-zI),
\end{align*}
using the inverse Lanczos method. Solving the linear systems with $N(z)$ as the coefficient matrix by means of QR or LU factorization in Lanczos iterations would be expensive, since such systems must be solved repeatedly at every grid point. An ingenious acceleration proposed by Lui \cite{lui,tre1} is to first triangularize $C$ via a Schur factorization, thereby avoiding the need to triangularize $N(z)$ at each grid point. This preliminary reduction lowers the overall complexity from $\mO(|\mG|n^3)$ to $\mO(n^3+|\mG|n^2)$, where $|\mG|$ denotes the number of grid points in $\mG$.

For rectangular matrices, Wright and Trefethen proposed a preliminary reduction based on a QR or QZ factorization of $C$, which substantially reduces the cost of QR factorization associated with the inverse Lanczos iteration at every grid point of $\mG$ \cite{wri2}. This reduction lowers the leading cost from $\mO(|\mG|mn^2)$, corresponding to performing a QR factorization independently at each grid point, to $\mO(mn^2 + |\mG|n^3)$ when $m \geq 2n$, and to $\mO(n^3 + |\mG|(m-n)n^2)$ when $n < m < 2n$. 

For sparse matrices, the preliminary reductions based on Schur or QZ factorizations are often impractical, since they typically produce dense factors. Consequently, the standard approach is to perform a sparse QR or LU factorization of $C(z)$ at each grid point and then apply inverse Lanczos- or Krylov-type iterative methods \cite{bag,kok,lar}. These approaches can be expensive when the factors exhibit substantial fill-in, which is often the case for large sparse matrices. 

Despite these differences in the preliminary reductions, we shall refer to these approaches collectively as the \emph{Lanczos-based} method, which serves as the baseline method in our numerical experiments.

To deal with large matrices, \emph{projection-based} methods first project $C$ onto a low-dimensional subspace, such as an invariant subspace or a Krylov subspace \cite{toh,tre1,tre2}. The reduced problem can then be handled using the same small-scale techniques employed for dense matrices, and the resolvent norm of the reduced matrix serves as a surrogate for that of the full matrix. However, the quality of the approximation may depend strongly on the choice of subspace, and reliable error control is often difficult.

Another line of acceleration exploits the fact that $C(z)$ varies smoothly with $z$ over a fine grid. The idea of \emph{continuation} is to reuse information from a neighboring grid point; for example, one may use the right singular vector associated with $\sigma_{\min}(C(z))$ as an initial guess for inverse Lanczos iteration at $z+\Delta z$ \cite{lui}. While such warm starts can reduce the iteration count in favorable situations, they may be unreliable when $\sigma_{\min}(C(z))$ belongs to a cluster of singular values, potentially leading to misconvergence \cite[\S 39]{tre2}.

More generally, one may regard $C(z)$ as a parameterized matrix. Along these lines, Sirkovi\'c proposed a \emph{reduced-basis} approach to pseudospectra computation \cite{sir}. While effective in some settings, this method has several limitations: (1) it can be difficult to handle small $\epsilon$ (when $C(z)$ is close to singular), since the Rayleigh--Ritz procedure is applied directly to $N(z)$; (2) computing residuals of the Ritz pairs can be expensive, with a cost that scales like $p^2$, where $p$ denotes the number of samples; and (3) the eigenvalues and eigenvectors of $C$ in the target region $\Omega$ are required to construct the initial sample set, but such information is not always available. An improved reduced-basis approximation for the smallest singular value was proposed in \cite{man}, which addresses limitation (1) of \cite{sir} through a two-sided procedure, but still suffers from limitations (2) and (3). Related approaches for parameterized eigenvalue problems have also been studied; see, e.g., \cite{ruy}.

From a broader numerical linear algebra perspective, pseudospectra computation on a grid can be viewed as solving a sequence of closely related problems. This viewpoint connects to a large body of work on \emph{recycling} techniques. For instance, in the context of linear systems, recycling Krylov subspaces can reduce iteration counts for sequences of similar systems \cite{par,ahu,car4}. Related ideas involving the recycling or updating of preconditioners have also been studied for sequences of matrices \cite{ben1,meu,bel,ber,cal,teb,dui,car1,zah}. In the context of eigenvalue problems, recycling invariant subspaces has been exploited to solve sequences of similar eigenproblems \cite{car2}. However, this approach only provides an initial subspace for Krylov--Schur iteration and does not fundamentally improve convergence rates.

In this work, we exploit the idea of recycling to develop fast algorithms for pseudospectra computation within a unified framework for both dense and sparse matrices. Specifically, we recycle singular subspaces spanned by singular vectors computed at adaptively selected grid points previously visited during the grid traversal. Unlike projection-based or reduced-basis methods, where the subspace is fixed, our adaptive strategy makes the approach more flexible for computing pseudospectra over a wide region $\Omega$. Moreover, we develop a fast Rayleigh--Ritz-SVD procedure for extracting Ritz singular pairs from the recycled singular subspaces, together with efficient procedures for evaluating the corresponding residuals, at a computational cost proportional to the number of recycled samples $p$. If the iterative SVD solver admits preconditioning, we can further accelerate the computation using a two-level preconditioner \cite{tan} by recycling the projection subspace associated with the preconditioner in a manner analogous to singular-subspace recycling, leading to a new preconditioner recycling strategy based on projection-subspace updates.

In principle, the proposed recycling strategies can be combined with a variety of iterative SVD solvers. In this work, however, we instantiate our algorithms using a locally optimal block preconditioned conjugate gradient (LOBPCG) method \cite{kny} for computing the smallest singular pairs of a given matrix. This method, referred to as LOBPCG-SVD, is described in \cref{sec:lobpcg} of the supplementary material accompanying this paper. We choose LOBPCG-SVD primarily because it naturally accommodates both unpreconditioned and preconditioned implementations. Extensive numerical experiments demonstrate that the two recycling strategies yield substantial reductions in computational cost while maintaining accuracy.

Throughout the paper, we denote by the asterisk $*$ the conjugate transpose of a matrix or scalar, by $i$ the imaginary unit, by $\ema$ machine epsilon, and by $I_n$ the $n \times n$ identity matrix. For submatrix notation, the colon $:$ indicates an index range. For a matrix $A$, $A_{j{:}k,\ell{:}q}$ denotes the submatrix consisting of rows $j$ through $k$ and columns $\ell$ through $q$. Moreover, $A_{\ell{:}q}$ denotes the submatrix consisting of columns $\ell$ through $q$, and $A_{j{:}k,{:}}$ denotes the submatrix consisting of rows $j$ through $k$. The symbol $\sim$ is used as a placeholder for quantities that need not be specified explicitly.

The rest of the paper is organized as follows. In \cref{sec:triRecy}, we develop a fast algorithm for pseudospectra computation based on triangularization and singular-subspace recycling. In \cref{sec:recyprojection}, we show that this algorithm can be further accelerated when the SVD solver admits preconditioning by a two-level preconditioner whose projection subspace is recycled using a strategy parallel to that for singular subspaces. Numerical experiments reported in \cref{sec:exp} demonstrate the significant speedups achieved by the proposed recycling strategies. We conclude in \cref{sec:out} with directions for future work.

\section{Computing pseudospectra by triangularization and singular subspace recycling}\label{sec:triRecy}
In this section, we propose a recycling strategy for exploiting useful information contained in the singular subspaces computed at previously visited grid points. This recycling approach features a fast Rayleigh--Ritz-SVD procedure and an adaptive update of the recycling subspace. Combining these ingredients yields a new fast algorithm for computing pseudospectra.

\subsection{Triangularization and preliminary reduction}\label{sec:prelimred}
The standard methods for computing pseudospectra usually require triangularizing $C(z)$ at every grid point in $\mG$ to cope with the ill-conditioning of $C(z)$. We reduce the cost of such triangularizations by first performing a preliminary reduction that replaces $C(z)$ by the matrix pencil
\begin{align}
M(z) := M - zS, \label{genForm}
\end{align}
whose singular values coincide with those of $C(z)$. This is followed by the thin QR factorization
\begin{align}
M(z) = Q(z) R(z). \label{triGen}
\end{align}
An iterative method, e.g., Lanczos iteration or LOBPCG, can then be applied to $R(z)^{-1}$ or $R(z)^{-*}$ for the largest singular values, from which the smallest singular values of $C(z)$ are obtained from the reciprocal \cite{tre1, tre2, bag, kok, lar}.

The specific preliminary reduction, which is crucial for minimizing the cost of the subsequent QR factorization that is performed repeatedly across the grid, depends on whether the original matrix is dense or sparse, and square or rectangular. For example, when $C$ is dense and square, one may perform a Schur decomposition of $C$ and take $M$ to be the resulting upper triangular factor and $S=I$ \cite{tre1, tre2}. Since $M(z)$ is already upper triangular, the QR factorization in \cref{triGen} can be skipped. For sparse matrices, $M$ and $S$ are obtained by applying suitable fill-reducing column permutations to $C$ and $I$ \cite{dav3}, thereby reducing the cost of the sparse QR factorization in \cref{triGen}. See \cref{sec:prelred} in the supplementary material for details on the preliminary reductions for different types of matrices.

\subsection{Recycling the singular subspaces}\label{sec:recycsing}
For $M(z)$, consider the associated Gram matrix
\begin{align}
T(z) := (M-zS)^*(M-zS) = M^*M - zM^*S - z^*S^*M + |z|^2S^*S, \label{affineN}
\end{align}
which depends on $z$ through an affine decomposition. For a general $z \in \mathbb{C}$, let
\begin{align*}
T(z)X = X\Sigma^2, \qquad \Sigma = \operatorname{diag}(\sigma_{1}, \dots, \sigma_{r}),
\end{align*}
where $\Sigma$ contains the smallest singular values of $M(z)$ and $X$ is the corresponding matrix of right singular vectors. Motivated by the ideas of continuation \cite{lui}, reduced-basis methods \cite{sir}, and subspace methods \cite{man,ruy}, we collect the right singular subspaces obtained at some previously visited grid points for extracting useful information. Let $s_1,\dots,s_p \in \mG$ denote these points, which we refer to as the \emph{recycling points}, and let $\mS = \{s_1,\dots,s_p\}$ denote the corresponding \emph{recycling set}. Moreover, let $X^k \in \mathbb{C}^{n\times r}$ denote the matrix of right singular vectors computed at $s_k$ for $k=1, \dots, p$, and define the \emph{recycling subspace}
\begin{align*}
\mV := \operatorname{span}\{X^1,\dots,X^p\}
     = \operatorname{span}\{V\},
\end{align*}
where $V \in \mathbb{C}^{n\times pr}$ is an orthonormal basis for $\mV$. From now on, we assume that $pr \ll m,n$, which is confirmed by our extensive numerical experiments.

% Using this recycling space, we first develop a fast Rayleigh--Ritz-SVD procedure for extracting approximate right singular vectors and then discuss how the recycled space is updated adaptively.

\subsubsection{Fast Rayleigh--Ritz-SVD procedure and residual computation}\label{sec:rss}
A natural way to extract approximate right singular vectors from $\mV$ is to apply the Rayleigh--Ritz procedure directly to $M(z)$ using the basis $V$. This amounts to computing a thin QR factorization
\begin{align*}
M(z)V = QR,
\end{align*}
followed by the SVD
\begin{align*}
R = J\Xi G^*,
\end{align*}
where $\Xi=\operatorname{diag}(\xi_1,\dots,\xi_{pr})$ contains the Ritz singular values. The corresponding right Ritz singular vectors are then obtained as
\begin{align*}
\tilde{X} = VG_{1{:}r}.
\end{align*}
We shall refer to this procedure for computing the singular value decomposition as the \emph{Rayleigh--Ritz-SVD procedure}. However, the Rayleigh--Ritz-SVD procedure requires $\mO(p^2r^2m + pr^2n)$ flops, which can be expensive even for moderate values of $p$. To reduce the cost, we adopt a two-step strategy.

In the first step, we apply the Rayleigh--Ritz procedure to the shifted matrix $T(z)+\varepsilon I_n$ using $V$, which leads to the projected eigenvalue problem
\begin{align}
H(z)G = G\Theta, \label{step1}
\end{align}
where $H(z) = V^* (T(z)+ \varepsilon I_n) V$ and $\Theta = \operatorname{diag}(\theta_1, \dots, \theta_{pr})$. The shift $\varepsilon$ is introduced to avoid possible breakdown caused by ill conditioning. It follows from the affine linear decomposition \cref{affineN} and the orthonormality of $V$ that
\begin{align}
H(z) = H_1 - zH_2 - z^*H_2^* + |z|^2H_3 + \varepsilon I_{pr}, \label{rayleighN}
\end{align}
where
\begin{align}
H_1 = (MV)^*MV, \qquad H_2 = (MV)^*SV, \qquad H_3 = (SV)^*SV. \label{H}
\end{align}
With precomputed $H_1$, $H_2$, and $H_3$, $H(z)$ can be assembled at very little additional cost for each $z$. Although this significantly reduces the computational cost, the matrix $G$ obtained by solving \cref{step1} coincides with that obtained from the Rayleigh--Ritz-SVD procedure applied to $M(z)$ using $V$ only in exact arithmetic. In fact, when \cref{step1} is solved using \textsc{Lapack}, the angle between the computed eigenvector $\hat{G}_{j{:}j}$ and the exact eigenvector $G_{j{:}j}$ satisfies
\begin{align}
\vartheta(G_{j{:}j}, \hat{G}_{j{:}j}) \leq \frac{c(pr)\theta_{pr}\ema}{\min_{k \neq j}|\theta_k - \theta_j|}, \label{errg}
\end{align}
where $c(pr)$ depends only on the value of $pr$ \cite{and}. Since $\mV$ often captures most components of the right singular vectors associated with the smallest $r$ singular values of $M(z)$, we expect $\theta_j \approx \sigma_j^2 + \varepsilon$ for $j = 1, \ldots, r$. The bound \cref{errg} suggests that the corresponding Ritz singular vectors $V\hat{G}_{j{:}j}$ may be highly inaccurate when $\sigma_1,\ldots,\sigma_r$ are very small, as is often the case since $M(z)$ is usually ill-conditioned. In addition to \cref{errg}, the angle between the computed subspace spanned by $\hat{G}_{1{:}j}$ and the exact subspace spanned by $G_{1{:}j}$, as well as the error in the computed eigenvalues, can also be bounded by
\begin{align}
\vartheta(G_{1{:}j}, \hat{G}_{1{:}j}) \leq \frac{c(pr)\theta_{pr}}{\theta_{j+1}-\theta_{j}}\ema, \qquad
|\theta_j - \hat{\theta}_j| \leq c(pr)\theta_{pr}\ema. \label{errG}
\end{align}
Motivated by these bounds, for a given tolerance $\tau$, we choose the smallest $\tilde{r} \ge r$ such that
\begin{align}
|\hat{\theta}_{\tilde{r}} - \hat{\theta}_{\tilde{r}+1}| > \frac{\hat{\theta}_{pr}\ema}{\tau}. \label{tau}
\end{align}
Thus, by \cref{errG}, $\hat{G}_{1{:}\tilde{r}}$ spans the desired subspace with accuracy $\mO(\tau)$. Consequently, the exact Ritz singular vectors $VG_{1{:}r}$ can be approximated by vectors lying in $\operatorname{span}\{V\hat{G}_{1{:}\tilde{r}}\}$ up to $\mO(\tau)$. Since $\mV$ only contains the right singular vectors associated with the smallest $r$ singular values, $\theta_j$ rarely approximates $\sigma_j^2 + \varepsilon$ well for $j>r$, and $\{\hat{\theta}_j\}_{j > r}$ are often well separated, even when $M(z)$ is ill-conditioned. As a result, \cref{tau} is usually satisfied for a relatively small $\tilde{r}$.

In the second step, we apply the Rayleigh--Ritz-SVD procedure to $M(z)$ using $V\hat{G}_{1{:}\tilde{r}}$. This requires a thin QR factorization followed by an SVD, i.e.,
\begin{align}
M(z)V\hat{G}_{1{:}\tilde{r}}
=
(W_1 - zW_2)\hat{G}_{1{:}\tilde{r}}
=
\tilde{Q}\tilde{R},
\qquad
\tilde{R}
=
\tilde{J}\tilde{\Xi}\tilde{G}^*,
\label{step2}
\end{align}
where $\tilde{\Xi} = \operatorname{diag}(\tilde{\xi}_1, \ldots, \tilde{\xi}_{\tilde{r}})$ contains the Ritz singular values, and $W_1$ and $W_2$ are given by
\begin{align}
W_1 = MV, \qquad W_2 = SV. \label{W12}
\end{align}
The Ritz singular vectors are then given by
\begin{align*}
\tilde{X} = V\hat{G}_{1{:}\tilde{r}} \tilde{G}_{1{:}r}.
\end{align*}
Furthermore, the residuals in the computed Ritz singular pairs $(\tilde{\Xi}, \tilde{X})$ can be evaluated efficiently using \cref{affineN}:
\begin{align}
\tilde{W}
=
T(z)\tilde{X}
-
\tilde{X}\tilde{\Xi}^2
=
(W_3 - zW_4 - z^*W_5 + |z|^2W_6)\hat{G}_{1{:}\tilde{r}} \tilde{G}_{1{:}r}
-
\tilde{X}\tilde{\Xi}^2,
\label{fastresidual}
\end{align}
where
\begin{align}
W_3 = M^*MV, \qquad
W_4 = M^*SV, \qquad
W_5 = S^*MV, \qquad
W_6 = S^*SV. \label{W3456}
\end{align}
Again, neither $M(z)V\hat{G}_{1{:}\tilde{r}}$ nor $T(z)\tilde{X}$ is formed explicitly, as doing so would be prohibitively expensive when $M$ is dense. Instead, $W_1,\dots,W_6$ are precomputed, and $M(z)V\hat{G}_{1{:}\tilde{r}}$ and $T(z)\tilde{X}$ are evaluated according to \cref{step2} and \cref{fastresidual}, respectively.

We summarize the fast Rayleigh--Ritz-SVD procedure and residual computation in \cref{alg:RecycleSingularSpace}, where the hat notation is omitted for simplicity and the computational complexities of the asymptotically dominant operations are listed. Here, \texttt{qr} denotes a function for computing a thin QR factorization, implemented, for instance, via Householder reflections or the more efficient Cholesky QR method \cite{fuk}, while \texttt{svd} denotes a function for computing an SVD. The overall computational cost is reduced to $\mO((pr\tilde{r} + \tilde{r}^2)m + 2pr^2n)$, scaling linearly with $p$.

\begin{algorithm}[t!]
\renewcommand{\algorithmicrequire}{\textbf{Input:}}
\renewcommand{\algorithmicensure}{\textbf{Output:}}
\caption{Fast recycling of the singular subspaces.}\label{alg:RecycleSingularSpace}
\begin{spacing}{1.05}
\begin{algorithmic}[1]
\Require $M, S \in \mathbb{C}^{m \times n}$, $z \in \mathbb{C}$, basis $V \in \mathbb{C}^{n \times pr}$, precomputed $H_1, H_2, H_3$ defined in \cref{H} and $W_1, \dots, W_6$ defined in \cref{W12} and \cref{W3456}, shift $\varepsilon$ and tolerance $\tau$.
\Ensure Approximate Ritz singular pairs $(\tilde{\Xi}, \tilde{X})$ and residual $\tilde{W}$.
\AlgoOutputRule
\Function{$[\tilde{X}, \tilde{\Xi}, \tilde{W}] = \tt{recySing}$}{$M, S, z, V, \{H_q\}_{q=1}^3, \{W_q\}_{q=1}^6, \varepsilon, \tau$}
\State $[G, \Theta] = {\tt{eig}}(H_1 - zH_2 - z^*H_2^* + |z|^2H_3 + \varepsilon I_{pr})$
\State $\tilde{r} = \min\{j : r \le j < pr,\ |\theta_j-\theta_{j+1}| > \theta_{pr}\ema/\tau\}$
\State $[\sim, \tilde{R}] = {\tt{qr}}((W_1 - zW_2)G_{1{:}\tilde{r}})$ \Comment{$\mathcal{O}((pr\tilde{r}+\tilde{r}^2)m)$}
\State $[\sim, \tilde{\Xi}, \tilde{G}] = {\tt{svd}}(\tilde{R})$, $\tilde{\Xi} \gets \tilde{\Xi}_{1{:}r,1{:}r}$
\State $\tilde{X} = VG_{1{:}\tilde{r}}\tilde{G}_{1{:}r}$ \Comment{$\mathcal{O}(pr^2n)$}
\State $\tilde{W} = (W_3 - zW_4 - z^*W_5 + |z|^2W_6)G_{1{:}\tilde{r}}\tilde{G}_{1{:}r} - \tilde{X}\tilde{\Xi}^2$ \Comment{$\mathcal{O}(pr^2n)$}
\State Return $\tilde{X}$, $\tilde{\Xi}$, and $\tilde{W}$
\EndFunction
\end{algorithmic}
\end{spacing}
\end{algorithm}

% Note that orthogonalizing $\hat{S}$ and constructing $\hat{H}_1$ and $\hat{H}_2$ are not included in \cref{alg:RecycleSingularSpace}. These entail a cost of $\mO(b^2r^2(m+n))$ flops. However, this cost is a one-off overhead that can be amortized over a multitude of distinct points. We defer the discussion about this to \cref{sec:upSingSp}.

% In the most extreme case, namely, after orthogonalizing $\hat{S}$ and forming $\hat{H}_1, \hat{H}_2, \hat{H}_3$, we immediately compute a new set of $b$ points and rebuild $\hat{S}$ the averaged cost is $\mO(br^2(m+n))$, which is also proportional to $b$.

% The complexity is reduced to $\mO (\widetilde{r}^2m+(b+1)r\widetilde{r}n)$, which is now proportional to $b$.

\subsubsection{Adaptive update of the recycling subspace}\label{sec:upSingSp}
To continually benefit from recycling without incurring excessive computational overhead, we must update the recycling set $\mS$ and the recycling subspace $\mV$, allowing the number $p$ to be adjusted adaptively.

Suppose that we are currently at a grid point $z$ that has not been visited previously. We first apply \cref{alg:RecycleSingularSpace} to obtain approximate Ritz singular values, Ritz singular vectors, and their residuals. We then test whether the residual satisfies the backward-stability criterion
\begin{align}
\lVert \tilde{W}_{1{:}1} \rVert \leq \max \{\eta\tilde{\xi}_1 \lVert M(z) \rVert,\ \lVert M(z) \rVert^2\ema\}, \label{tolC}
\end{align}
where $\eta$ is the backward-stability tolerance \cite{bag}. If this criterion is not satisfied, we employ LOBPCG-SVD to refine the Ritz singular vectors to the desired accuracy. Once the Ritz singular vectors are refined, the point $z$ is added to the recycling set $\mS$, and the corresponding right singular subspace $\operatorname{span}\{X\}$ is appended to the recycling subspace $\mV$. In other words, the recycling points are precisely the grid points at which the refinement of Ritz singular vectors is carried out.

Although the complexity of \cref{alg:RecycleSingularSpace} depends only linearly on the number $p$ of recycling points, we still cannot afford to continually add new singular subspaces to $\mV$. Therefore, before adding $X$ to $V$, we determine whether any columns of the current $V$ are no longer sufficiently relevant and can be discarded. To this end, we adaptively prune the recycling subspace $\mV$ by first identifying the largest $j \leq p$ such that the inequality
\begin{align}
\lVert X_{1{:}1} - \mP_j X_{1{:}1} \rVert \leq \omega \lVert X_{1{:}1} - \mP_0 X_{1{:}1}\rVert, \label{omegaCond}
\end{align}
is satisfied for a given threshold $\omega \in [1, \infty)$. Here, $\mP_j$ denotes the orthogonal projector onto $\operatorname{span}\{X^{j+1}, \dots, X^{p}\}$. The norms on the left- and right-hand sides of \cref{omegaCond} represent the projection errors of $X_{1{:}1}$ onto $\operatorname{span}\{X^{j+1}, \dots, X^p\}$ and $\operatorname{span}\{X^{1}, \dots, X^p\}$, respectively. If \cref{omegaCond} holds for a moderate value of $\omega$ and a given $j$, it suggests that discarding $\operatorname{span}\{X^{1}, \dots, X^j\}$ causes little loss in the ability of the reduced recycling subspace to represent $X_{1{:}1}$. This is precisely the strategy we employ to keep the size of $\mV$ manageable.

In the extreme case $\omega = 1$, \cref{omegaCond} is rarely satisfied and, consequently, we keep adding new singular subspaces without removing any obsolete information. This results in an ever-growing recycling subspace and a high per-iteration cost for \cref{alg:RecycleSingularSpace}. At the other extreme, when $\omega = \infty$, the recycling subspace $\mV$ contains only the right singular vectors from the most recently visited grid point, which often leads to poor-quality approximate Ritz singular vectors. In practice, we find that $\omega \in [1.2, 2.0]$ usually provides satisfactory performance.

To check whether \cref{omegaCond} holds for a given $j$, we compute the thin block QL factorization of $[X^{1}\ \dots\ X^{p}]$ as
\begin{align}
[X^{1}\ \dots\ X^{p}] = V L, \label{QLfact}
\end{align}
where $V = [V^{1}\ \dots\ V^{p}]$ has orthogonal columns and $L \in \mathbb{C}^{pr \times pr}$ is a lower block-triangular matrix with each block of size $r \times r$. Thus, the projection error can be computed as
\begin{align*}
X_{1{:}1} -\mP_j X_{1{:}1} = X_{1{:}1} - V_j V_j^* X_{1{:}1},
\end{align*}
where $V_j = [V^{j+1} \ \dots \ V^{p}]$. Suppose that $l$ is the largest such $j$ for which \cref{omegaCond} holds. We then remove $s_1,\ldots,s_l$ from $\mS$ and discard the subspace $\operatorname{span}\{X^{1},\dots,X^{l}\}$ by removing the first $l$ block columns of $V$.

To add $\operatorname{span}\{X\}$ to $\mV$, we apply a block Gram--Schmidt-type orthogonalization \cite{car3} of $X$ against $V^{l+1},\ldots,V^p$, obtaining
\begin{align*}
V^{p+1}U_{p+1,p+1} = X - \sum_{j=l+1}^{p} V^j U_{j,p+1},
\end{align*}
where $U_{j,p+1}=(V^j)^*X\in\mathbb{C}^{r\times r}$ ($j = l+1, \dots, p$) are the projection coefficients, and $V^{p+1}\in\mathbb{C}^{n\times r}$ and $U_{p+1,p+1}\in\mathbb{C}^{r\times r}$ are obtained from the thin QR factorization of the right-hand side. Let $L_{j,k}$ denote the $(j,k)$th block of the pre-updated $L$. After the new singular subspace $X$ is appended, the QL factorization \cref{QLfact} becomes
\begin{align}
[X^{l+1}\ \dots\ X^p\ X] =
\underbrace{
[V_l\ V^{p+1}]
}_{\mathclap{\textstyle\text{new $V$}}}
\underbrace{
\begin{bmatrix}
L_{l+1,l+1} &                   &                 & U_{l+1, p+1} \\
\vdots            & \ddots            &                 & \vdots \\[0.3ex]
L_{p,l+1}   & \dots             & L_{p,p}   & U_{p, p+1} \\
0                 & \cdots            & 0               & U_{p+1,p+1}
\end{bmatrix}
}_{\mathclap{\textstyle\text{new $L$}}}.
\label{blkGS}
\end{align}

For notational simplicity, we continue to denote the new factors by $V$ and $L$, although the updated $L$ is no longer block lower triangular. The remaining task is to eliminate the blocks $U_{l+1,p+1},\dots,U_{p,p+1}$
in the last column to restore the block triangular structure. This is accomplished by a sequence of block unitary transformations. For $j=l+1,\dots,p$, we form the stacked block matrix
\begin{align*}
Y_j =
\begin{bmatrix}
U_{p+1,p+1}\\[0.5ex]
U_{j,p+1}
\end{bmatrix}
\in\mathbb{C}^{2r\times r}
\end{align*}
and compute its QR factorization $Y_j = Q_j^{\mathrm{loc}} R_j^{\mathrm{loc}}$, where $Q_j^{\mathrm{loc}}\in\mathbb{C}^{2r\times2r}$ is unitary. Thus, it follows that
\begin{align*}
(Q_j^{\mathrm{loc}})^*
\begin{bmatrix}
U_{p+1,p+1}\\[0.5ex]
U_{j,p+1}
\end{bmatrix}
=
\begin{bmatrix}
\sim\\[0.5ex]
0
\end{bmatrix}.
\end{align*}
We then embed $Q_j^{\mathrm{loc}}$ into an identity matrix to form a block unitary matrix $Q_j\in\mathbb{C}^{(p-l+1)r\times(p-l+1)r}$ so that pre-multiplying $L$ by $Q_j^*$ is equivalent to applying $(Q_j^{\mathrm{loc}})^*$ to the $(j-l)$th and $(p-l+1)$th block rows of $L$. Thus, updating $V$ and $L$ via $V \gets VQ_j$ and $L \gets Q_j^*L$ annihilates $U_{j,p+1}$ while preserving the orthogonality of $V$. After $p-l$ such transformations, $L$ becomes block lower triangular again, and the thin QL factorization is restored. The total cost of this update is $\mathcal{O}((p-l)r^2n)$.

The auxiliary matrices $H_1$, $H_2$, $H_3$ and $W_1, \dots, W_6$ are updated in a similar manner. The leading blocks of these matrices are simply discarded when the subspace is truncated. Once $V^{p+1}$ is available, the associated blocks $MV^{p+1}$ and $SV^{p+1}$ are computed to update $W_1$ and $W_2$, while $W_3, \dots, W_6$ are updated using $M^*MV^{p+1}$, $M^*SV^{p+1}$, $S^*MV^{p+1}$, and $S^*SV^{p+1}$. We update $H_1$, $H_2$, and $H_3$ by computing $W_3^*V^{p+1}$, $W_5^*V^{p+1}$, and $W_6^*V^{p+1}$. Finally, consistency with the rotated basis is maintained by applying the unitary transformations $Q_j$ through the updates
\begin{align*}
H_q &\gets Q_j^* H_q Q_j, \quad q=1,2,3, \\
W_q &\gets W_q Q_j, \quad q=1,\ldots,6.
\end{align*}

% In fact, one could also use a sequence of Givens rotations to eliminate the blocks $R_{l+1,p+1},\dots,R_{p,p+1}$, with the same asymptotic complexity. However, the block unitary rotations can be implemented using Level-3 BLAS operations, whereas classical Givens updates in LAPACK are based on Level-1/2 BLAS operations, which makes the block unitary transformations more efficient in practice.

We summarize the overall procedure for updating the recycling subspace in \cref{alg:UpdateRecycle}, where an additional parameter $\breve{r} \le r$ is introduced to allow the generalization of $X_{1{:}1}$ in \cref{omegaCond} to $X_{1{:}\breve{r}}$ for use in \cref{sec:recyprojection}.

\begin{algorithm}[t!]
\renewcommand{\algorithmicrequire}{\textbf{Input:}}
\renewcommand{\algorithmicensure}{\textbf{Output:}}
\caption{Adaptive update of recycling subspace}\label{alg:UpdateRecycle}
\begin{spacing}{1.05}
\begin{algorithmic}[1]
\Require $M, S \in \mathbb{C}^{m \times n}$, new singular subspace basis $X \in \mathbb{C}^{n \times r}$, current recycling basis $V \in \mathbb{C}^{n \times pr}$ and block lower triangular factor $L \in \mathbb{C}^{pr \times pr}$, an integer $\breve{r} \leq r$, scaling parameter $\omega$, auxiliary matrices $\{H_q\}_{q=1}^3$ and $\{W_q\}_{q=1}^6$.
\Ensure Updated $V$, $L$, $\{H_q\}_{q=1}^3$, and $\{W_q\}_{q=1}^6$ in place.
\AlgoOutputRule
\Function{\tt{updateRecy}}{$M, S, X, V, L, \{H_q\}_{q=1}^3, \{W_q\}_{q=1}^6, \omega, \breve{r}$}
\State $\rho_0 = \lVert (I - VV^*) X_{1{:}\breve{r}} \rVert$
\For{$j = 1, \dots, p$}
    \State $V_j = V_{jr+1{:}pr}$
    \State Compute $\rho_j = \lVert (I - V_j V_j^*) X_{1{:}\breve{r}} \rVert$
    \Comment{$\mathcal{O}((p-j)\breve{r}rn)$}
    \State \textbf{if} $\rho_j/\rho_0 > \omega$ \textbf{break}
\EndFor
\State $l = j - 1$
\State Orthogonalize $X$ against $V_l$ and update $V$, $L$ in \cref{blkGS}
\Comment{$\mathcal{O}((p-l+1)r^2n)$}
\State Update $W_1, W_2$ by $MV^{p+1}$, $SV^{p+1}$
\State Update $W_3, \dots, W_6$ by $M^*MV^{p+1}, M^*SV^{p+1}, S^*MV^{p+1}, S^*SV^{p+1}$
\State Update $H_1, H_2, H_3$ by $W_3^*V^{p+1}, W_5^*V^{p+1}, W_6^*V^{p+1}$
\Comment{$\mathcal{O}((p-l)r^2n)$}
\For{$j = l+1, \dots, p$}
    \State Construct block unitary matrix $Q_j$
    \State $V \gets V Q_j$
    \Comment{$\mathcal{O}(r^2n)$}
    \State $L \gets Q_j^* L$
    \State $H_q \gets Q_j^* H_q Q_j, \quad q=1,2,3$
    \State $W_q \gets W_q Q_j, \quad q=1,\ldots,6$
    \Comment{$\mathcal{O}(r^2m)$}
\EndFor
\State $p \gets p - l + 1$
\EndFunction
\end{algorithmic}
\end{spacing}
\end{algorithm}

\subsection{Grid traversal and full triangularization-based algorithm}\label{sec:traversal}
Before presenting the full algorithm for computing pseudospectra, we first discuss the traversal of the grid points. The most commonly used grid in pseudospectra computation is a Cartesian grid over a rectangular region of the complex plane. For such a grid, one of the simplest traversal strategies is a row-wise zigzag: the first row is traversed from left to right, the second from right to left, and so on. With this ordering, only the singular subspaces associated with the few most recently visited points on the same horizontal line are strongly correlated with that of the current point. It is therefore more sensible to traverse the grid in such a way that, for each grid point currently being visited, as many previously visited points as possible lie in its vicinity. To this end, we employ a batch-zigzag strategy. We first partition the grid into column batches, each with batch size $b$:
\begin{align}
\{x + iy,\ x + i(y + \Delta y),\ \dots,\ x + i[y + (b-1)\Delta y]\}. \label{batch}
\end{align}
We then traverse the grid in nested loops---the outer loop performs a row-wise zigzag over the batches, while in the inner loop the $b$ points within each batch are visited vertically. \Cref{fig:column} shows the batch-zigzag ordering with column batches in a $9\times 9$ grid with $b=3$, whereas \cref{fig:row} shows the corresponding version with row batches, which serves the same purpose. For column (row) batches with row-wise (column-wise) zigzag traversal, the parameter $b$ should roughly match the number of grid points in the vertical (horizontal) direction whose right singular subspaces associated with the few smallest singular values exhibit significant correlations. Choosing $b$ either too large or too small may reduce the effectiveness of the recycling process.

\begin{figure}[t!]
\centering
\subfloat[]{\includegraphics[width=0.46\textwidth]{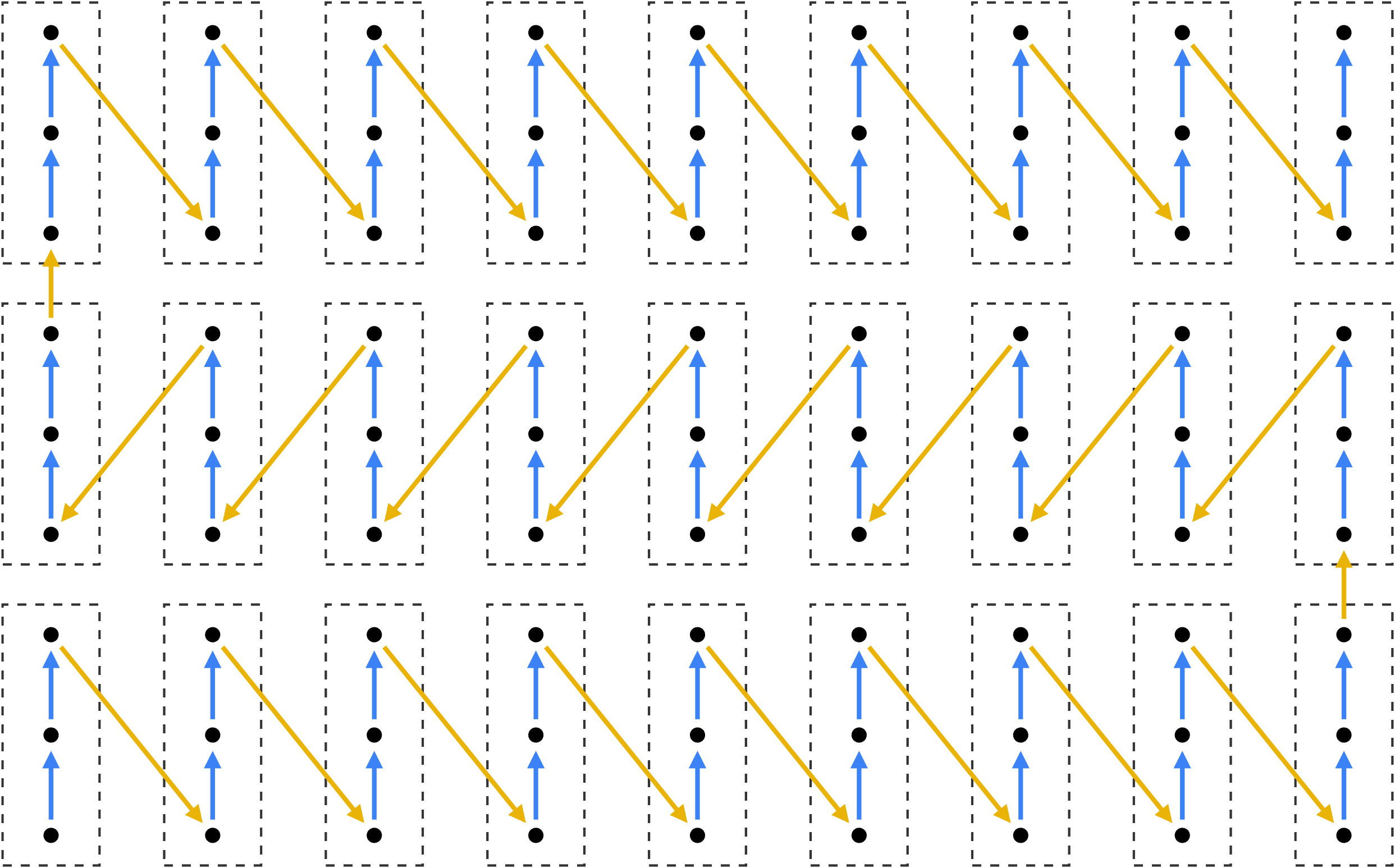}\label{fig:column}}
\hfill
\subfloat[]{\includegraphics[width=0.46\textwidth]{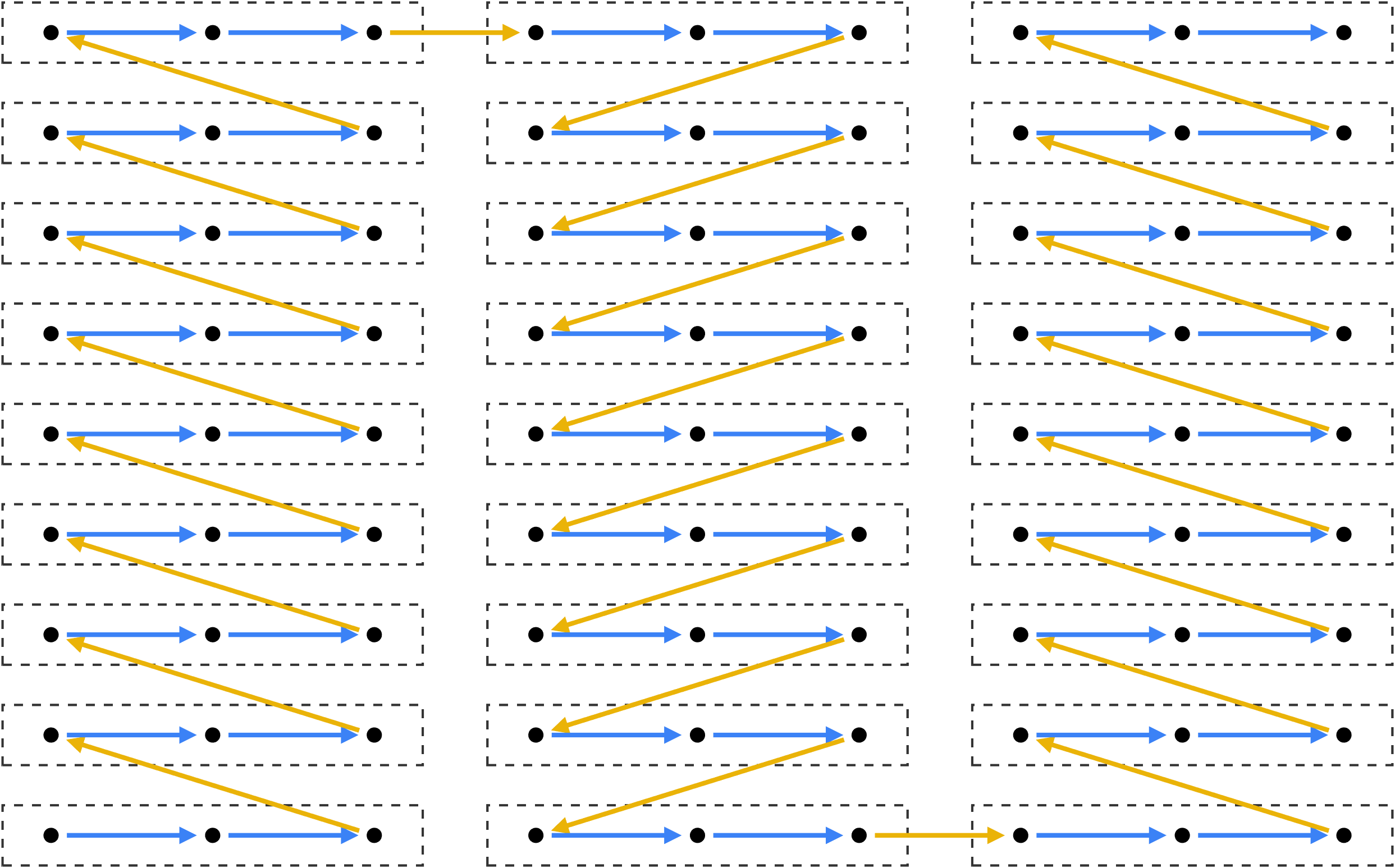}\label{fig:row}}
\caption{Batch-based zigzag traversal on a $9 \times 9$ grid with $b=3$: (a) column batches with row-wise zigzag. (b) row batches with column-wise zigzag.}\label{fig:batch-zigzag-traversal}
\end{figure}

The full algorithm for computing pseudospectra is summarized in \cref{alg:psTriRecy}. The function $[M,S]=$ \texttt{preliminary}$(C)$ performs the appropriate preliminary reduction described at the beginning of \cref{sec:prelimred} and detailed in \cref{sec:prelred} of the supplementary material. In line 12, the LOBPCG-SVD function \texttt{lobpcgSvd} is called. This function takes as input the matrix whose smallest singular pairs are to be computed, the preconditioner, the initial right singular subspace, and the desired number of converged right singular vectors, and returns the corresponding right singular vectors and singular values. See \cref{sec:lobpcgsvd} of the supplementary material for details.

At each grid point, \cref{alg:RecycleSingularSpace} first produces approximate Ritz singular values, Ritz singular vectors, and their residuals. The approximation obtained using recycled information is accepted if \cref{tolC} is satisfied. If not, the recycled Ritz singular vectors are used as the initial singular subspace for LOBPCG-SVD refinement applied to $R(z)^{-*}$. Finally, the smallest singular values of $M(z)$ is obtained as reciprocal of the computed largest singular value of $R(z)^{-*}$. We use $R(z)^{-*}$ rather than $R(z)^{-1}$ because their singular values are identical, while the right singular vectors of $R(z)^{-*}$ coincide with those of $M(z)$.

\begin{algorithm}[t!]
\renewcommand{\algorithmicrequire}{\textbf{Input:}}
\renewcommand{\algorithmicensure}{\textbf{Output:}}
\caption{Computing pseudospectra by singular-subspace recycling and triangularization.}\label{alg:psTriRecy}
\begin{spacing}{1.05}
\begin{algorithmic}[1]
\Require $C \in \mathbb{C}^{m \times n}$, grid $\mG = \{z_\ell\}_{\ell=1}^{|\mG|}$, batch size $b$, shift $\varepsilon$ and tolerances $\tau$, $\omega$, $\eta$.
\Ensure $\{\sigma_{\ell, 1}\}_{\ell=1}^{|\mG|}$, where $\sigma_{\ell, 1}$ is the smallest singular value of $C - z_\ell I$.
\AlgoOutputRule
\Function{$\{\sigma_{\ell, 1}\}_{\ell=1}^{|\mG|} = \tt{psTriRecy}$}{$C, \mG, b, \varepsilon, \tau, \omega, \eta$}
\State $[M, S] = {\tt{preliminary}}(C)$ \Comment{see \cref{sec:prelred}}
\State Reorder $\mG$ by batch zigzag with batch size $b$
\State Initialize $V, L, \{H_q\}_{q=1}^3$ and $\{W_q\}_{q=1}^6$
\For{$\ell = 1, \dots, |\mG|$}
    \State $z = z_\ell$
    \State $[\tilde{X}, \tilde{\Xi}, \tilde{W}] = {\tt{recySing}}(M, S, z, V, \{H_q\}_{q=1}^3, \{W_q\}_{q=1}^6, \varepsilon, \tau)$
    \If{$\cref{tolC}$ holds}
        \State $\sigma_{\ell, 1} = \tilde{\xi}_1$ and \textbf{continue}
    \Else
        \State $[\sim, R(z)] = {\tt{qr}}(M(z))$
        \State $[X, \Sigma] = {\tt{lobpcgSvd}}(R(z)^{-*}, I_{n}, \tilde{X}, 1)$
        \State ${\tt{updateRecy}}(M, S, X, V, L, \{H_q\}_{q=1}^3, \{W_q\}_{q=1}^6, \omega, 1)$
        \State $\sigma_{\ell, 1} = 1/\sigma_1$
    \EndIf
\EndFor
\State Return $\{\sigma_{\ell, 1} \}_{\ell=1}^{|\mG|}$
\EndFunction
\end{algorithmic}
\end{spacing}
\end{algorithm}

\section{Recycling the preconditioners}\label{sec:recyprojection}
Algorithm~\ref{alg:psTriRecy} substantially accelerates the computation of pseudospectra---the recycling approach, together with the fast Rayleigh--Ritz-SVD procedure, provides sufficiently accurate approximations to the desired singular values and singular vectors, thereby avoiding expensive refinement steps that would otherwise be carried out at every grid point in a non-recycling method. However, whenever refinement is necessary, the dense or sparse QR factorizations and triangular solves (lines 11 and 12 in \cref{alg:psTriRecy}) may still be costly, especially for dense matrices and sparse matrices with substantial fill-in.

To address this issue, in the refinement phase we bypass the QR factorization of $M(z)$ altogether by applying LOBPCG-SVD directly to $C(z)$ with a given symmetric positive-definite preconditioner $K \in \mathbb{C}^{n \times n}$ \cite{ben2,sco2,sco1} satisfying
\begin{align*}
K \approx N(z).
\end{align*}
With a sufficiently effective preconditioner $K$, LOBPCG-SVD converges rapidly to the smallest singular values and the corresponding singular vectors of $C(z)$. Since the spectral property of $C(z)$ may vary significantly over $z \in \Omega$, the same preconditioner $K$ cannot, in general, be reused at different grid points. Moreover, constructing $K$ from scratch at every grid point is impractical, as computing an effective preconditioner is often expensive. Hence, our strategy is to employ the two-level preconditioner \cite{tan} built upon $K$ and recycle it whenever possible.

\subsection{The two-level preconditioner}

The convergence of LOBPCG has been studied extensively in \cite{kny,ney,sta} when applied to compute the smallest eigenvalues of a symmetric positive-definite matrix $A$ using a symmetric positive-definite preconditioner $K$. In what follows, we use the effective spectral condition number of the preconditioned matrix $K^{-1}A$ as a measure of the convergence of LOBPCG with block size $r$
\begin{align*}
\kappa_r^{\mathrm{eff}}(K^{-1}A) = \frac{\lambda_{\max}(K^{-1}A)}{\lambda_{r+1}(K^{-1}A)}.
\end{align*}
In our context, applying LOBPCG-SVD to $C(z)$ is equivalent to computing its smallest eigenvalues by applying LOBPCG to $N(z)$. Thus, the relevant effective spectral condition number is
\begin{align}
\kappa_r^{\mathrm{eff}}(K^{-1}N(z)) = \frac{\bar{\lambda}_n}{\bar{\lambda}_{r+1}}, \label{condLOBPCG}
\end{align}
where $\bar{\lambda}_1 \leq \cdots \leq \bar{\lambda}_{r+1} \leq \cdots \leq \bar{\lambda}_n$ denote the eigenvalues of the pencil $(N(z),K)$ in ascending order.

Since $N(z)$ is often very ill-conditioned, a base preconditioner $K$ may bring the largest eigenvalue $\bar{\lambda}_n$ close to one while leaving $\bar{\lambda}_{r+1}$ still very small. To further reduce the effective spectral condition number \cref{condLOBPCG}, we define a two-level preconditioner as follows. Given a base preconditioner $K$ for $N(z)$ and a projection subspace spanned by $Z\in \mathbb{C}^{n \times k}$, the two-level preconditioner $K_{Z}$ is defined through its inverse \cite{tan}
\begin{align}
K_{Z}^{-1} = K^{-1} + Z(Z^*N(z)Z)^{-1}Z^*. \label{twolevel}
\end{align}

The following lemma shows that, with a properly chosen $Z$, the two-level preconditioner improves upon the base preconditioner $K$.

\begin{lemma}\label{twolevelCond}
For $N(z)$ and a symmetric positive-definite preconditioner $K$, suppose that the eigenvalues of the pencil $(N(z), K)$ are $\bar{\lambda}_1, \dots, \bar{\lambda}_n$ in ascending order. Let $\bar{E}$ be the matrix that contains the $\bar{r}$ $K$-orthonormal eigenvectors $\bar{e}_1, \dots, \bar{e}_{\bar r}$ corresponding to $\bar{\lambda}_1, \dots, \bar{\lambda}_{\bar{r}}$. That is,
\begin{align}
N(z)\bar{E} = K\bar{E}\bar{\Lambda}, \quad \text{s.t.} \quad \bar{E}^*K\bar{E} = I_{\bar{r}}, \label{genM}
\end{align}
where $\bar{\Lambda} = \operatorname{diag}(\bar{\lambda}_1, \dots, \bar{\lambda}_{\bar{r}})$ and $\bar{E}= [\bar{e}_1, \dots, \bar{e}_{\bar{r}}]$. If $\bar{\lambda}_{\bar{r}+r+1} < 1$ and $\bar{\lambda}_{\bar{r}} + 1 < \bar{\lambda}_n$, then
\begin{align}
\kappa_{r}^{\mathrm{eff}}(K_{\bar{E}}^{-1}N(z)) = \frac{\bar{\lambda}_n}{\bar{\lambda}_{\bar{r}+r+1}}. \label{cond}
\end{align}
\end{lemma}

\begin{proof}
It follows from \cref{twolevel} that
\begin{align*}
K_{\bar{E}}^{-1}N(z)\bar{e}_j =
\begin{cases}
(\bar{\lambda}_j + 1)\bar{e}_j, & j = 1, \dots, \bar{r} \\
\bar{\lambda}_j\bar{e}_j, & j = \bar{r}+1, \dots, n.
\end{cases}
\end{align*}
Hence the eigenvalues of $K_{\bar{E}}^{-1}N(z)$ are
\begin{align*}
\bar{\lambda}_1 + 1, \dots, \bar{\lambda}_{\bar{r}} + 1, \bar{\lambda}_{\bar{r}+1}, \dots, \bar{\lambda}_n,
\end{align*}
which may no longer be in ascending order. However, since we assume that $\bar{\lambda}_{\bar{r}+r+1} < 1$ and $\bar{\lambda}_{\bar{r}} + 1 < \bar{\lambda}_n$, the largest eigenvalue remains $\bar{\lambda}_n$, and the $(r+1)$th smallest eigenvalue is $\bar{\lambda}_{\bar{r}+r+1}$, from which \cref{cond} follows.
\end{proof}

This lemma shows that the two-level preconditioner $K_{\bar{E}}$ improves upon the base preconditioner $K$ by elevating the smallest $\bar{r}$ eigenvalues of $(N(z), K)$ by $1$, which in turn reduces the effective spectral condition number. Consequently, even a small value of $\bar{r}$ can yield a substantial improvement when only a few eigenvalues lie close to zero. Furthermore, the two-level preconditioner can be recycled across different grid points by keeping $K$ fixed and updating only the projection matrix $Z$.

\subsection{Recycling the projection subspaces, refining, and updating}
The task now reduces to the computation of $\bar{E}$. Rather than computing $\bar{E}$ by applying LOBPCG directly to the generalized eigenvalue problem \cref{genM}, which would require multiplying $K$ by multiple vectors at each iteration, we transform it into the equivalent standard eigenvalue problem
\begin{align}
F^{-*}N(z)F^{-1}\bar{X} = \bar{X}\bar{\Lambda}, \label{barX}
\end{align}
where $\bar{X} = F\bar{E}$ and $F \in \mathbb{C}^{n \times n}$ satisfies $K = F^*F$.

Moreover, to make the correspondence with \cref{genForm,affineN} explicit, we introduce the analogous quantities
\begin{align*}
\bar{M} = C F^{-1}, \quad \bar{S} = I F^{-1}, \quad \bar{M}(z) := \bar{M} - z\bar{S},
\end{align*}
where $\bar{M}$ and $\bar{S}$ are introduced only as formal products; in practice, they are not formed explicitly, and the action of $F^{-1}$ is effected implicitly through linear system solves. With these notations, we can define
\begin{align*}
\bar{T}(z) := (\bar{M} - z\bar{S})^*(\bar{M} - z\bar{S}) = F^{-*}N(z)F^{-1},
\end{align*}
which places the preconditioned eigenproblem \cref{barX} within the framework of \cref{sec:recycsing}. Let $\bar{X}^1, \dots, \bar{X}^{\bar{p}}$ denote the eigenvector matrices associated with previously visited grid points at which \cref{barX} is solved. We refer to
\begin{align}
\bar{\mV} = \operatorname{span}\{\bar{X}^1, \dots, \bar{X}^{\bar{p}}\}
= \operatorname{span}\{\bar{V}\}, \label{recyProj}
\end{align}
as the \emph{recycling subspace for projection}. Whenever no ambiguity arises, we continue to refer to it simply as the recycling subspace. Here, $\bar{V}$ is the column-orthonormal matrix whose columns span $\bar{\mV}$. The corresponding auxiliary matrices $\{\bar{H}_q\}_{q=1}^3$ and $\{\bar{W}_q\}_{q=1}^6$ can then be constructed exactly as in \cref{sec:rss} with $\bar{M}$, $\bar{S}$, and $\bar{V}$ replacing $M$, $S$, and $V$, respectively. With these components available, we extract useful information from the recycling subspace by projecting the eigenvalue problem of $\bar{T}(z)$ onto $\bar{\mV}$, which yields the Rayleigh--Ritz problem
\begin{align}
\bar{H}(z)\check{G} = \check{G} \check{\Theta}, \label{HG}
\end{align}
where $\bar{H}(z) = \bar{V}^* \bar{T}(z) \bar{V}$ and $\check{\Theta} = \operatorname{diag}(\check{\theta}_1, \dots, \check{\theta}_{\bar{r}})$ contains the Ritz singular values. Analogous to \cref{rayleighN}, the projected matrix $\bar{H}(z)$ can be assembled inexpensively. Once \cref{HG} is solved, the Ritz singular vector $\check{X}$ and the projection vector $\check{E}$ are obtained as
\begin{align*}
\check{X} = \bar{V}\check{G}_{1{:}\bar{r}}, \quad
\check{E} = (\bar{W}_2)_{1{:}n, {:}}\check{G}_{1{:}\bar{r}}.
\end{align*}
Similarly, the associated residual $\check{W}$ can be computed explicitly using the auxiliary matrices $\{\bar{W}_q\}_{q=3}^6$ following the same pattern as in \cref{fastresidual}.

The overall procedure for recycling the projection subspace is summarized in \cref{alg:recyProj}.

\begin{algorithm}[t!]
  \renewcommand{\algorithmicrequire}{\textbf{Input:}}
  \renewcommand{\algorithmicensure}{\textbf{Output:}}
  \caption{Recycling projection subspaces.}\label{alg:recyProj}
  \begin{spacing}{1.05}
  \begin{algorithmic}[1]
  \Require $\bar{M}, \bar{S} \in \mathbb{C}^{m \times n}$, $z \in \mathbb{C}$, basis $\bar{V} \in \mathbb{C}^{n \times \bar{p}\bar{r}}$, precomputed $\bar{H}_1, \bar{H}_2, \bar{H}_3 \in \mathbb{C}^{\bar{p}\bar{r} \times \bar{p}\bar{r}}$ and $\bar{W}_1, \dots, \bar{W}_6 \in \mathbb{C}^{n \times \bar{p}\bar{r}}$, and target block size $\bar{r}$.
  \Ensure Ritz singular pairs $(\check{X}, \check{\Theta})$ of $\bar{T}(z)$, projection vectors $\check{E}$, and residual $\check{W}$.
  \AlgoOutputRule
  \Function{$[\check{X}, \check{\Theta}, \check{E}, \check{W}] = \tt{recyProj}$}{$\bar{M}, \bar{S}, z, \bar{V}, \{\bar{H}_q\}_{q=1}^3, \{\bar{W}_q\}_{q=1}^6$}
  \State $[\check{G}, \check{\Theta}] = {\tt{eig}}(\bar{H}_1 - z\bar{H}_2 - \bar{z}\bar{H}_2^* + |z|^2\bar{H}_3)$, \ $\check{\Theta} \gets \check{\Theta}_{1{:}\bar{r},1{:}\bar{r}}$
  \State $\check{X} = \bar{V}\check{G}_{1{:}\bar{r}}$, \ $\check{E} = (\bar{W}_2)_{1{:}n, {:}}\check{G}_{1{:}\bar{r}}$ \Comment{$\mathcal{O}(\bar{p}\bar{r}^2n)$}
  \State $\check{W} = (\bar{W}_3 - z\bar{W}_4 - z^*\bar{W}_5 + |z|^2\bar{W}_6)\check{G}_{1{:}\bar{r}} - \check{X}\check{\Theta}$ \Comment{$\mathcal{O}(\bar{p}\bar{r}^2n)$}
  \State Return $\check{X}$, $\check{\Theta}$, $\check{E}$, and $\check{W}$
  \EndFunction
  \end{algorithmic}
  \end{spacing}
\end{algorithm}

Since the projection subspace is recycled only for preconditioning purposes, we use a more relaxed convergence criterion than that in \cref{tolC} to determine whether the approximate Ritz singular pairs are sufficiently accurate. The backward-stability criterion for eigenvalue problems \cite{due} leads to the convergence criterion
\begin{align}
\lVert \check{W}_{j{:}j} \rVert \leq \eta \lVert \bar{T}(z) \rVert, \quad j = 1,\dots,\bar{r}. \label{tolTz}
\end{align}
% where $\check{r} \leq \bar{r}$ is the largest index satisfying $\check{\theta}_{\check{r}} \leq \tau_{\min}$. Here, $\tau_{\min}$ is a threshold the eigenvalues below which are the ones that we intend to control. In practice, we estimate $\lVert \bar{T}(z) \rVert$ by $(\sqrt{\bar{\lambda}_n^{\dagger}} + \gamma |\delta z|)^2$ using \cref{weyl}.
If \cref{tolTz} is not satisfied, we apply LOBPCG-SVD to $\bar{M}(z)$ using the approximate Ritz singular vectors $\check{X}$ as the initial guess, yielding $\bar{X}$ and $\bar{\Sigma} = \operatorname{diag}(\bar{\sigma}_1, \dots, \bar{\sigma}_{\bar{r}})$ as refinements of $\check{X}$ and $\check{\Theta}$, respectively. The recycling subspace $\bar{\mV}$ and its orthogonal basis $\bar{V}$ are updated in the same way as in \cref{sec:upSingSp}, except that the condition \cref{omegaCond} now becomes
\begin{align*}
\lVert \bar{X}_{1{:}\breve{r}} - \bar{\mP}_l\bar{X}_{1{:}\breve{r}} \rVert
\leq
\omega
\lVert \bar{X}_{1{:}\breve{r}} - \bar{\mP}_0\bar{X}_{1{:}\breve{r}} \rVert,
\end{align*}
where $\bar{\mP}_l$ is the orthogonal projector onto $\operatorname{span}\{\bar{X}^{l+1}, \dots, \bar{X}^{\bar{p}}\}$. Here, $\breve{r} \leq \bar{r}$ is the largest index such that $\bar{\sigma}_{\breve{r}}^2 \leq \tau_{\min}$; see \cref{tauRecycle} for the role and interpretation of $\tau_{\min}$.

\begin{algorithm}[t!]
  \renewcommand{\algorithmicrequire}{\textbf{Input:}}
  \renewcommand{\algorithmicensure}{\textbf{Output:}}
  \caption{Computing pseudospectra by singular-subspace recycling and preconditioning equipped with projection-subspace recycling.}\label{alg:psPrecRecy}
  \begin{spacing}{1.05} 
  \begin{algorithmic}[1]
  \Require $C \in \mathbb{C}^{m \times n}$, grid $\mG = \{z_\ell\}_{\ell=1}^{|\mG|}$, batch size $b$, shift $\varepsilon$ and tolerances $\tau_{\max}, \tau_{\min}$, $\tau$, $\omega$ and $\eta$.
  \Ensure $\{\sigma_{\ell,1}\}_{\ell=1}^{|\mG|}$, where $\sigma_{\ell,1}$ is the smallest singular value of $C - z_\ell I$.
  \AlgoOutputRule
  \Function{$\{\sigma_{\ell,1}\}_{\ell=1}^{|\mG|} = \tt{psPrecRecy}$}{$C, \mG, b, \varepsilon, \tau_{\max}, \tau_{\min}, \tau, \omega, \eta$}
  \State Reorder $\mG$ by batch zigzag with batch size $b$
  \State $M = C,\ S = I$ and initialize $V, L, \{H_q\}_{q=1}^3$ and $\{W_q\}_{q=1}^6$
  \State Set $z^{\dagger} = z_1$, compute $K = F^*F \approx N(z^{\dagger})$, $\gamma = \lVert F^{-1} \rVert_2$, and $\bar{\lambda}^{\dagger}_n$
  \State Set $\bar{M} = CF^{-1}$ and $\bar{S} = IF^{-1}$, and initialize $\bar{V}, \bar{L}, \{\bar{H}_q\}_{q=1}^3$, and $\{\bar{W}_q\}_{q=1}^6$
  \For{$\ell = 1, \dots, |\mG|$}
    \State $z = z_\ell$
    \State $[\tilde{X}, \tilde{\Xi}, \tilde{W}] = {\tt{recySing}}(M, S, z, V, \{H_q\}_{q=1}^3, \{W_q\}_{q=1}^6, \varepsilon, \tau)$
    \If{$\tilde{W}$ satisfies \cref{tolC}}
      \State $\sigma_{\ell,1} = \tilde{\xi}_1$ and \textbf{continue}
    \EndIf
    \State $[\check{X}, \check{\Theta}, \check{E}, \check{W}] = {\tt{recyProj}}(\bar{M}, \bar{S}, z, \bar{V}, \{\bar{H}_q\}_{q=1}^3, \{\bar{W}_q\}_{q=1}^6)$
    \If{$\check{W}$ satisfies \cref{tolTz}}
      \State $\bar{E} = \check{E}$, $\bar{\lambda}_{\bar{r}} = \check{\theta}_{\bar{r}}$
    \Else
      \State $[\bar{X}, \bar{\Sigma}] = {\tt{lobpcgSvd}}(\bar{M}(z), I_n, \check{X}, \bar{r})$, $\bar{E} = F^{-1}\bar{X}$, $\bar{\lambda}_{\bar{r}} = \bar{\sigma}_{\bar{r}}^2$
      \State $\breve{r} = \max\{j : 1 \leq j \leq \bar{r},\ \bar{\sigma}_j^2 \leq \tau_{\min}\}$
      \State ${\tt{updateRecy}}(\bar{M}, \bar{S}, \bar{X}, \bar{V}, \bar{L}, \{\bar{H}_q\}_{q=1}^3, \{\bar{W}_q\}_{q=1}^6, \omega, \breve{r})$
    \EndIf
    \State $ \bar{\lambda}_{n} = (\sqrt{\bar{\lambda}_n^{\dagger}} + \gamma |z - z^{\dagger}|)^2$
    \If{$\bar{\lambda}_{n} > \tau_{\max}$ or $\bar{\lambda}_{\bar{r}} < \tau_{\min}$}
      \State Set $z^{\dagger} = z$, compute $K = F^*F \approx N(z^{\dagger})$, $\gamma = \lVert F^{-1} \rVert_2$, and $\bar{\lambda}^{\dagger}_n$
      \State Set $\bar{M} = CF^{-1}$ and $\bar{S} = IF^{-1}$
      \State Initialize $\bar{V}, \bar{L}, \{\bar{H}_q\}_{q=1}^3$, and $\{\bar{W}_q\}_{q=1}^6$
      \State $[\bar{X}, \bar{\Sigma}] = {\tt{lobpcgSvd}}(\bar{M}(z), I_n, \check{X}, \bar{r})$, $\bar{E} = F^{-1}\bar{X}$
      \State $\breve{r} = \max\{j : 1 \leq j \leq \bar{r},\ \bar{\sigma}_j^2 \leq \tau_{\min}\}$
      \State ${\tt{updateRecy}}(\bar{M}, \bar{S}, \bar{X}, \bar{V}, \bar{L}, \{\bar{H}_q\}_{q=1}^3, \{\bar{W}_q\}_{q=1}^6, \omega, \breve{r})$
    \EndIf
    \State Construct $K_{\bar{E}}$ according to \cref{twolevel}
    \State $[X, \Sigma] = {\tt{lobpcgSvd}}(M(z), K_{\bar{E}}, \tilde{X}, 1)$
    \State $\sigma_{\ell,1} = \sigma_1$
    \State ${\tt{updateRecy}}(M, S, X, V, L, \{H_q\}_{q=1}^3, \{W_q\}_{q=1}^6, \omega, 1)$
  \EndFor
  \State Return $\{\sigma_{\ell,1}\}_{\ell=1}^{|\mG|}$
  \EndFunction
  \end{algorithmic}
  \end{spacing}
\end{algorithm}

\subsection{Reconstruction of the base preconditioner}
A base preconditioner may not be sufficiently effective over the entire grid, and therefore it must be updated whenever necessary. Suppose that the base preconditioner $K$ is constructed at a previously visited grid point $z^{\dagger}$. At a new grid point $z = z^{\dagger} + \Delta z$, let $\bar{\lambda}_{\bar{r}}$ denote the computed approximation to the $\bar{r}$th smallest eigenvalue of $\bar{T}(z)$. Since $\bar{\lambda}_{\bar{r}} \leq \bar{\lambda}_{\bar{r}+r+1}$, this quantity provides a computable lower bound for the denominator in \cref{cond}. On the other hand, the perturbation relation
\[
\bar{M}(z)=\bar{M}(z^{\dagger})-\Delta z\,\bar{S}
\]
together with Weyl's inequality \cite{ste} implies that
\begin{align*}
\bar{\lambda}_n \leq \big(\sqrt{\bar{\lambda}_n^{\dagger}} + \gamma |\Delta z|\big)^2,
\end{align*}
where $\gamma = \lVert F^{-1} \rVert_2$, and $\bar{\lambda}_n^{\dagger}$ is the largest eigenvalue of $(N(z^{\dagger}), K)$. By \cref{cond}, the effective condition number of the two-level preconditioner satisfies
\begin{align*}
\kappa_{r}^{\mathrm{eff}}(K_{\bar{E}}^{-1}N(z)) = \frac{\bar{\lambda}_n}{\bar{\lambda}_{\bar{r}+r+1}} \leq \frac{\big(\sqrt{\bar{\lambda}_n^{\dagger}} + \gamma |\Delta z|\big)^2}{\bar{\lambda}_{\bar{r}}}.
\end{align*}
Significant growth of $\kappa_{r}^{\mathrm{eff}}(K_{\bar{E}}^{-1}N(z))$ suggests the ineffectiveness of the current base preconditioner. We thus require
\begin{align}
\big(\sqrt{\bar{\lambda}_n^{\dagger}} + \gamma |\Delta z|\big)^2 \leq \tau_{\max}, \qquad \bar{\lambda}_{\bar{r}} \geq \tau_{\min} \label{tauRecycle}
\end{align}
to keep the effective spectral condition number of
$K_{\bar E}^{-1}N(z)$ from varying excessively. If \cref{tauRecycle} holds, we leave $K$ unchanged; otherwise, $K$ is recomputed at the new grid point $z$.

We are now able to assemble the full algorithm, listed in \cref{alg:psPrecRecy}, for computing the pseudospectra of a given matrix by recycling both the singular and projection subspaces.
  
\section{Experiments}\label{sec:exp}
In this section, we evaluate the performance of the proposed algorithms. All numerical experiments are implemented in \textsc{Julia} v1.12 and executed using a single thread on a desktop equipped with a 2.90 GHz Intel Core i5 processor and 16 GB of RAM.

All the experiments use the same parameter settings. We take the block sizes in \texttt{lobpcgSvd} to be $r=6$ and $\bar{r}=40$ for the singular and projection subspace refinements, respectively, while the tolerance $\eta$ is set to $10^{-4}$. For \texttt{recySing}, we use the shift $\varepsilon = 10^{-14}\lVert C \rVert^2$ and the tolerance $\tau = 10^{-10}$. For \texttt{updateRecy}, the threshold is set to $\omega = 1.6$, and the maximum number of recycling points is capped at $p_{\max}=60$. We construct the base preconditioner using the incomplete Cholesky factorization implemented in HSL\_MI28 \cite{sco2}. For projection subspace recycling, we let $\tau_{\min} = 10^{-3}$ and $\tau_{\max} = 30$. For the grid $\mG$, we always use an equispaced Cartesian grid with the same number $n_{\mathrm{grid}}$ of points in both directions.

We benchmark the proposed algorithms against the Lanczos-based method, which serves as the core algorithm of \textsc{EigTool} \cite{wri3,tre2}. Since our experiments involve only square matrices, the Lanczos-based method first upper-triangularizes the given matrix via Schur factorization and then applies the Lanczos iteration to the inverse of the triangular factor. In our implementation, the Lanczos iteration is effected by \textsc{KrylovKit}'s function \texttt{svdsolve} \cite{hae}, which is based on a thick-restarted partial Lanczos bidiagonalization. In the rest of this paper, we refer to our \textsc{Julia} implementation of the Lanczos-based method as \texttt{psInvLanc}, which is available in \cite{den1}.

\begin{table}[t]
  \caption{Test matrices.}
  \label{tab:matrices}
  \centering
  \resizebox{\textwidth}{!}{%
  \renewcommand{\arraystretch}{1.1}%
  \begin{tabular}{ccccc}
  \hline
  matrix & type & size & region & cond$_{\max}$  \\
  \hline
  landau\_$40\pi$& dense  & $5{,}000 \times 5{,}000$   & $[-1.1,1.2]\times[-1.1,1.1]$   & $1.32\times 10^{15}$ \\
  basor          & dense  & $2{,}000 \times 2{,}000$   & $[-4.0,6.5]\times[-5.0,2.0]$   & $7.76\times 10^8$ \\
  af23560        & sparse & $23{,}560 \times 23{,}560$ & $[-1.6,0.8]\times[-1.55,1.55]$ & $1.64\times 10^{11}$ \\
  skewlap3d      & sparse & $13{,}824 \times 13{,}824$ & $[-550,100]\times[-200,200]$   & $8.35\times 10^{13}$ \\
  sparse random  & sparse & $4{,}000 \times 4{,}000$   & $[0.45,0.55]\times[0.0,0.1]$   & $2.19\times 10^{6}$ \\
  \hline
  \end{tabular}%
  }
\end{table}

\begin{table}[t]
  \caption{Execution times and operation counts.}
  \label{tab:costdense}
  \centering
  \small
  \resizebox{\textwidth}{!}{%
  \renewcommand{\arraystretch}{1.1}%
  \begin{tabular}{ccc|ccc|ccc}
  \hline
  \multirow{2}{*}{matrix} & \multicolumn{2}{c|}{\tt{psInvLanc}}
  & \multicolumn{3}{c|}{\tt{psTriRecy}}
  & \multicolumn{3}{c}{\tt{psPrecRecy}} \\
  \cline{2-9}
  & time (s) & sv$_R$ & time (s) & sv$_R$ & mv$_M$ & time (s) & sv$_F$ & mv$_M$\\
  \hline
  landau\_$40\pi$ & $9{,}069$  & $821{,}658$  & $2{,}006$  & $63{,}276$  & $61{,}740$   &  &  & \\
  basor           & $7{,}251$   & $3{,}569{,}406$ & $1{,}352$  & $492{,}132$ & $72{,}990$   &  &  & \\
  af23560         & $129{,}936$ & $642{,}686$ & $8{,}150$ & $30{,}120$ & $27{,}540$ & $7{,}333$ & $499{,}176$ & $564{,}456$ \\
  skewlap3d       & $136{,}575$ & $336{,}134$ & $7{,}390$ & $26{,}016$ & $25{,}956$ & $3{,}968$ & $272{,}544$ & $381{,}594$ \\
  sparserandom   & $92{,}846$  & $465{,}586$ & $7{,}848$ & $60{,}348$ & $60{,}336$ & $1{,}606$ & $701{,}160$ & $924{,}780$ \\
  \hline
  \end{tabular}%
  }
\end{table}

Table~\ref{tab:matrices} lists the five test matrices used in our experiments, including two dense and three sparse matrices, along with their sizes and computational regions $\Omega$. Here, cond$_{\max}$ denotes the maximum condition number of $C(z)$ over the entire grid.

Table~\ref{tab:costdense} reports the overall computational cost of the three methods for $n_{\mathrm{grid}}=200$ and $b=20$ in terms of the total execution time and the numbers of three key matrix--vector operations:

\begin{itemize}[leftmargin=1.5em]
\item sv$_R$: the total number of triangular solves involving the triangular factor $R(z)$ in the Lanczos-based method or in \texttt{lobpcgSvd}.

\item sv$_F$: the total number of triangular solves involving the triangular factors $F$ and $F^*$ that occur in \texttt{psPrecRecy}. These solves arise in \texttt{updateRecy} for projection subspace recycling and in \texttt{lobpcgSvd} for both singular and projection subspace refinements.

\item mv$_M$: the total number of multiplications of $M$ by a vector. For \texttt{psTriRecy}, this operation is performed in \texttt{updateRecy}. For \texttt{psPrecRecy}, where $M=C$, this operation occurs in \texttt{updateRecy} for both singular and projection subspace recycling, as well as in \texttt{lobpcgSvd} for singular and projection subspace refinements.
\end{itemize}

For many of the matrix--vector operations that are effected through matrix--matrix operations, we count each such operation as multiple matrix--vector operations. For example, a matrix--matrix operation involving a matrix with $r$ or $\bar{r}$ columns is counted as $r$ or $\bar{r}$ matrix--vector operations. In addition, we omit the results of \texttt{psPrecRecy} for the two dense examples because $M(z)$ is already upper triangular, so neither additional QR factorizations nor preconditioning is required.

The figures presenting the results for all five examples follow a common layout. Panel (a) shows the computed pseudospectral contours, verifying that the proposed algorithms produce sensible results. Panel (b) shows the recycling points as blue dots (\pselegenddot{psePrelim}), namely the grid points at which the singular subspaces are refined by \texttt{lobpcgSvd} and stored for recycling. Panel (c) reports the corresponding recycling statistics of \texttt{recySing}, including the number $p$ of recycled singular subspaces and the value of $\tilde r$ determined by \cref{tau}. These quantities reflect the difficulty of singular subspace recycling.

For the two dense examples, panels (d)--(f) compare \texttt{psTriRecy} with the Lanczos-based method \texttt{psInvLanc} on increasingly finer grids with $n_{\mathrm{grid}}=200,400,800$. Panel (d) reports the operation counts, panel (e) shows the execution times, and panel (f) provides a breakdown of the total execution time, including the time spent on LOBPCG-SVD (\pselegendbox{pseLobpcg}), inverse Lanczos iteration (\pselegendbox{pseInvLanc}), \texttt{updateRecy} (\pselegendbox{pseUpdateRecy}), \texttt{recySing} (\pselegendbox{pseRecySing}), and \texttt{preliminary} (\pselegendbox{psePrelim}).

For the sparse examples, panels (d)--(f) instead summarize the performance of \texttt{psPrecRecy}. Panel (d) shows the grid points at which singular subspace recycling (\pselegenddot{psePrelim}), projection subspace recycling (\pselegenddot{psePredPts}), and reconstruction of the base preconditioner (\pselegendcross{black}) take place. Panel (e) reports statistics related to adaptive preconditioner recycling. Panel (f) presents a breakdown of the total execution time for all three algorithms into the time spent on LOBPCG-SVD (\pselegendbox{pseLobpcg}), inverse Lanczos iteration (\pselegendbox{pseInvLanc}), \texttt{updateRecy} (\pselegendbox{pseUpdateRecy}), \texttt{recySing} (\pselegendbox{pseRecySing}), preconditioner-related computations (\pselegendbox{psePrecond}), and \texttt{qr} (\pselegendbox{pseQr}). In the following subsections, we focus only on features that are specific to each example.

\subsection{\texorpdfstring{landau\_$40\pi$}{landau-40pi}}
\begin{figure}[t!]
\centering
\subfloat[\centering pseudospectra contours]{\includegraphics[width=0.32\linewidth]{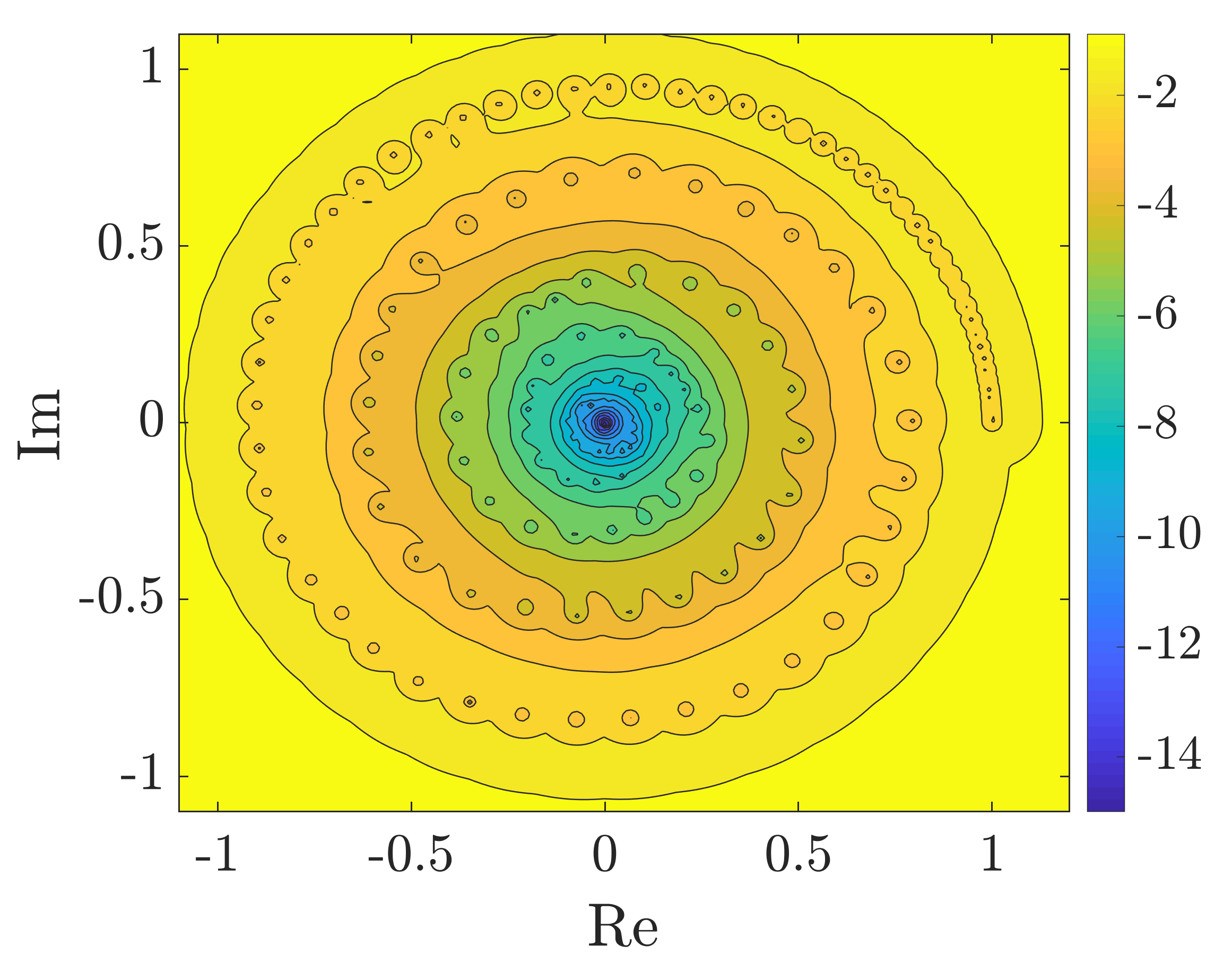}\label{fig:landau-contour}}
\hfill
\subfloat[\centering recycling points]{\includegraphics[width=0.32\linewidth]{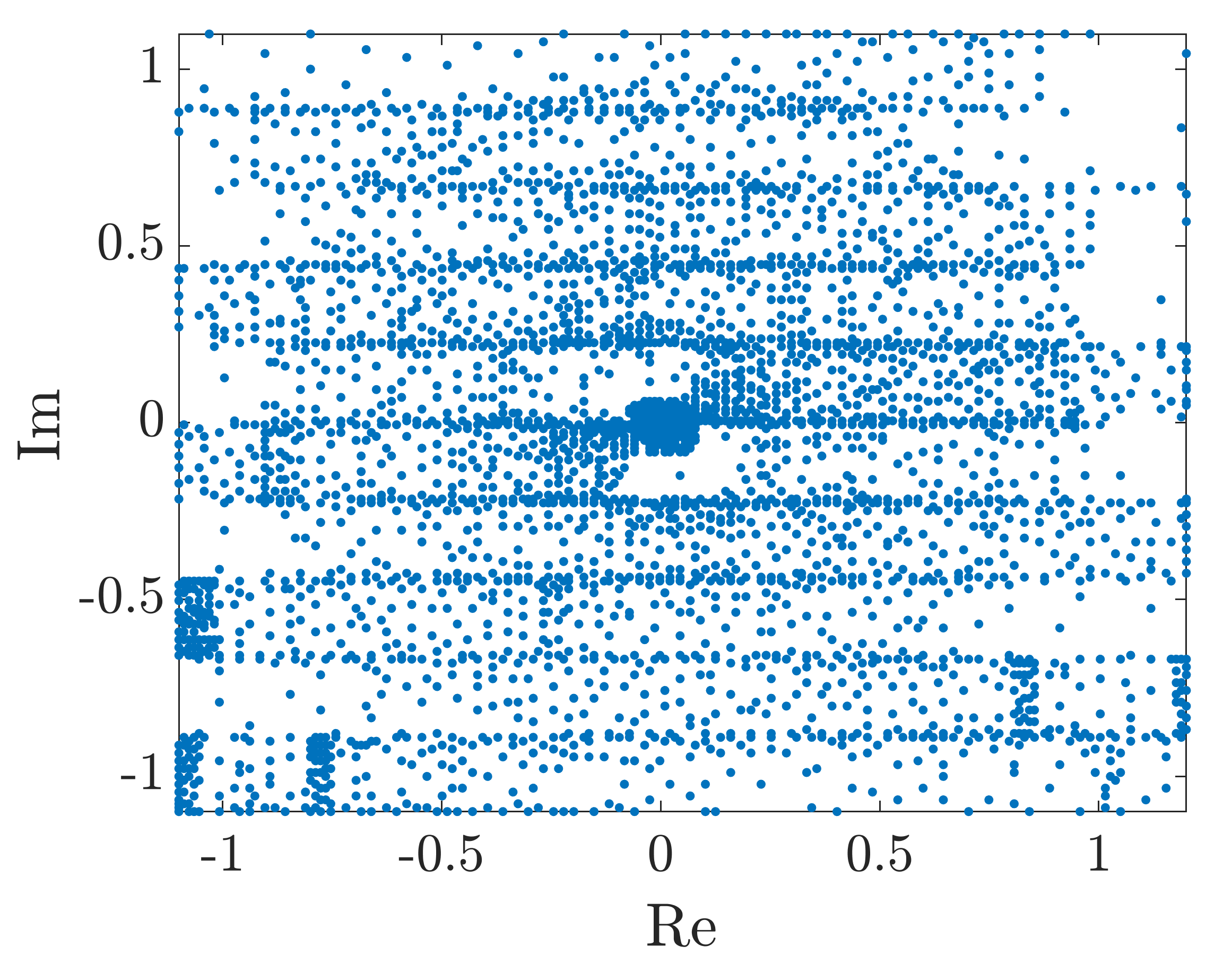}\label{fig:landau-pts}}
\hfill
\subfloat[\centering recycling statistics]{\includegraphics[width=0.32\linewidth]{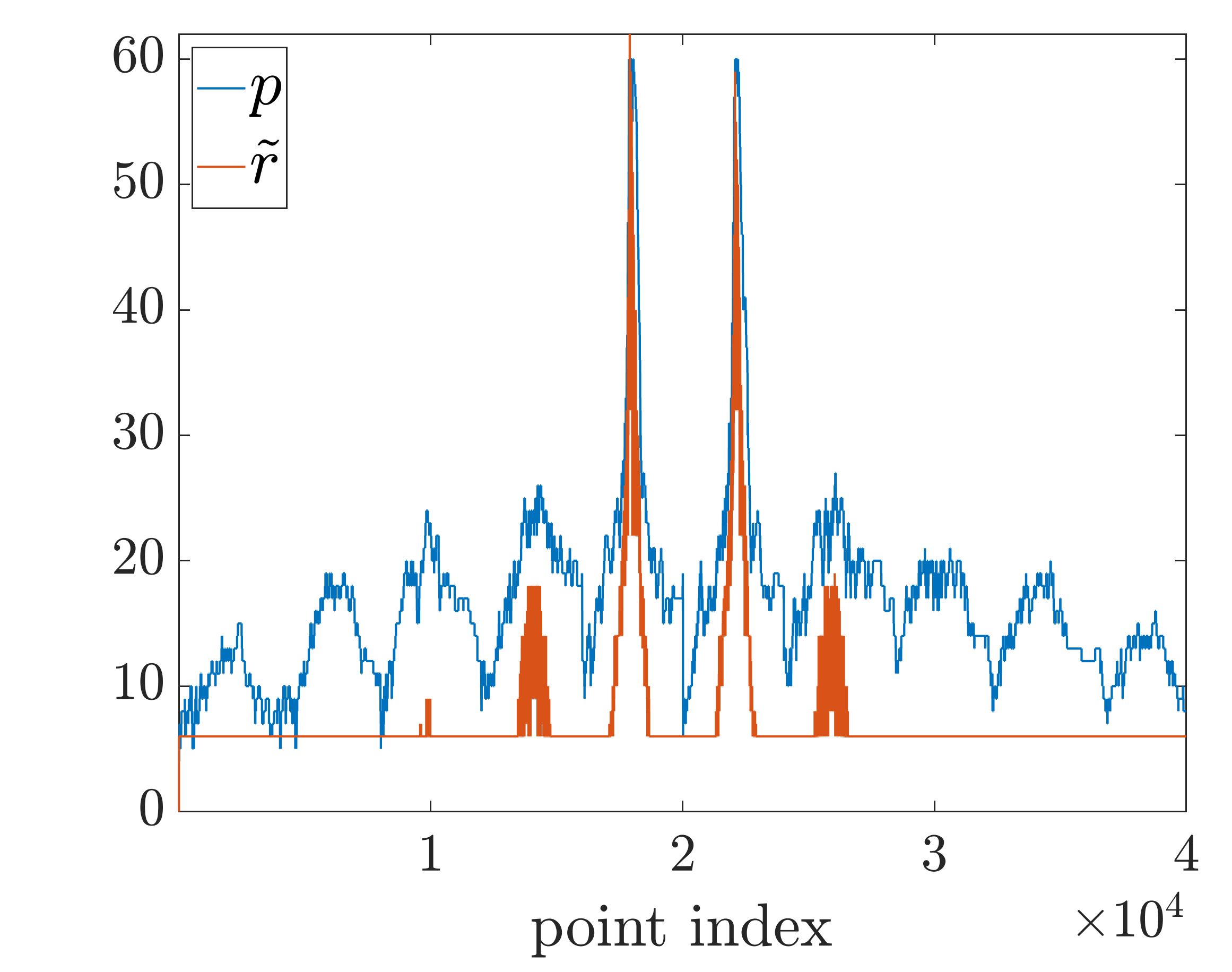}\label{fig:landau-stats}}
\par\medskip
\subfloat[\centering operation counts]{\includegraphics[width=0.32\linewidth]{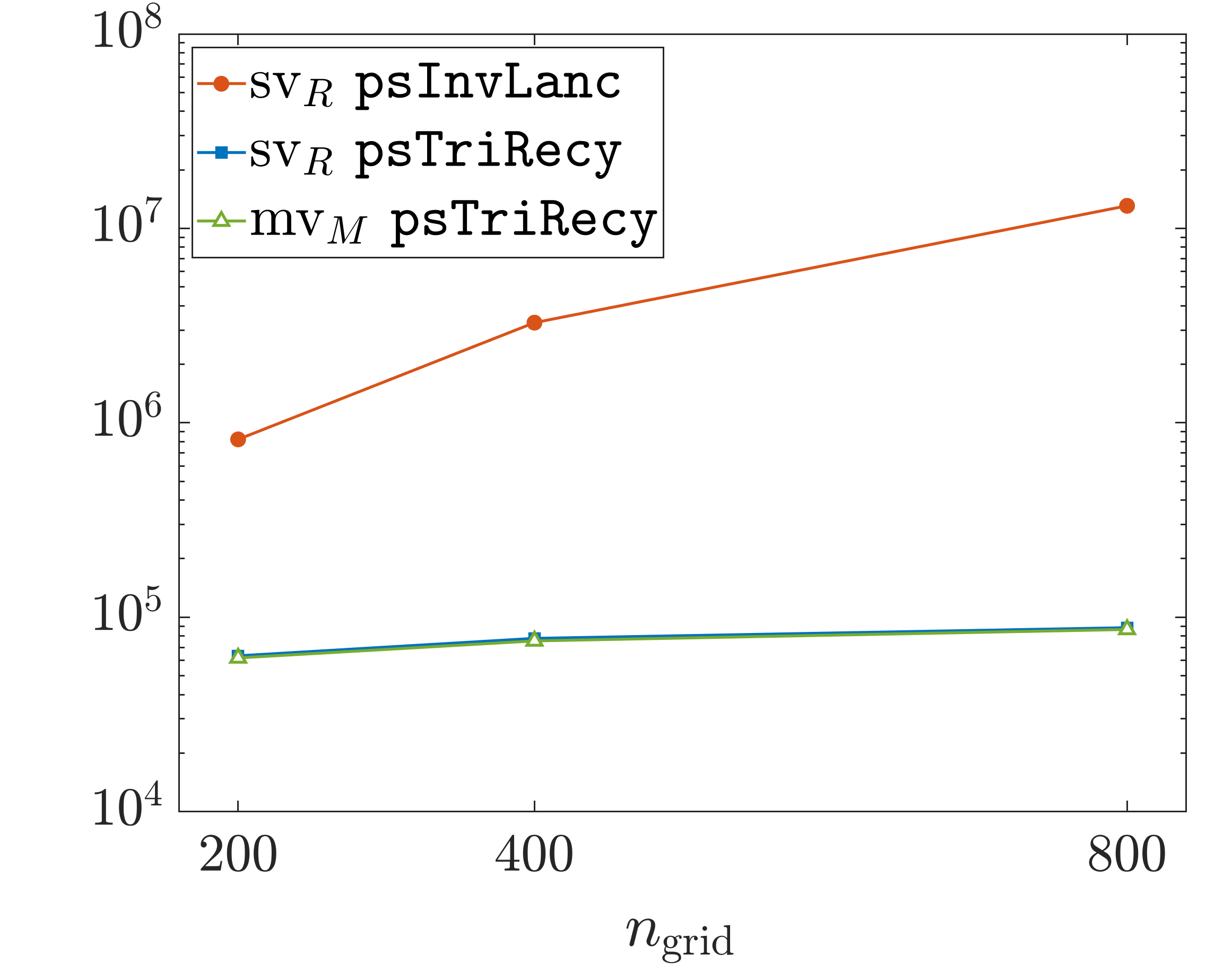}\label{fig:landau-ops}}
\hfill
\subfloat[\centering execution time]{\includegraphics[width=0.32\linewidth]{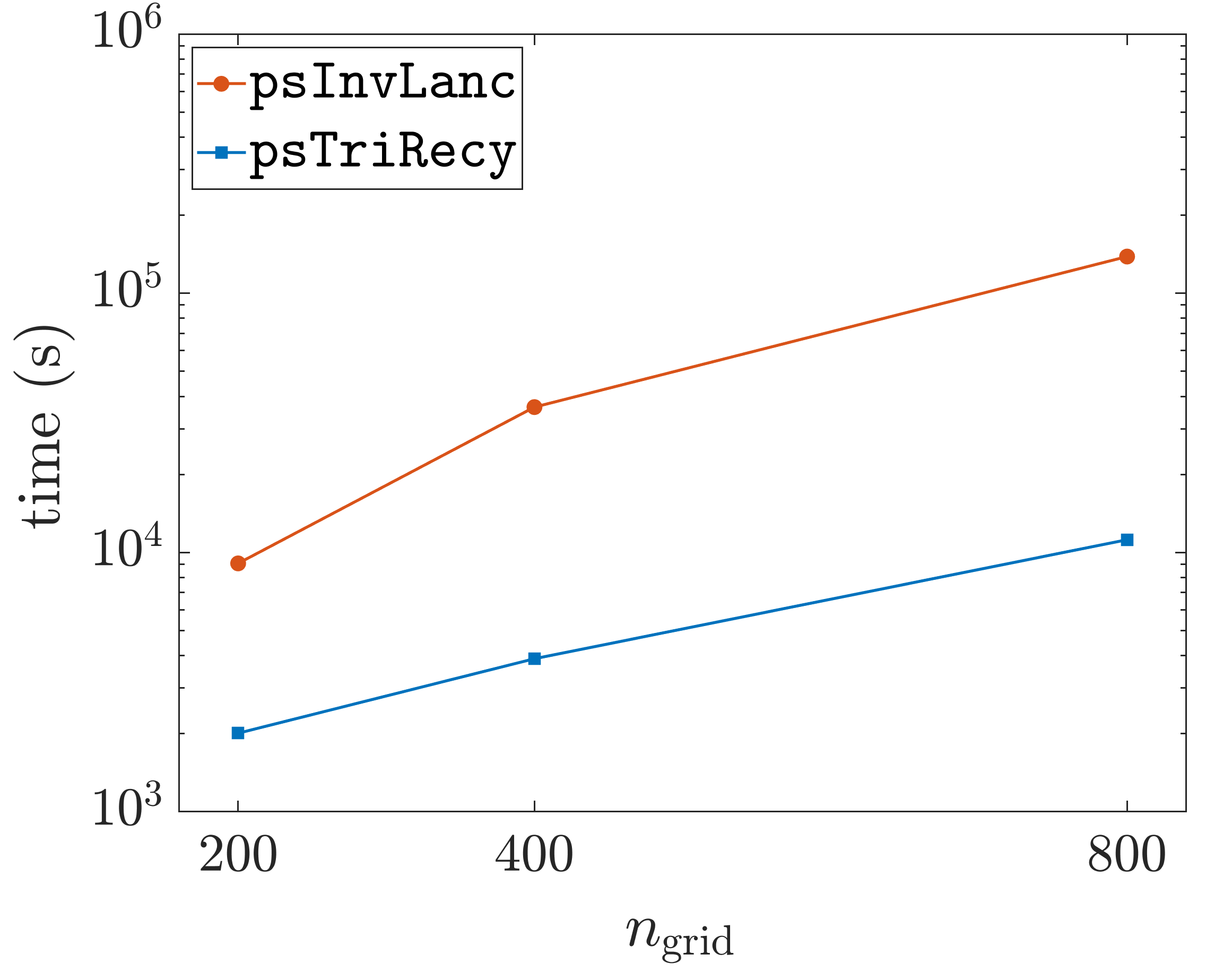}\label{fig:landau-elapsed}}
\hfill
\subfloat[\centering execution time breakdown]{\includegraphics[width=0.32\linewidth]{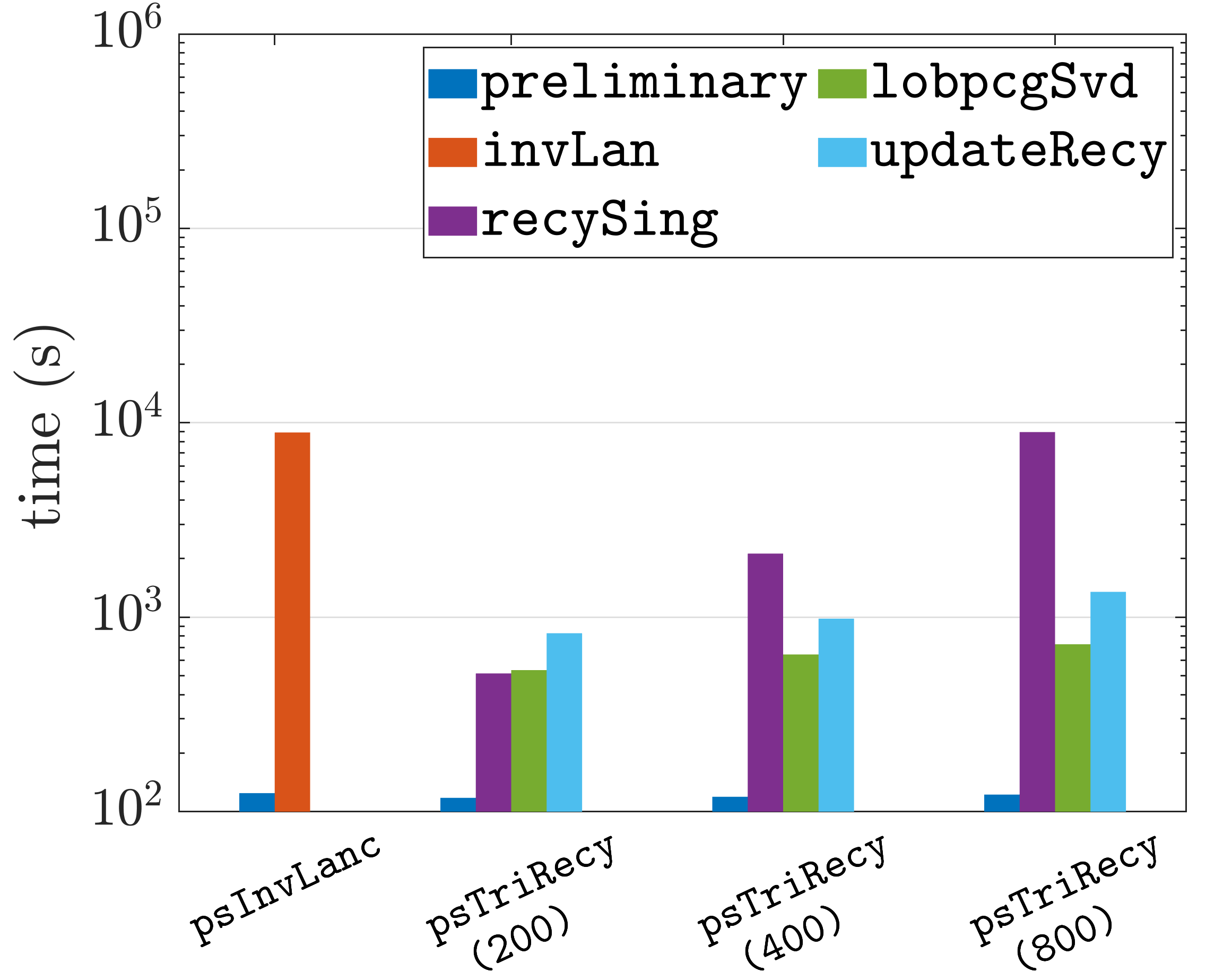}\label{fig:landau-time}}
\caption{Results for landau\_$40\pi$.}\label{fig:landau-results}
\end{figure}
Our first example is the matrix obtained by discretizing the Huygens--Fresnel operator for Fresnel number $40\pi$ with $5{,}000$ discretization points. This operator is used to model the laser problem and serves as a classical test problem in computing pseudospectra \cite[\S 60]{tre2}. The contours in \cref{fig:landau-contour} are a spot-on match to those obtained by ${\tt psInvLanc}$ and \textsc{EigTool}.

It is at only $3{,}430$ of the $40{,}000$ grid points that the refinement of the singular subspace via LOBPCG-SVD is required and the resulting right singular subspace is incorporated into the recycling subspace. These recycling grid points are plotted in \cref{fig:landau-pts}, serving as an indicator of the efficiency of the proposed recycling strategy. In other words, the right singular subspace is obtained without invoking LOBPCG-SVD at more than $90\%$ of the grid points. \cref{fig:landau-stats} further shows how the number $p$ of the recycled singular subspaces and the effective dimension $\tilde r$ are adapted across the entire grid. The two large spikes are associated with grid points near the origin, where recycling adaptively works by using large values of $p$ and $\tilde r$; in between, the singular subspaces have to be refined by LOBPCG-SVD. The adaptivity in recycling singular subspaces is crucial for efficiency, as it allows the fast Rayleigh--Ritz-SVD procedure to be applied to an effective subspace of minimal dimensions at a cost of $\mO(pr\tilde{r}n)$ flops, thereby avoiding Rayleigh--Ritz-SVD applied to the entire recycling subspace, which would require $\mO(p^2r^2n)$ flops.

The advantage of recycling becomes more pronounced as the grid becomes denser. As shown in \cref{fig:landau-ops,fig:landau-elapsed}, the operation counts and execution time of ${\tt psTriRecy}$ grow much more slowly than those of ${\tt psInvLanc}$. The corresponding speedups are $4.52$, $9.35$, and $12.37$ for $n_{\mathrm{grid}}=200$, $400$, and $800$, respectively. The breakdown in \cref{fig:landau-time} shows that, as the grid becomes denser, the dominant cost shifts from repeated singular-subspace refinement via LOBPCG-SVD to the cheaper adaptive recycling procedure ${\tt recySing}$.

\subsection{basor}
\begin{figure}[t!]
\centering
\subfloat[\centering pseudospectral contour]{\includegraphics[width=0.32\linewidth]{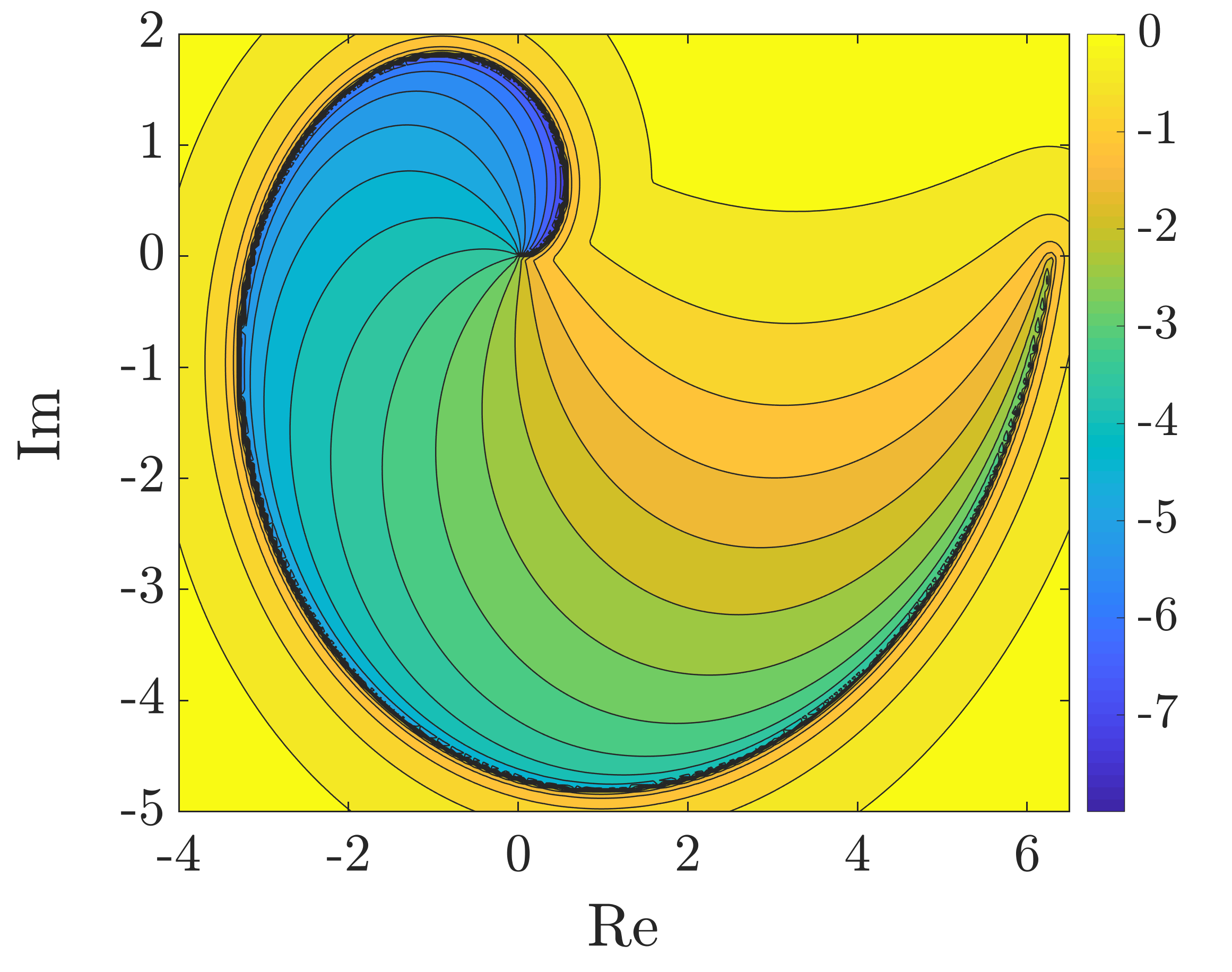}\label{fig:basor-contour}}
\hfill
\subfloat[\centering recycling points]{\includegraphics[width=0.32\linewidth]{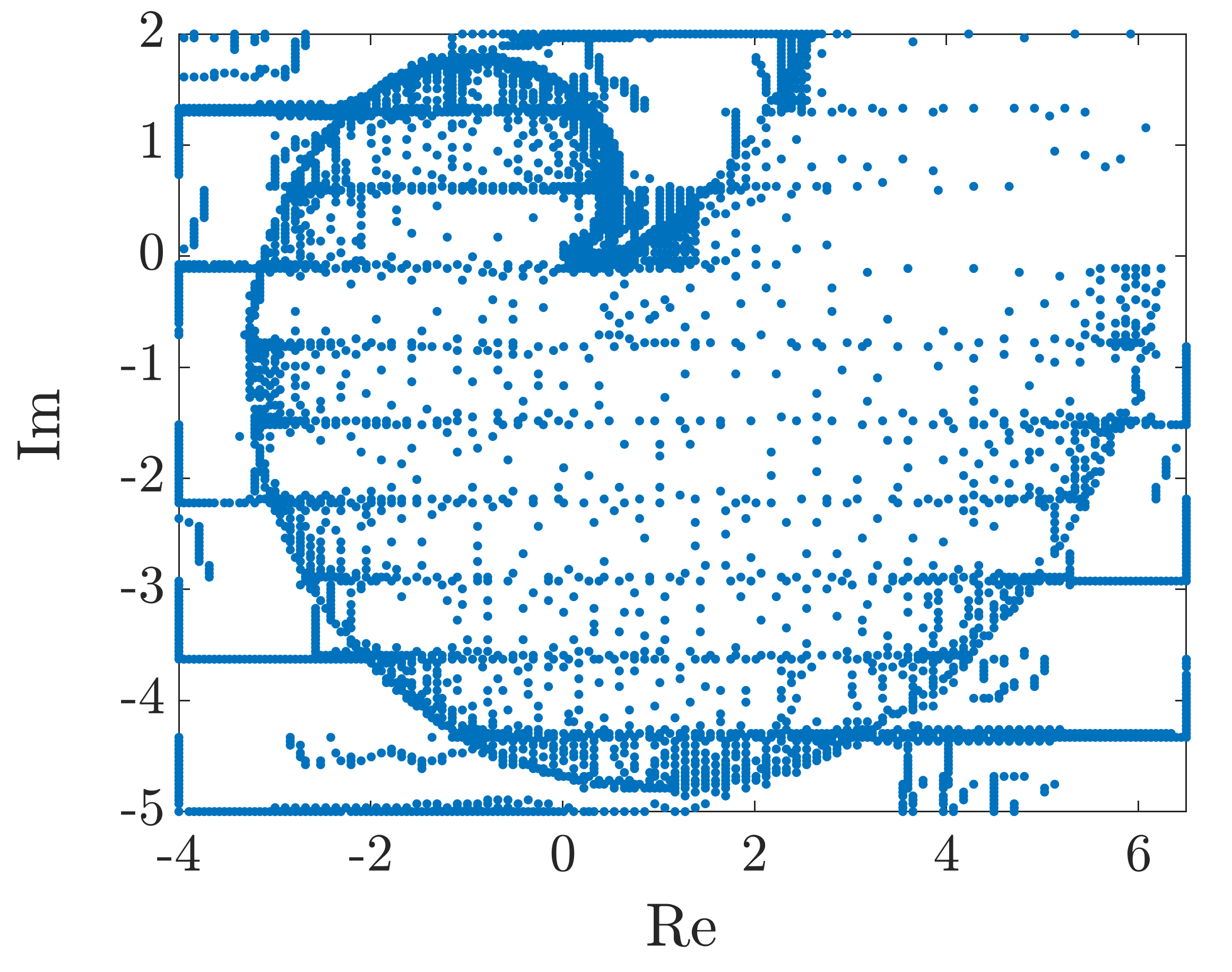}\label{fig:basor-pts}}
\hfill
\subfloat[\centering recycling statistics]{\includegraphics[width=0.32\linewidth]{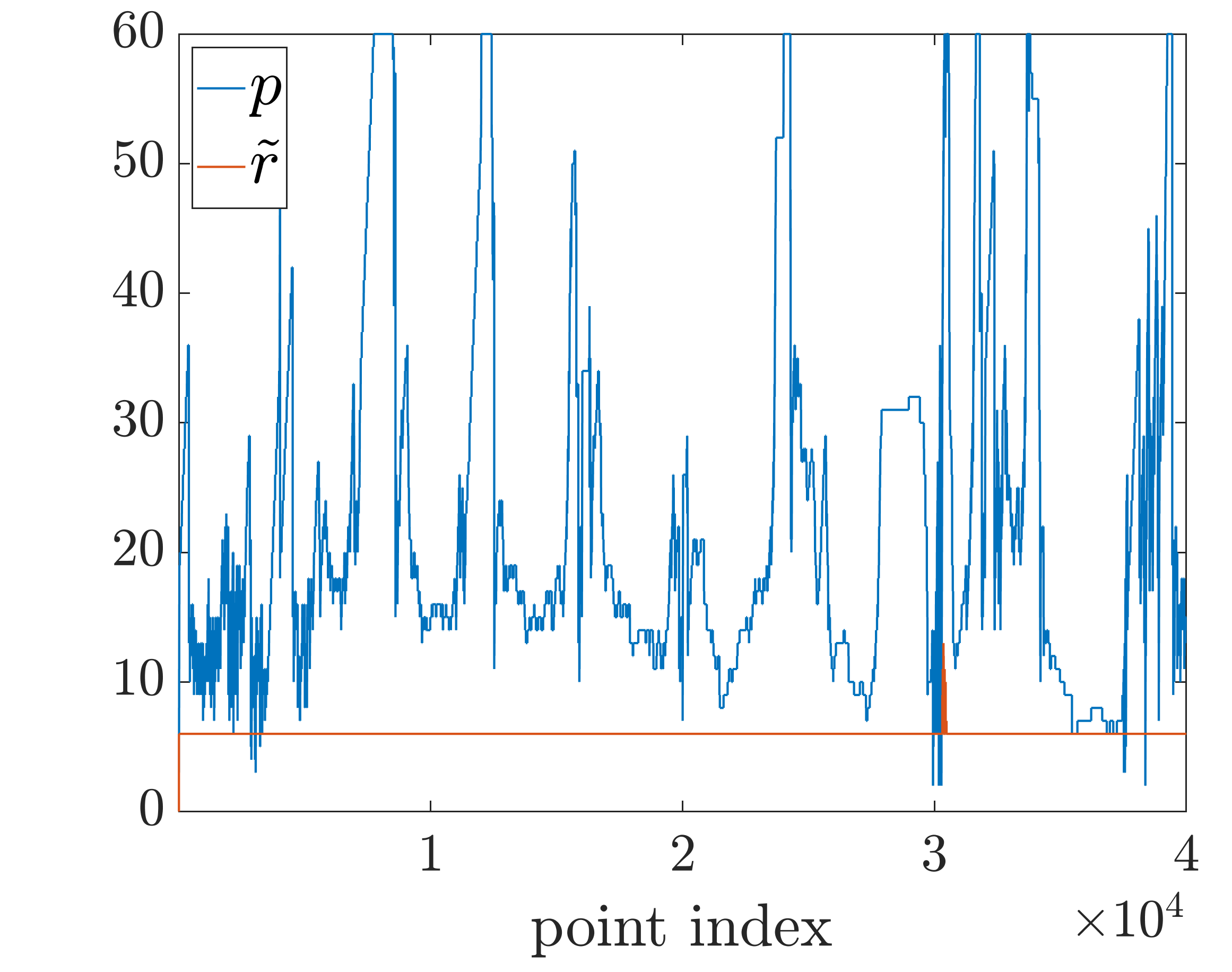}\label{fig:basor-stats}}
\par\medskip
\subfloat[\centering operation counts]{\includegraphics[width=0.32\linewidth]{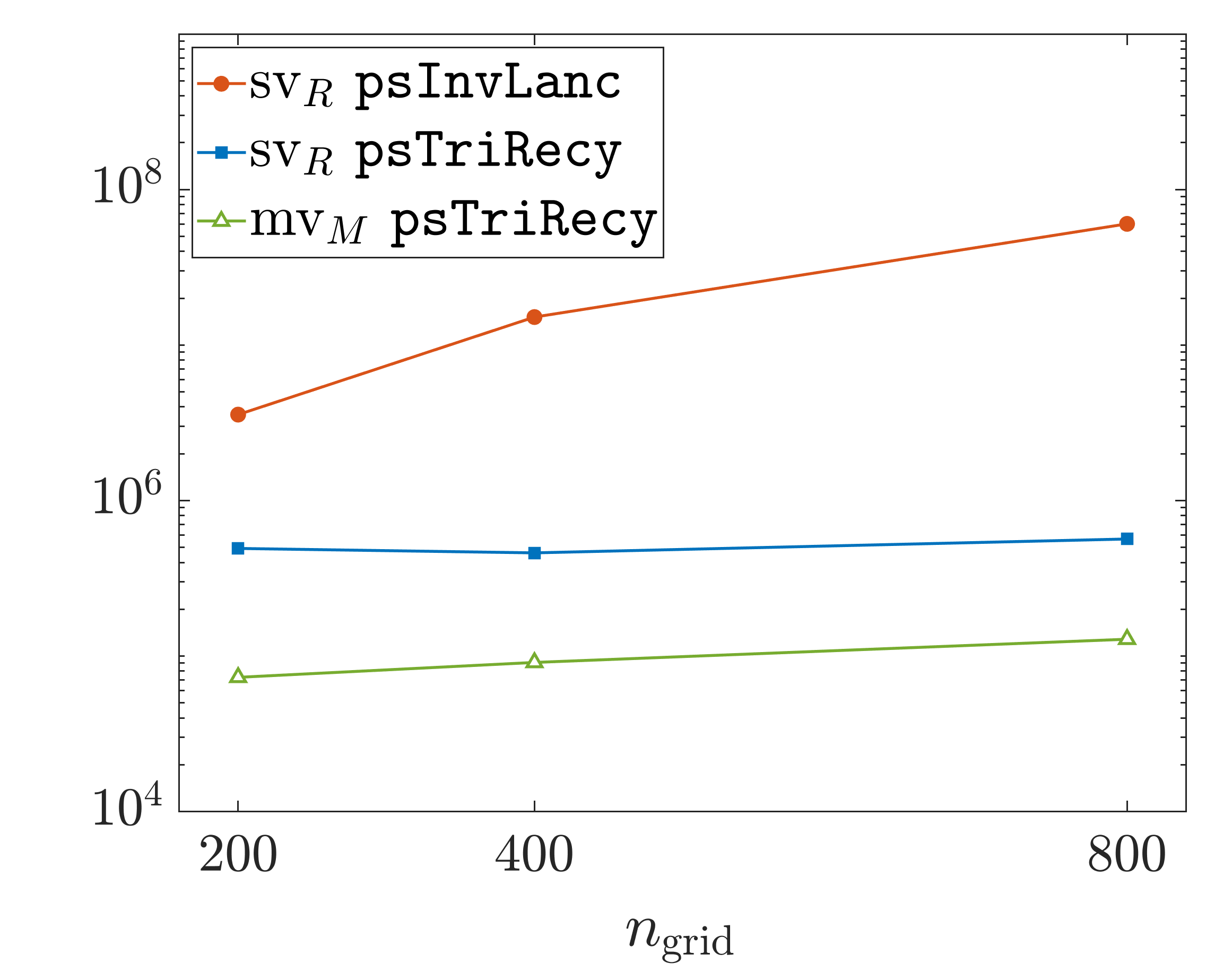}\label{fig:basor-ops}}
\hfill
\subfloat[\centering execution time]{\includegraphics[width=0.32\linewidth]{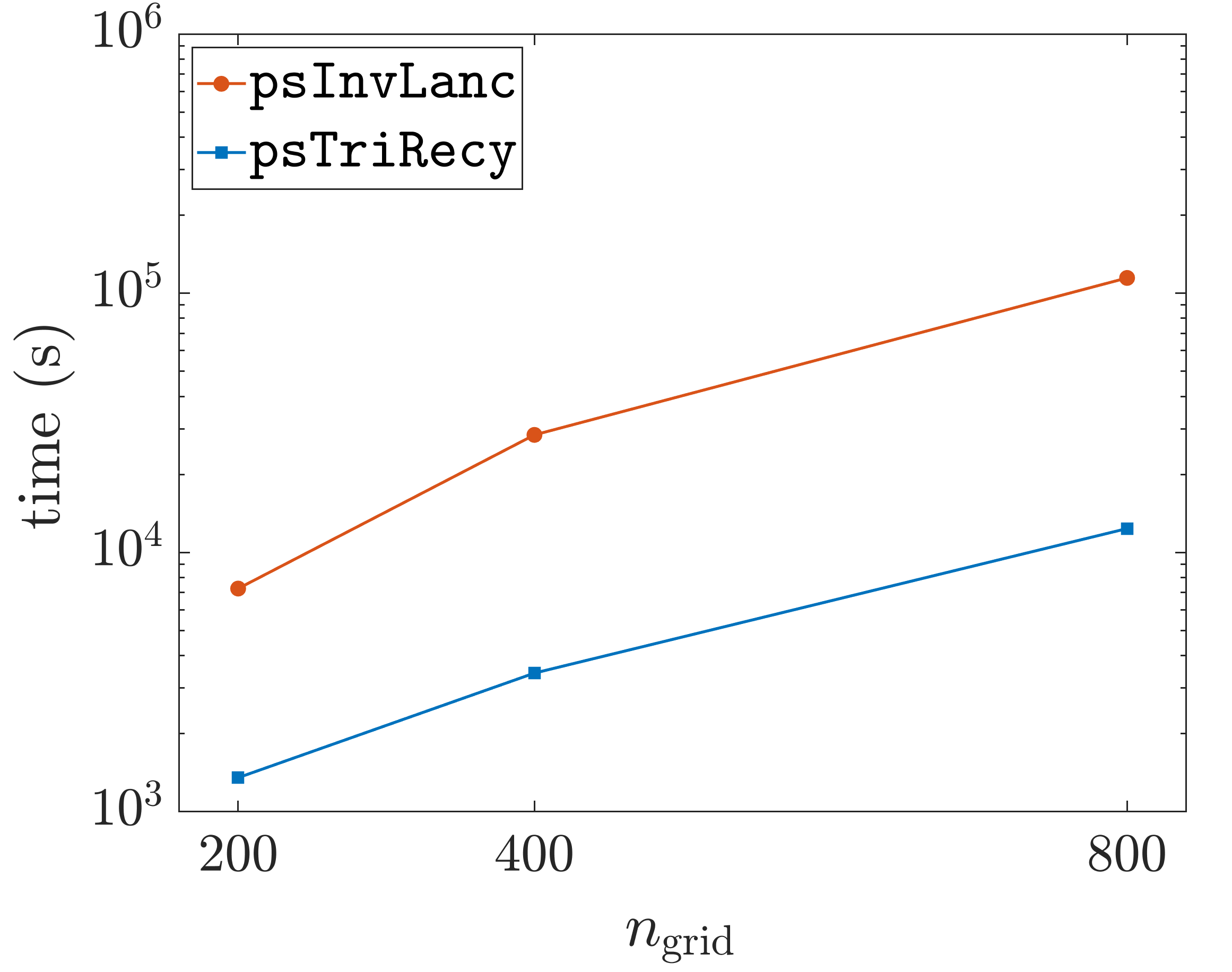}\label{fig:basor-elapsed}}
\hfill
\subfloat[\centering execution time breakdown]{\includegraphics[width=0.32\linewidth]{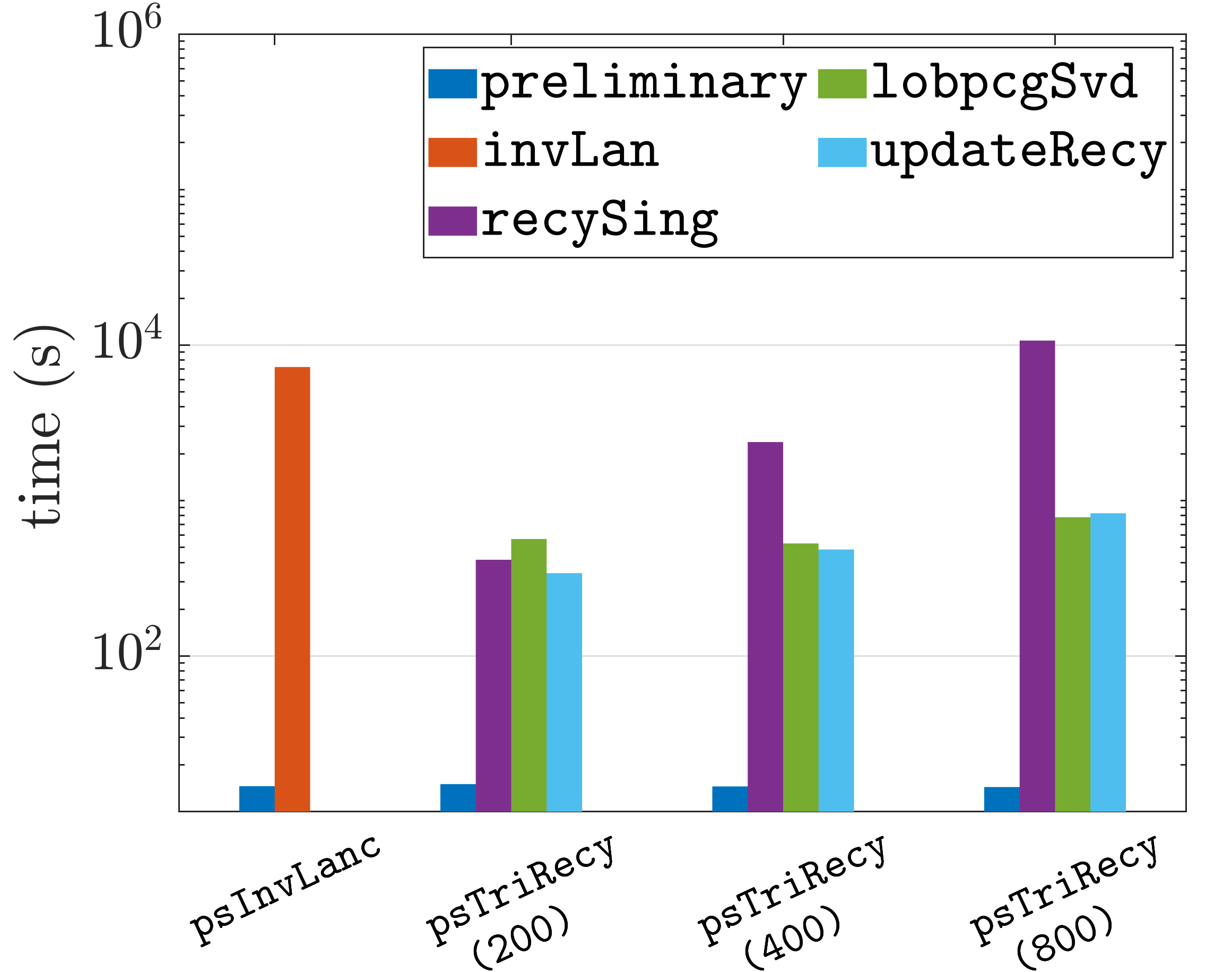}\label{fig:basor-time}}
\caption{Results for basor.}
\label{fig:basor-results}
\end{figure}
The basor matrix from \textsc{EigTool} \cite{wri3} is a Toeplitz example with a piecewise continuous symbol, and its pseudospectral contours are shown in \cref{fig:basor-contour}.

This example has a more localized refinement pattern than the previous one. Refinement of the singular subspace is required at only $4{,}056$ of the $40{,}000$ grid points, and these points are concentrated near the spectrum, where the local singular value problem is most difficult; see \cref{fig:basor-pts}. The recycling statistics in \cref{fig:basor-stats} show a stronger variation in the number of recycled blocks, highlighting the need for adaptive basis updates. In fact, in this example $\tilde r=r$ for most grid points, which makes the fast Rayleigh--Ritz-SVD particularly efficient. The few exceptional points with $\tilde r>r$ show why the dimension of the effective subspace should be chosen adaptively rather than fixed.

The speedup of ${\tt psTriRecy}$ over ${\tt psInvLanc}$ increases from $5.37$ to $8.30$ and $9.27$ for $n_{\mathrm{grid}}=200$, $400$, and $800$, respectively; see \cref{fig:basor-elapsed}. As shown by \cref{fig:basor-time}, the main costly part shifts from repeated triangular solves to ${\tt recySing}$, just as in the previous example.

\subsection{af23560}
\begin{figure}[t!]
\centering
\subfloat[\centering pseudospectral contour]{\includegraphics[width=0.32\linewidth]{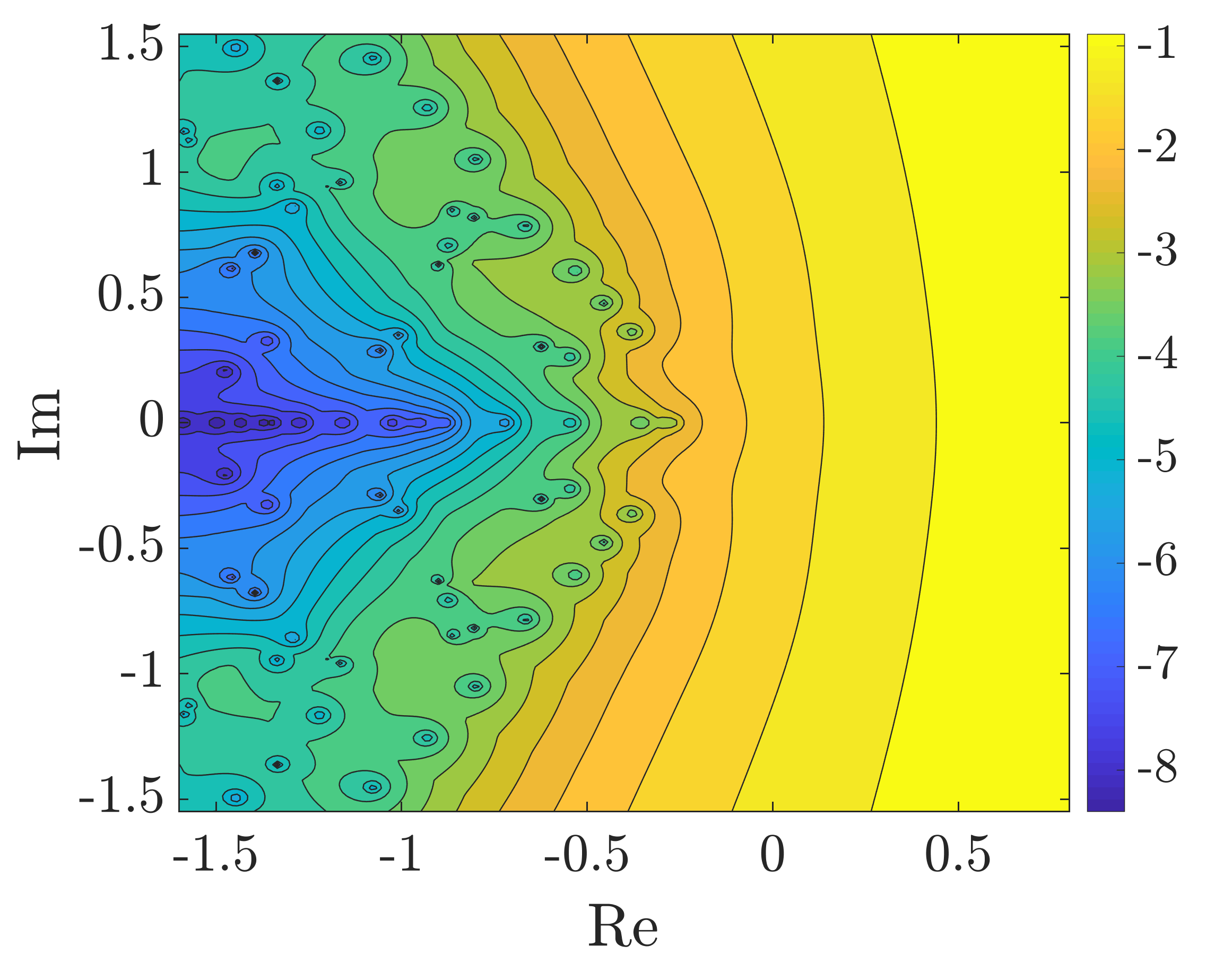}\label{fig:af23560-contour}}
\hfill
\subfloat[\centering recycling points]{\includegraphics[width=0.32\linewidth]{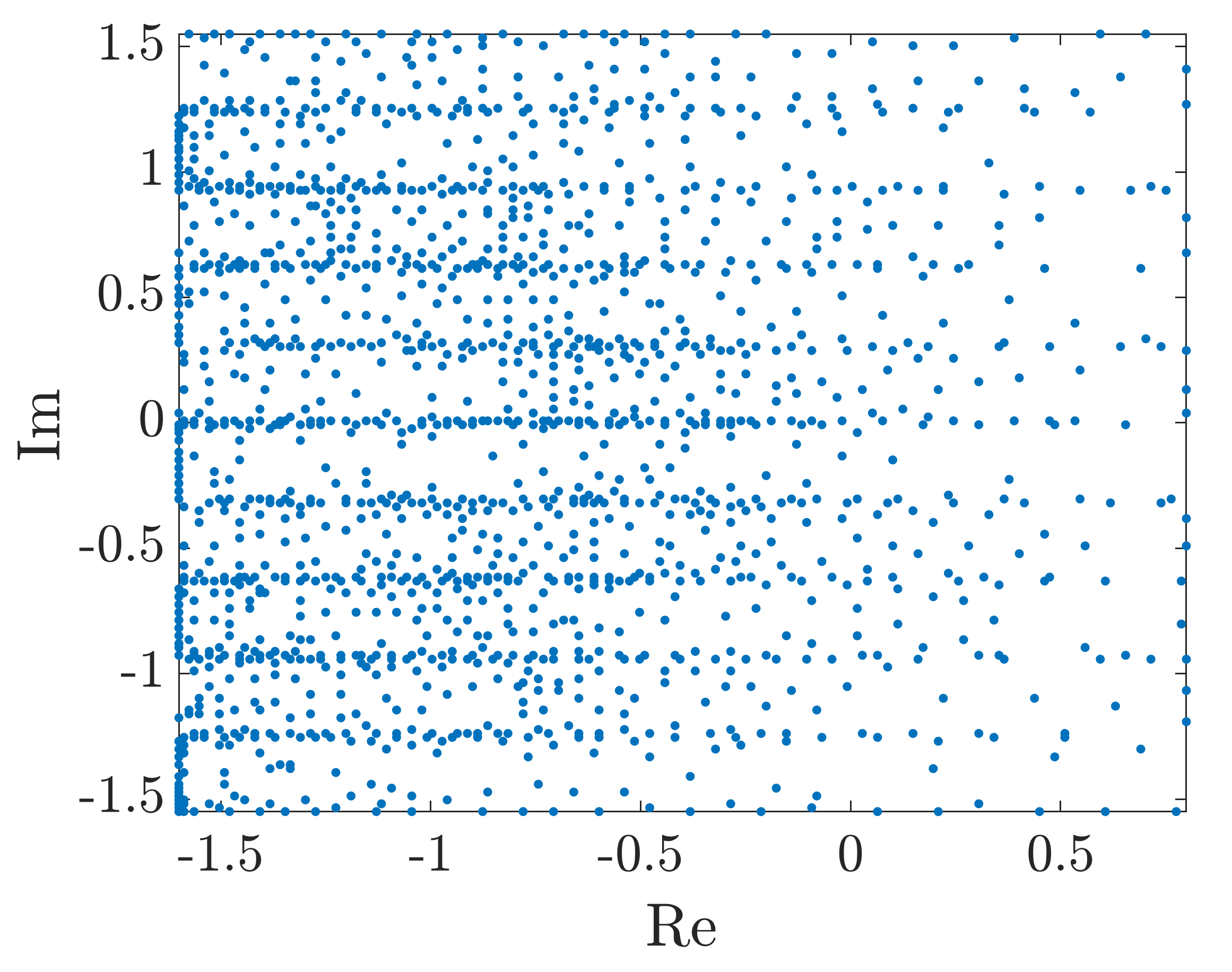}\label{fig:af23560-pts}}
\hfill
\subfloat[\centering {\tt psTriRecy} statistics]{\includegraphics[width=0.32\linewidth]{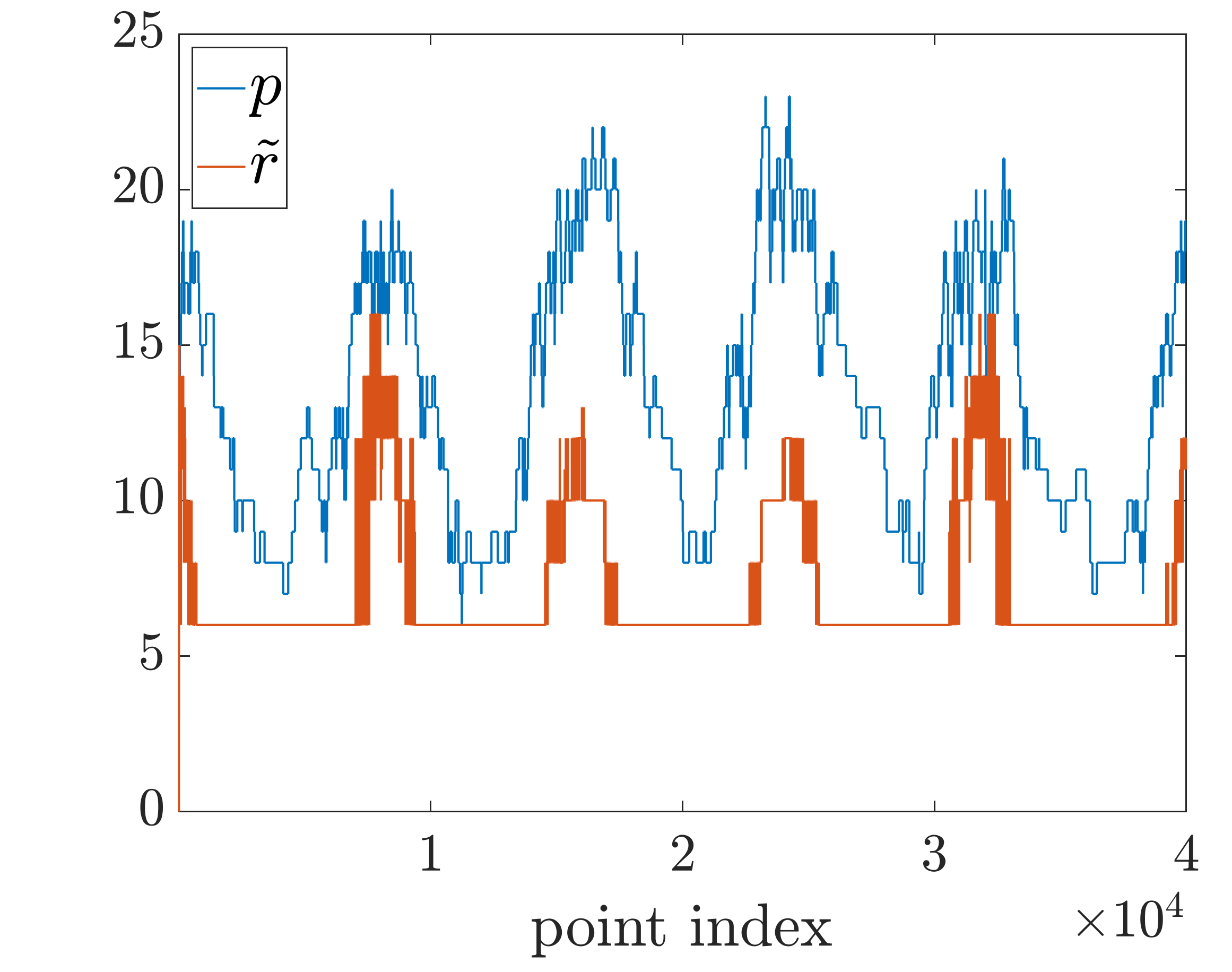}\label{fig:af23560-tri-stats}}
\par\medskip
\subfloat[\centering recycling points]{\includegraphics[width=0.32\linewidth]{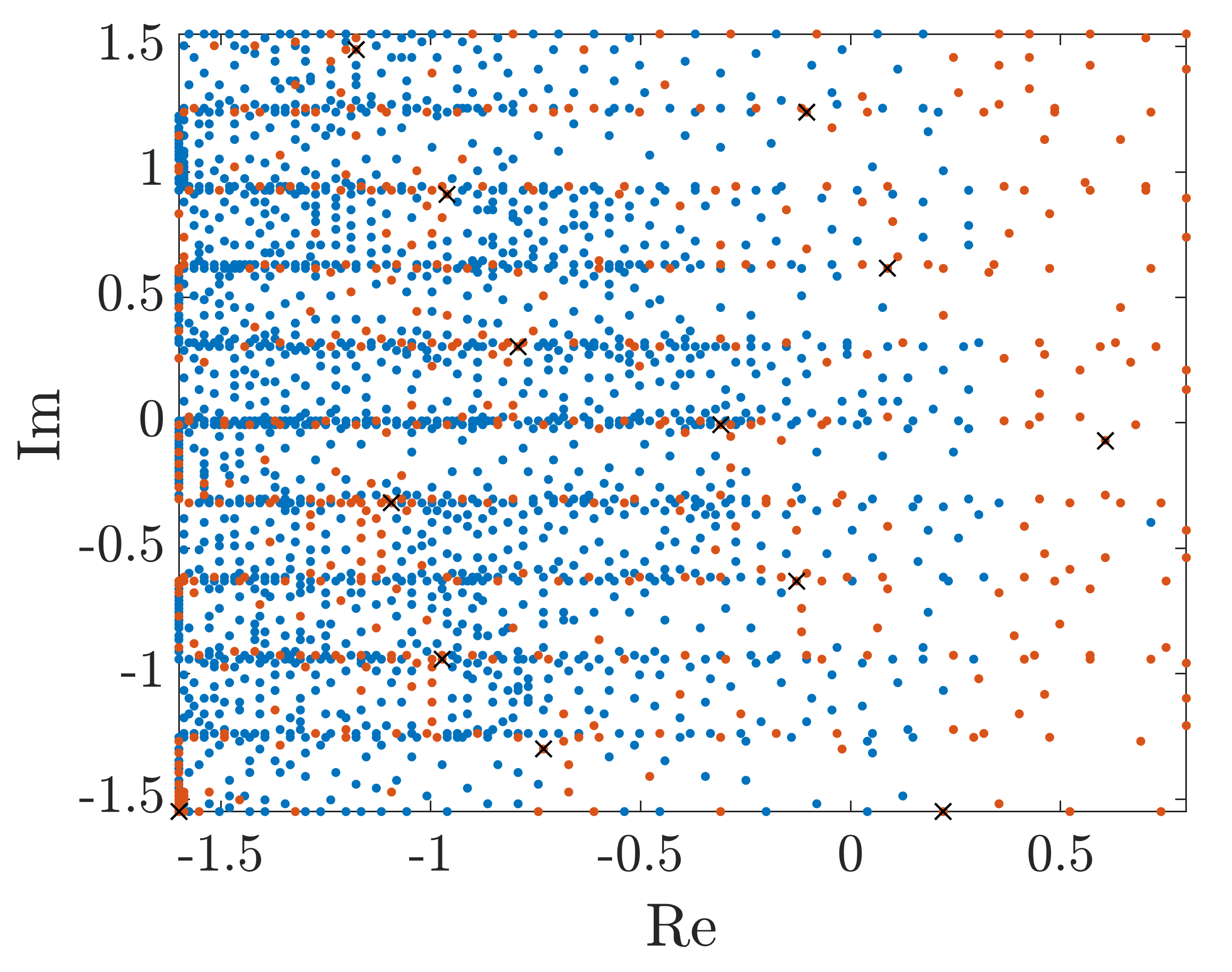}\label{fig:af23560-pred-pts}}
\hfill
\subfloat[\centering {\tt psPrecRecy} statistics]{\includegraphics[width=0.32\linewidth]{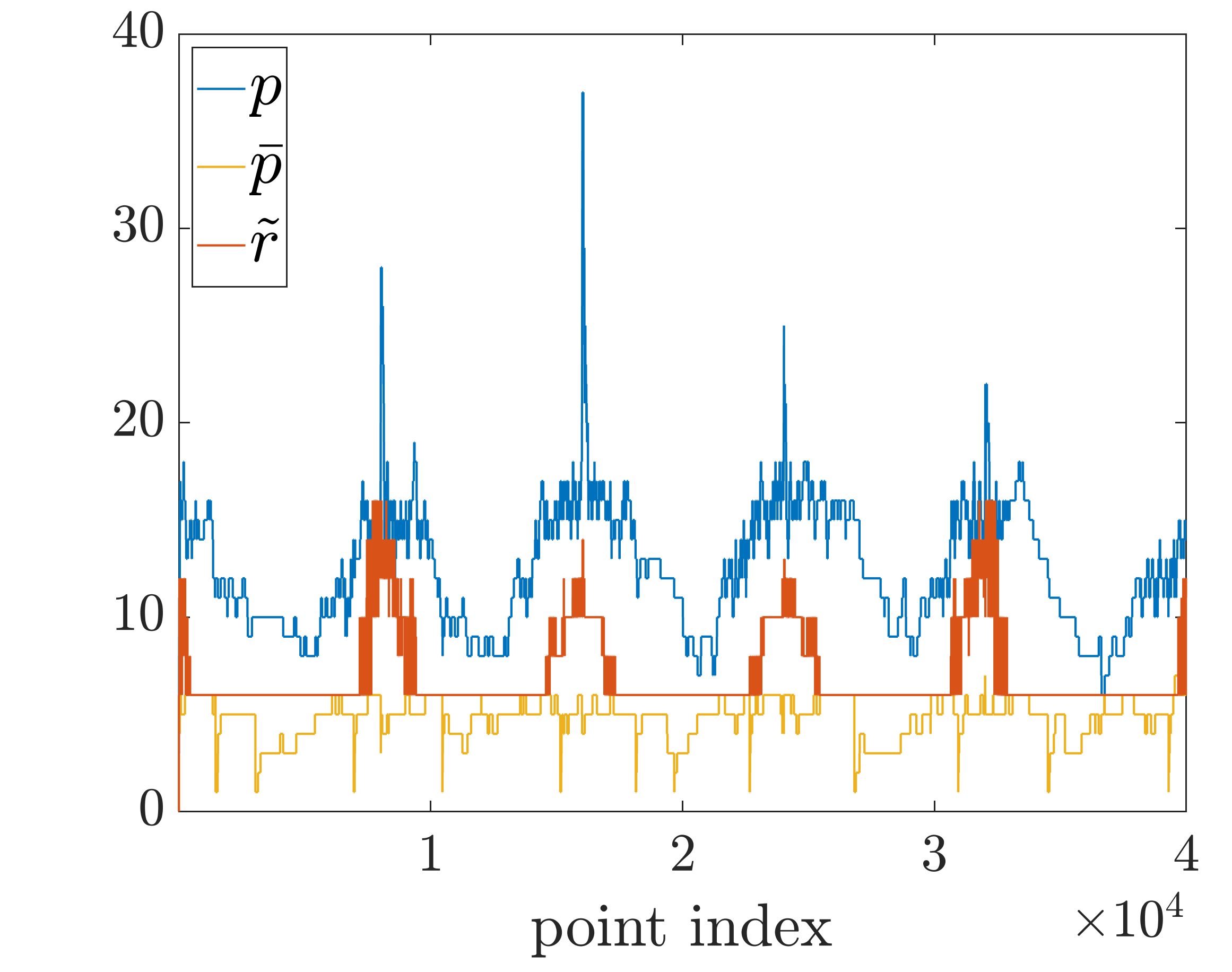}\label{fig:af23560-pred-stats}}
\hfill
\subfloat[\centering execution time breakdown]{\includegraphics[width=0.32\linewidth]{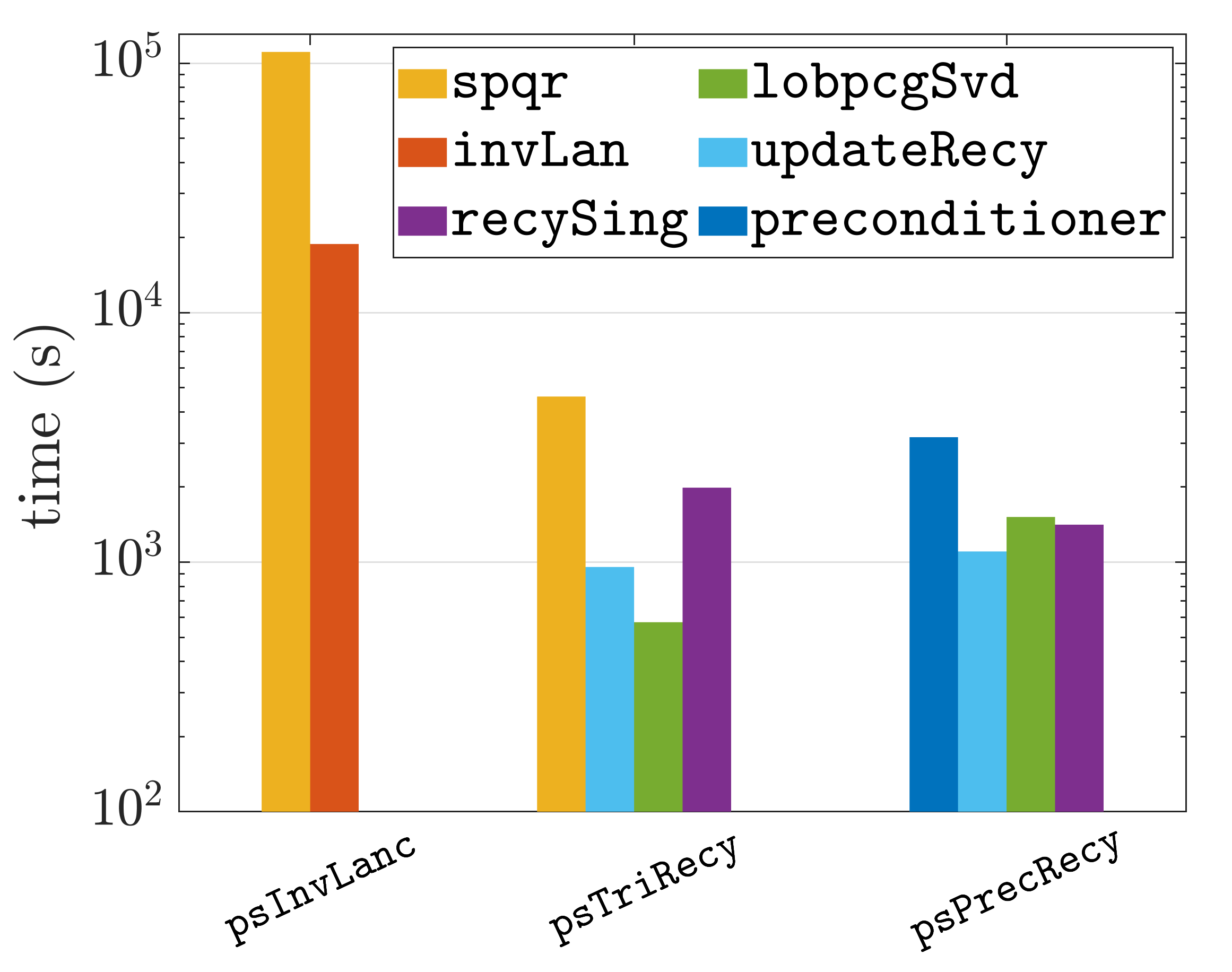}\label{fig:af23560-time}}
\caption{Results for af23560.}
\label{fig:af23560-results}
\end{figure}
The af23560 matrix, taken as an example in \cite{wri1} and available from the Matrix Market collection \cite{boi}, provides a large-scale sparse test problem. As shown in \cref{fig:af23560-contour}, our method produces the correct pseudospectral contour, whereas the contours obtained by projection onto Krylov subspaces are correct only in their rightmost part \cite{wri1}.

At only $1{,}530$ of the $40{,}000$ grid points does ${\tt psTriRecy}$ require refinement of the singular subspaces via sparse QR; see \cref{fig:af23560-pts}. Consequently, ${\tt psTriRecy}$ spends only $4{,}613$ seconds on QR-related computations, compared to $110{,}985$ for ${\tt psInvLanc}$; see \cref{fig:af23560-time}. The statistics characterizing the adaptivity in \cref{fig:af23560-tri-stats} show that $p$ and $\tilde{r}$ vary significantly across the grid due to the local difficulty.

When this problem is solved using ${\tt psPrecRecy}$, the cost of the expensive sparse QR for triangularization in ${\tt psTriRecy}$ is saved by the even lower cost of applying the preconditioner and recycling the projection subspace; see \cref{fig:af23560-time}. More precisely, only $3{,}169$ seconds are spent on preconditioner-related computations. Among the $2{,}088$ singular-subspace recycling points determined by ${\tt psPrecRecy}$, only $499$ and $13$ of them correspond to projection-subspace recycling and base-preconditioner reconstruction, respectively; see \cref{fig:af23560-pred-pts}. Since there are slightly more recycling points than in ${\tt psTriRecy}$, as shown in \cref{fig:af23560-pred-stats}, ${\tt psPrecRecy}$ is not significantly faster than ${\tt psTriRecy}$. The execution times are $8{,}150$ and $7{,}333$ seconds for ${\tt psTriRecy}$ and ${\tt psPrecRecy}$, respectively, corresponding to speedups of $15.94$ and $17.72$ over ${\tt psInvLanc}$.

\subsection{skewlap3d}
\begin{figure}[t!]
\centering
\subfloat[\centering pseudospectral contour]{\includegraphics[width=0.32\linewidth]{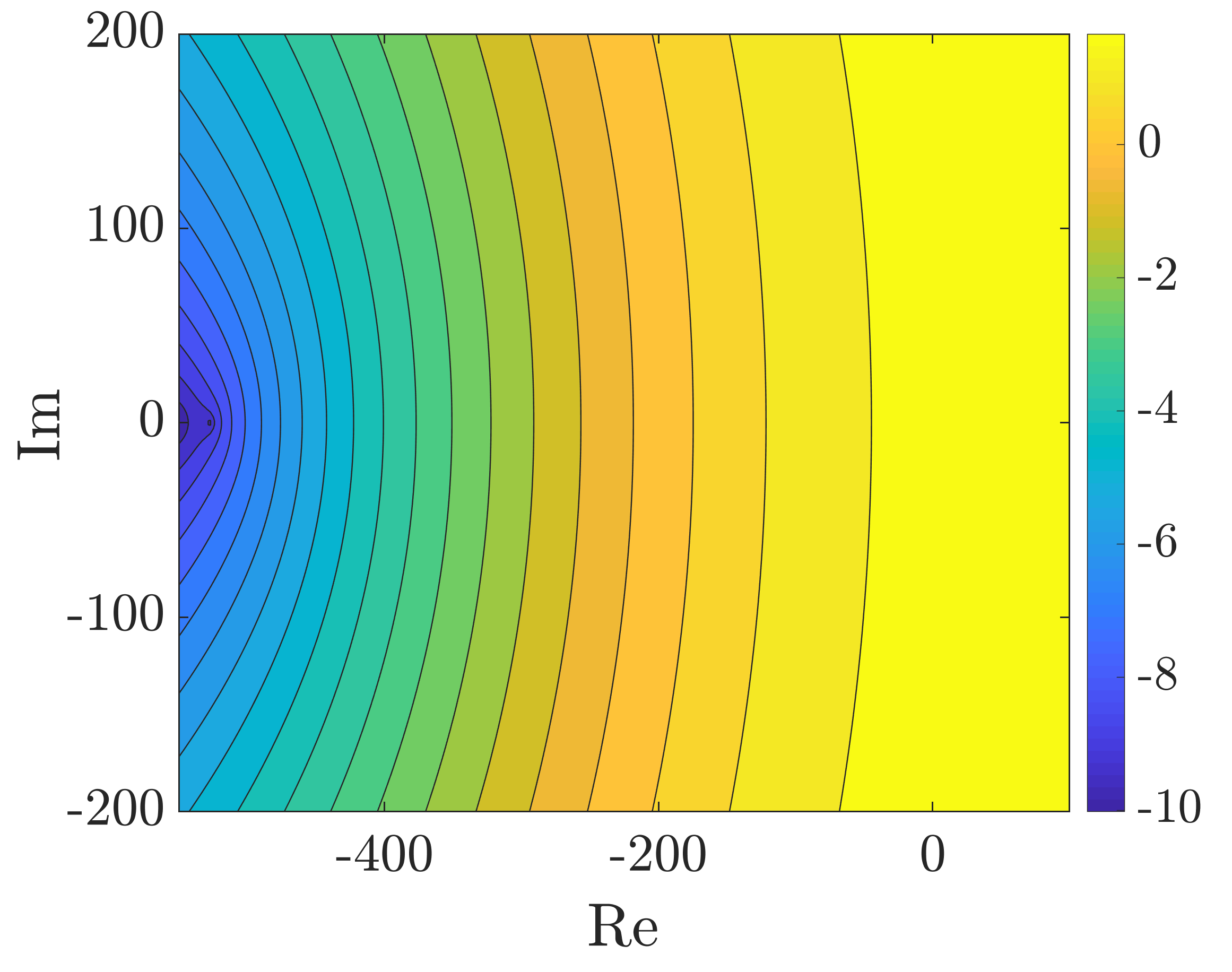}\label{fig:skewlap3d-contour}}
\hfill
\subfloat[\centering recycling points]{\includegraphics[width=0.32\linewidth]{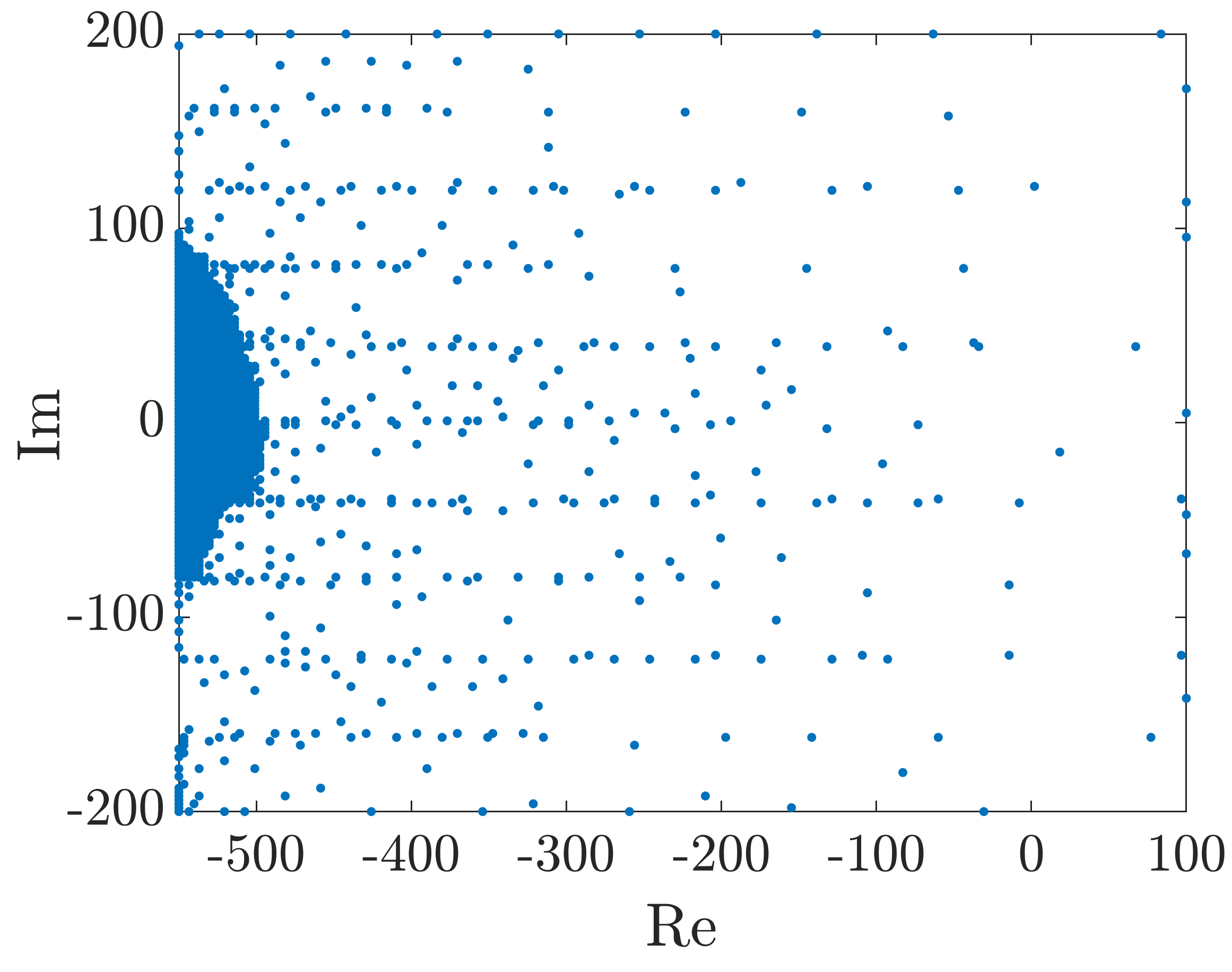}\label{fig:skewlap3d-pts}}
\hfill
\subfloat[\centering {\tt psTriRecy} statistics]{\includegraphics[width=0.32\linewidth]{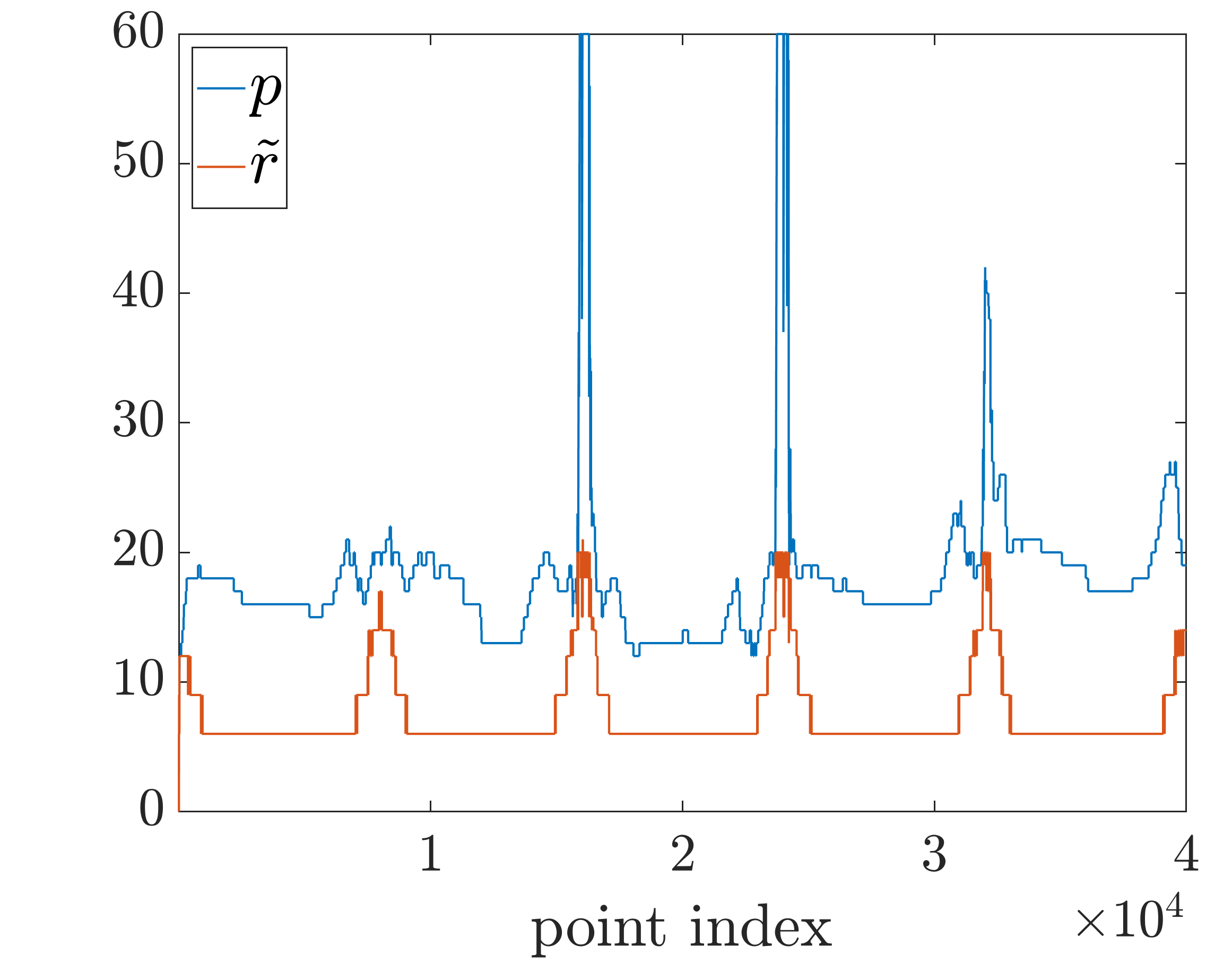}\label{fig:skewlap3d-tri-stats}}
\par\medskip
\subfloat[\centering recycling points]{\includegraphics[width=0.32\linewidth]{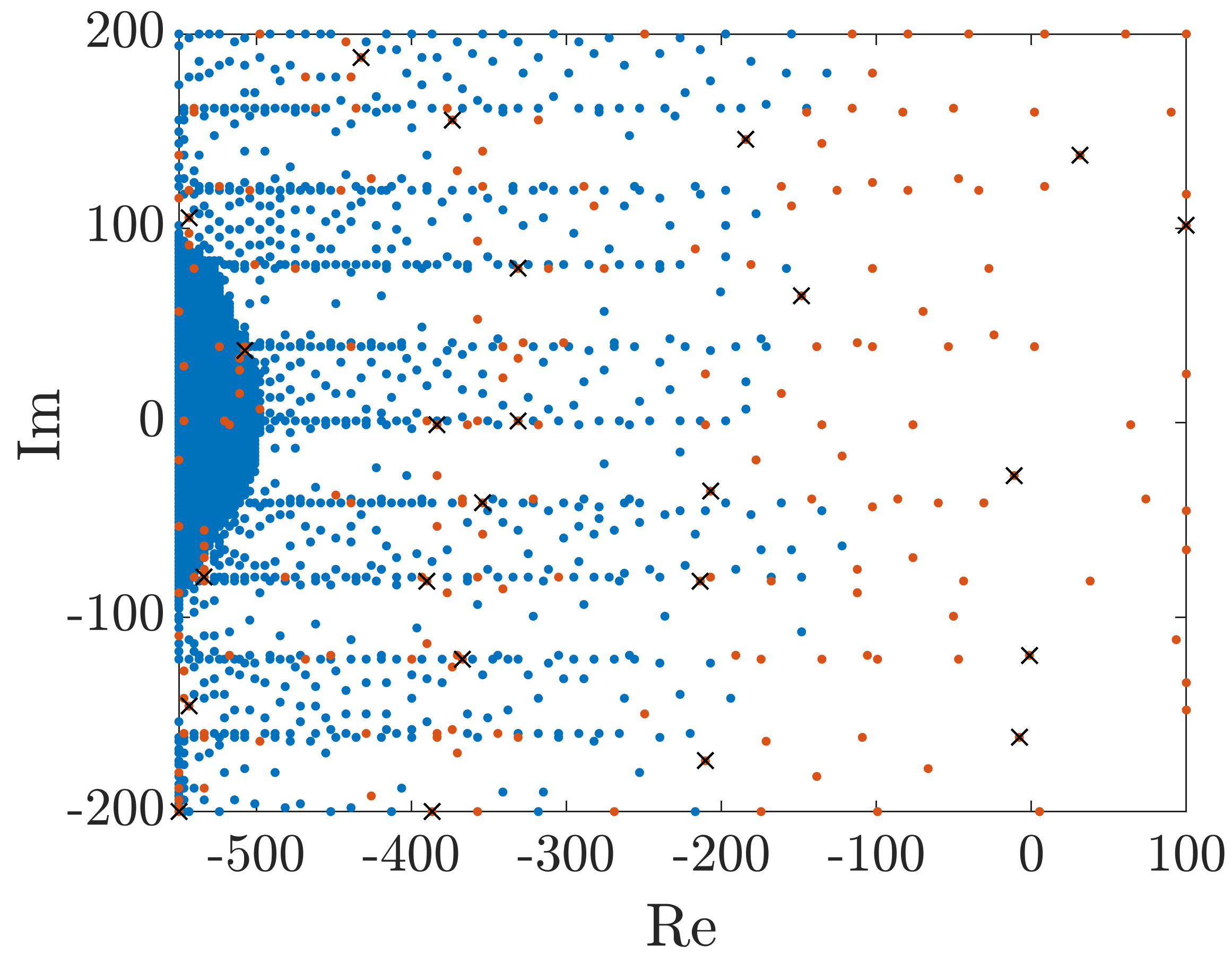}\label{fig:skewlap3d-pred-pts}}
\hfill
\subfloat[\centering {\tt psPrecRecy} statistics]{\includegraphics[width=0.32\linewidth]{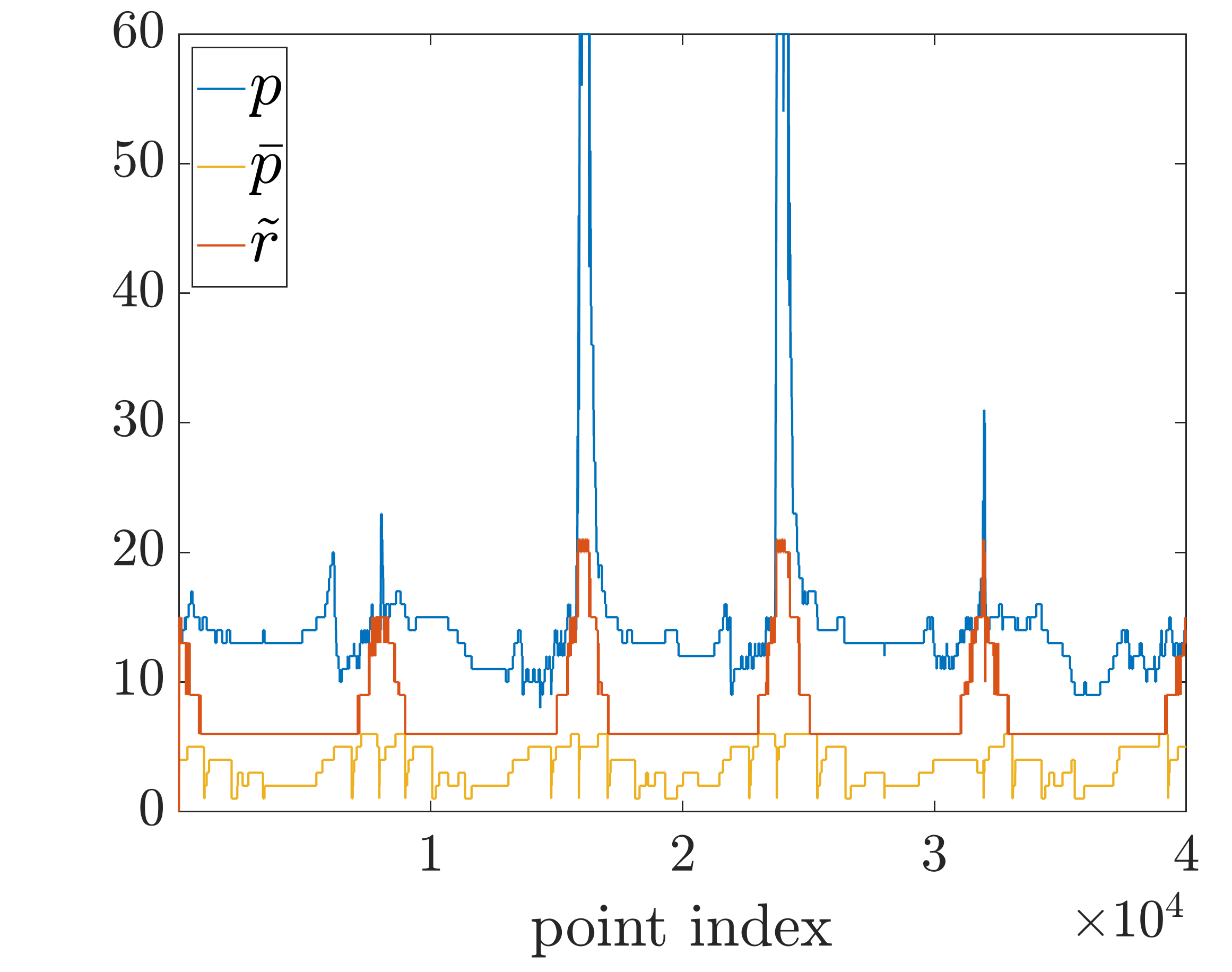}\label{fig:skewlap3d-pred-stats}}
\hfill
\subfloat[\centering execution time breakdown]{\includegraphics[width=0.32\linewidth]{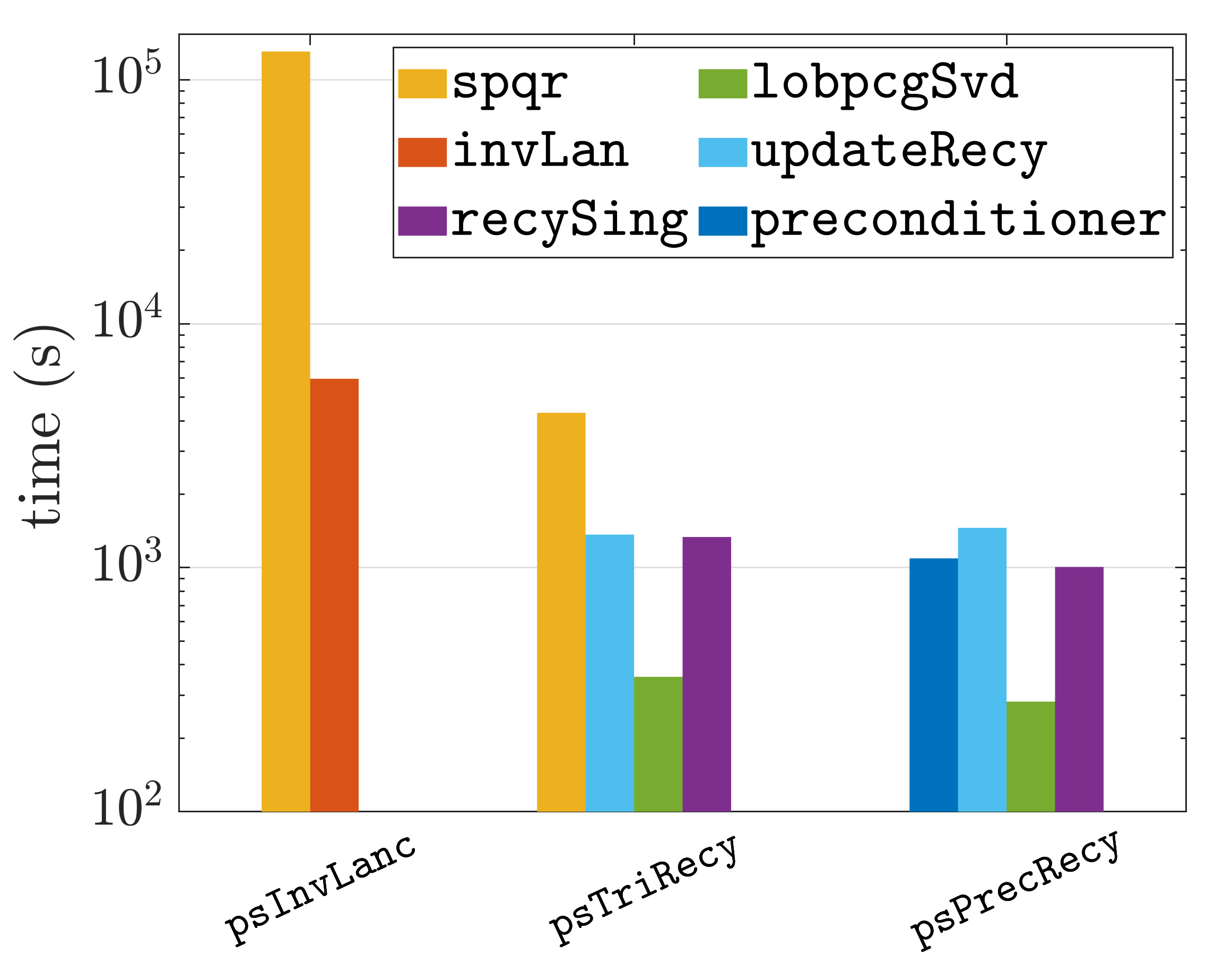}\label{fig:skewlap3d-time}}
\caption{Results for skewlap3d.}
\label{fig:skewlap3d-results}
\end{figure}
Our fourth example, the skewlap3d matrix, is also borrowed from \textsc{EigTool} \cite{wri3}. When the discretization parameter is set to $25$, this matrix serves as the most ill-conditioned sparse example in our test set.

Although the maximum condition number is $8.35\times 10^{13}$ across the grid, \texttt{psTriRecy} produces smooth and well-resolved pseudospectral contours, as shown in \cref{fig:skewlap3d-contour}. As shown in \cref{fig:skewlap3d-pts}, refinement is required at only $1{,}442$ of the $40{,}000$ grid points and occurs mainly in the left-central part of the region of interest, where the local singular value problem is most ill-conditioned. Recycling reduces the time spent on sparse QR from $130{,}590$ seconds for ${\tt psInvLanc}$ to $4{,}313$ seconds for ${\tt psTriRecy}$, as shown in \cref{fig:skewlap3d-time}. The statistics in \cref{fig:skewlap3d-tri-stats} show that the number of recycled blocks reaches the prescribed cap $p_{\max}=60$ near the difficult region. The value of $\tilde r$ is close to $r$ throughout.

When solved using \texttt{psPrecRecy}, this example benefits substantially from preconditioning---only $1{,}092$ seconds are spent on preconditioner-related computations throughout the entire computation; see \cref{fig:skewlap3d-time}. As shown in \cref{fig:skewlap3d-pred-pts}, the grid points at which projection-subspace recycling is performed follow a distribution similar to that of the singular-subspace recycling points. There are $2{,}027$ singular-subspace recycling points, $205$ projection-subspace recycling points, and $24$ base-preconditioner reconstructions. In the region where the recycling points are concentrated, the method uses only one base preconditioner. The statistics in \cref{fig:skewlap3d-pred-stats} show that $p$ and $\tilde r$ follow patterns similar to those in \cref{fig:skewlap3d-tri-stats}, while the dimension $\bar p$ of the recycled projection subspace remains small throughout the computation. The total execution time is reduced from $136{,}575$ seconds for ${\tt psInvLanc}$ to $7{,}390$ seconds for ${\tt psTriRecy}$ and further to $3{,}968$ seconds for ${\tt psPrecRecy}$, corresponding to speedups of $18.48$ and $34.42$, respectively.

\subsection{sparserandom}
Finally, we consider the sparse random matrix generated by \texttt{sparserandom\_demo.m} from \cite{wri3} with dimension $n=4{,}000$. This example represents a different challenge---although the matrix is not severely ill-conditioned, its QR factors are much denser and therefore expensive to apply.

The pseudospectral contours computed by ${\tt psTriRecy}$ are shown in \cref{fig:sprand-contour}. In contrast to the previous sparse examples, the $3{,}352$ grid points where singular-subspace recycling is performed by ${\tt psTriRecy}$ are distributed rather uniformly over the computational region; see \cref{fig:sprand-pts}. Nevertheless, recycling still reduces the sparse QR time from $89{,}528$ seconds for ${\tt psInvLanc}$ to $6{,}990$ seconds for ${\tt psTriRecy}$, as shown by the golden yellow bars in \cref{fig:sprand-time}. The statistics in \cref{fig:sprand-tri-stats} show that $p$ remains modest, mostly below $20$, and that $\tilde r=r$ throughout the computation, consistent with the moderate conditioning of this problem.

This example particularly favors ${\tt psPrecRecy}$, with preconditioner-related computations taking only $500$ seconds (the second purple bar in \cref{fig:sprand-time}). Among the $4{,}654$ singular-subspace recycling points in ${\tt psPrecRecy}$, there are only $461$ projection-subspace recycling points and $11$ base-preconditioner reconstructions; see \cref{fig:sprand-pred-pts}. The total execution time drops from $92{,}846$ seconds for ${\tt psInvLanc}$ to $7{,}848$ seconds for ${\tt psTriRecy}$ and further to $1{,}606$ seconds for ${\tt psPrecRecy}$. The speedup over ${\tt psInvLanc}$ therefore increases from $11.83$ for ${\tt psTriRecy}$ to $57.83$ for ${\tt psPrecRecy}$.

\begin{figure}[t!]
\centering
\subfloat[\centering pseudospectral contour]{\includegraphics[width=0.32\linewidth]{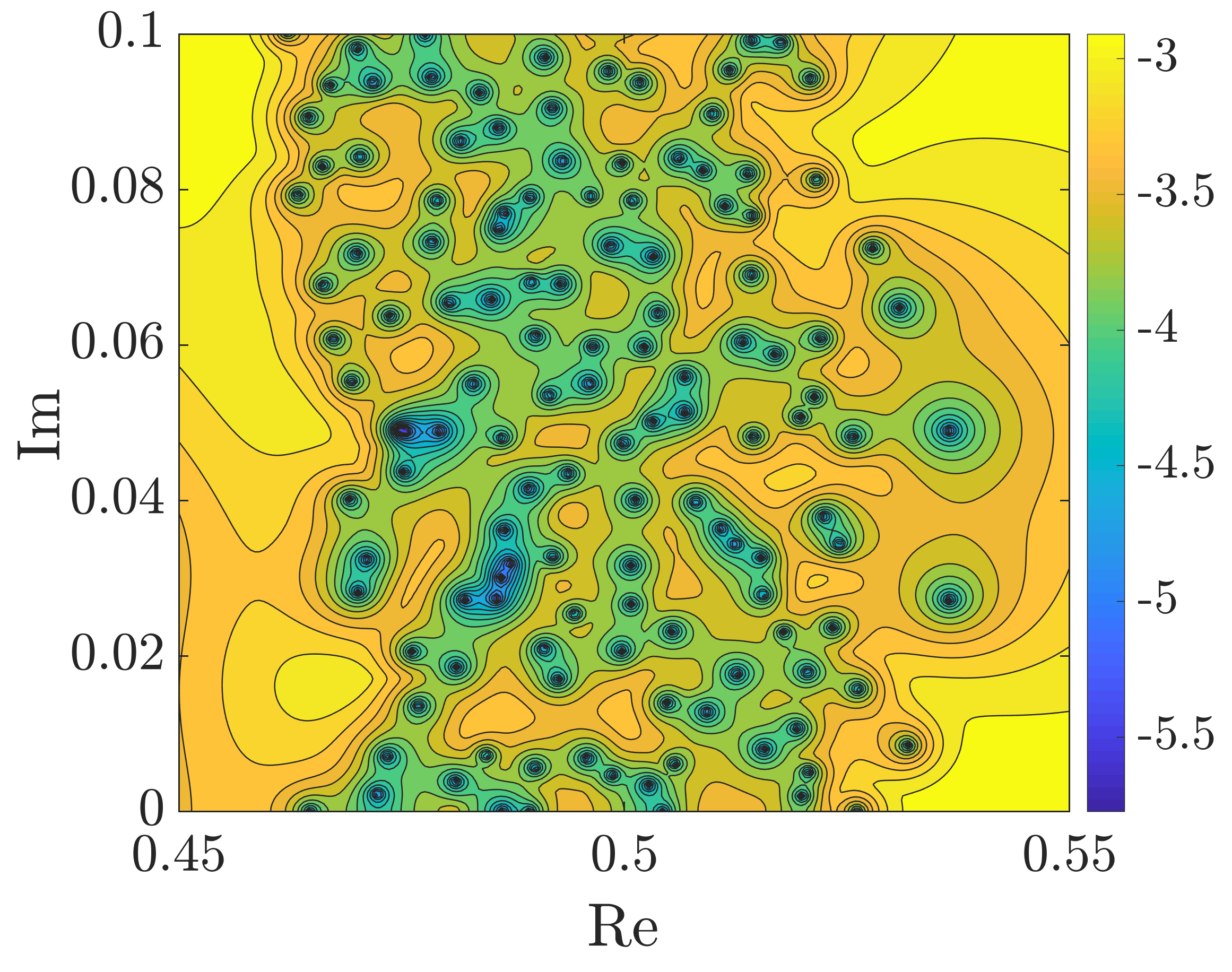}\label{fig:sprand-contour}}
\hfill
\subfloat[\centering recycling points]{\includegraphics[width=0.32\linewidth]{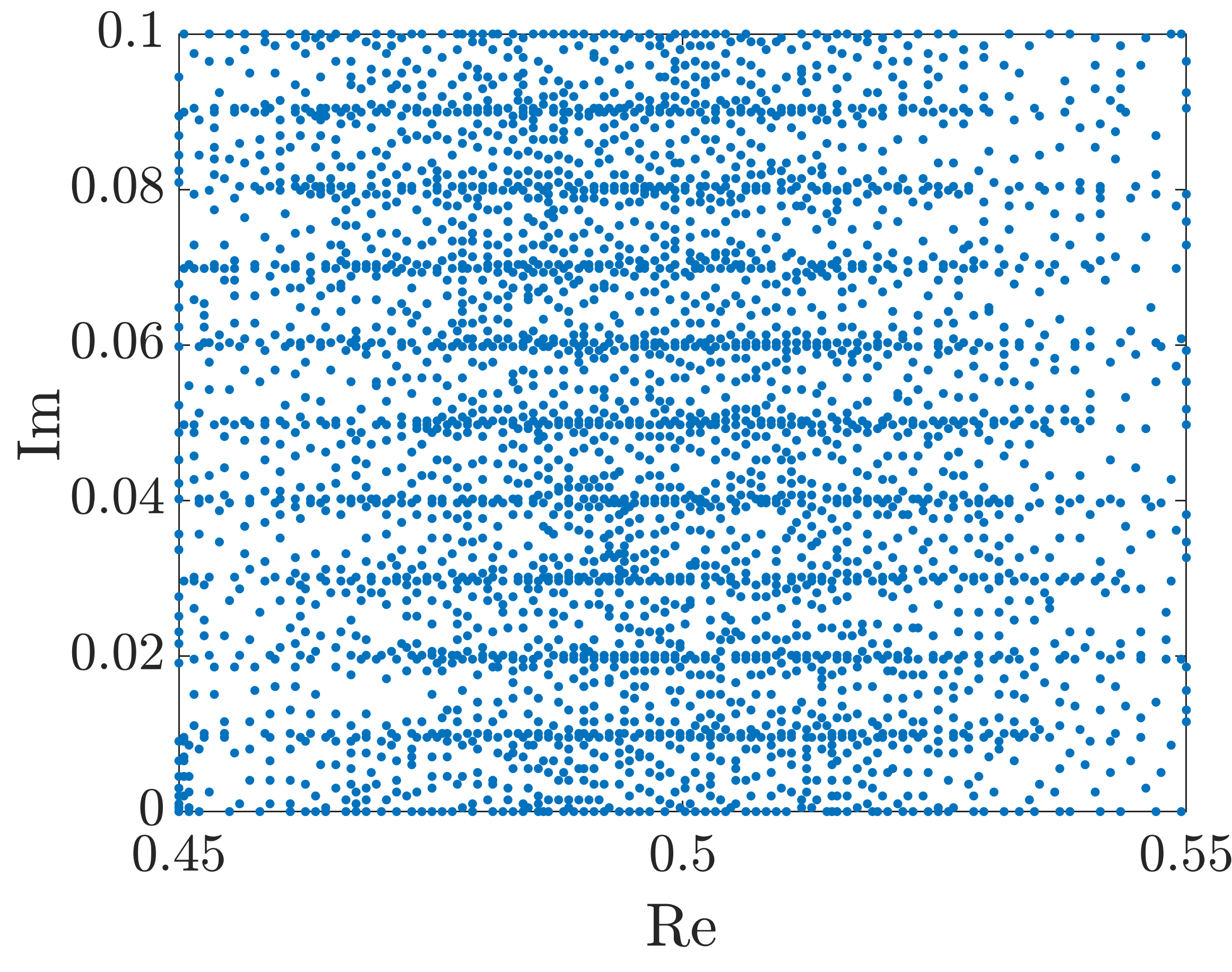}\label{fig:sprand-pts}}
\hfill
\subfloat[\centering {\tt psTriRecy} statistics]{\includegraphics[width=0.32\linewidth]{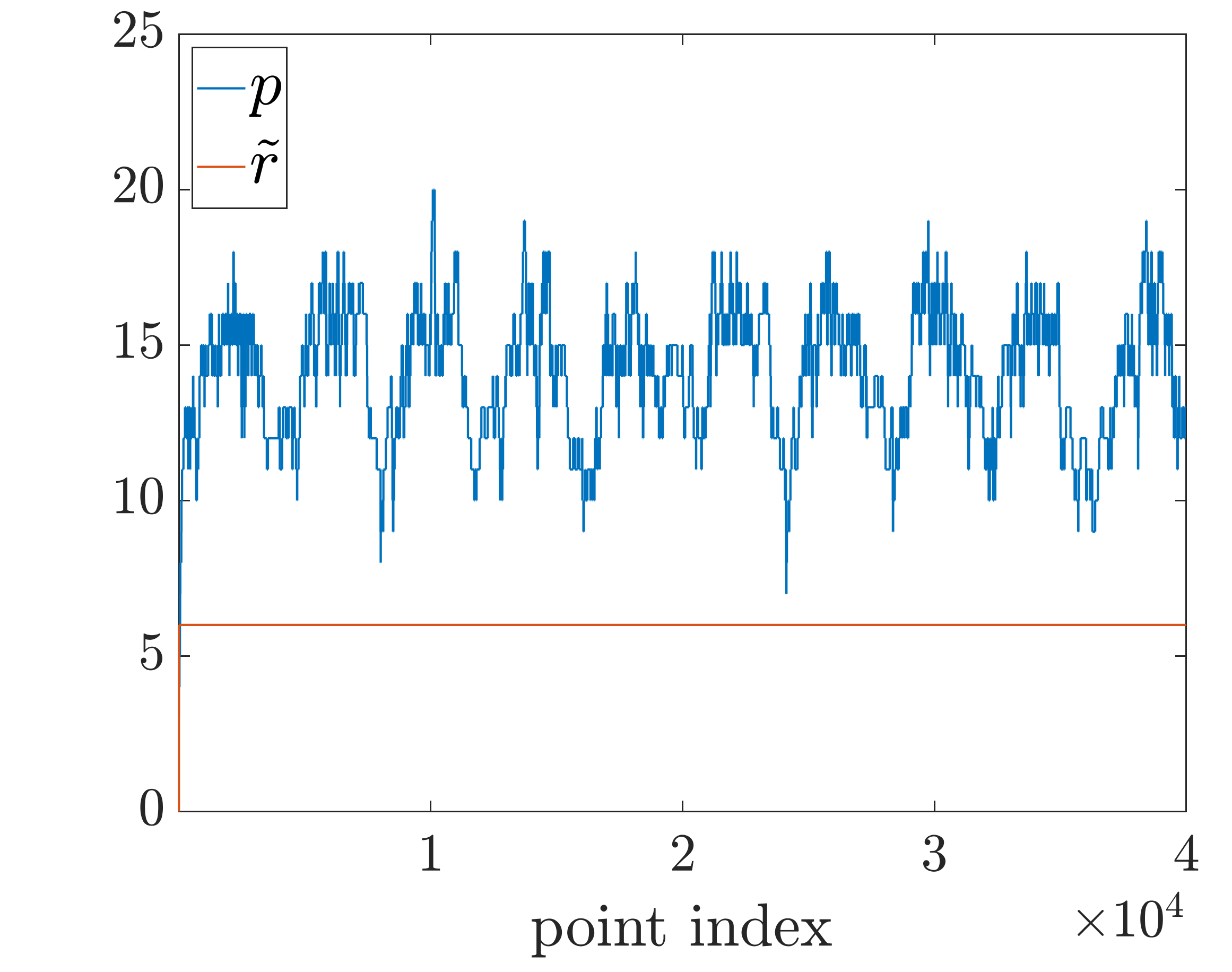}\label{fig:sprand-tri-stats}}
\par\medskip
\subfloat[\centering recycling points]{\includegraphics[width=0.32\linewidth]{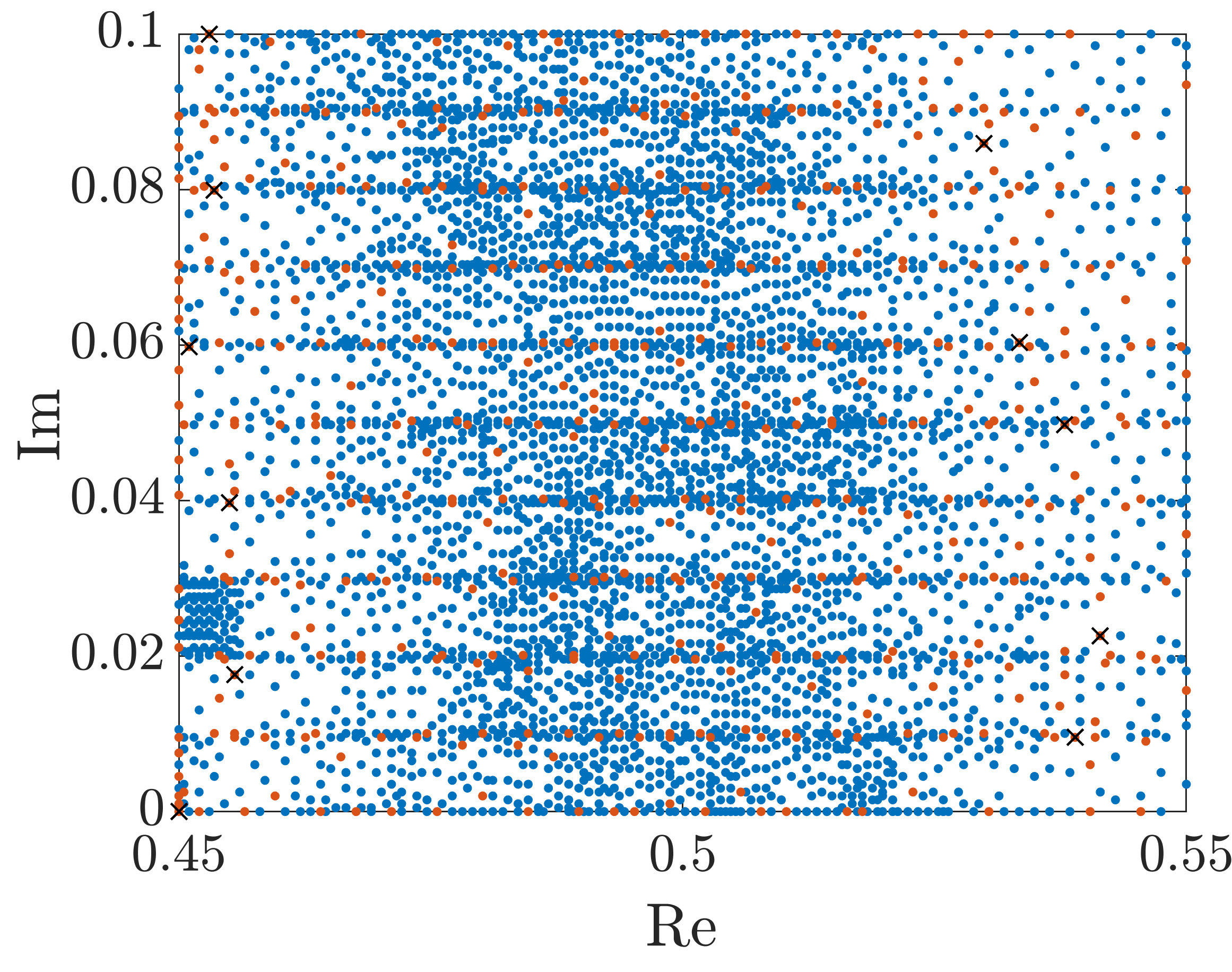}\label{fig:sprand-pred-pts}}
\hfill
\subfloat[\centering {\tt psPrecRecy} statistics]{\includegraphics[width=0.32\linewidth]{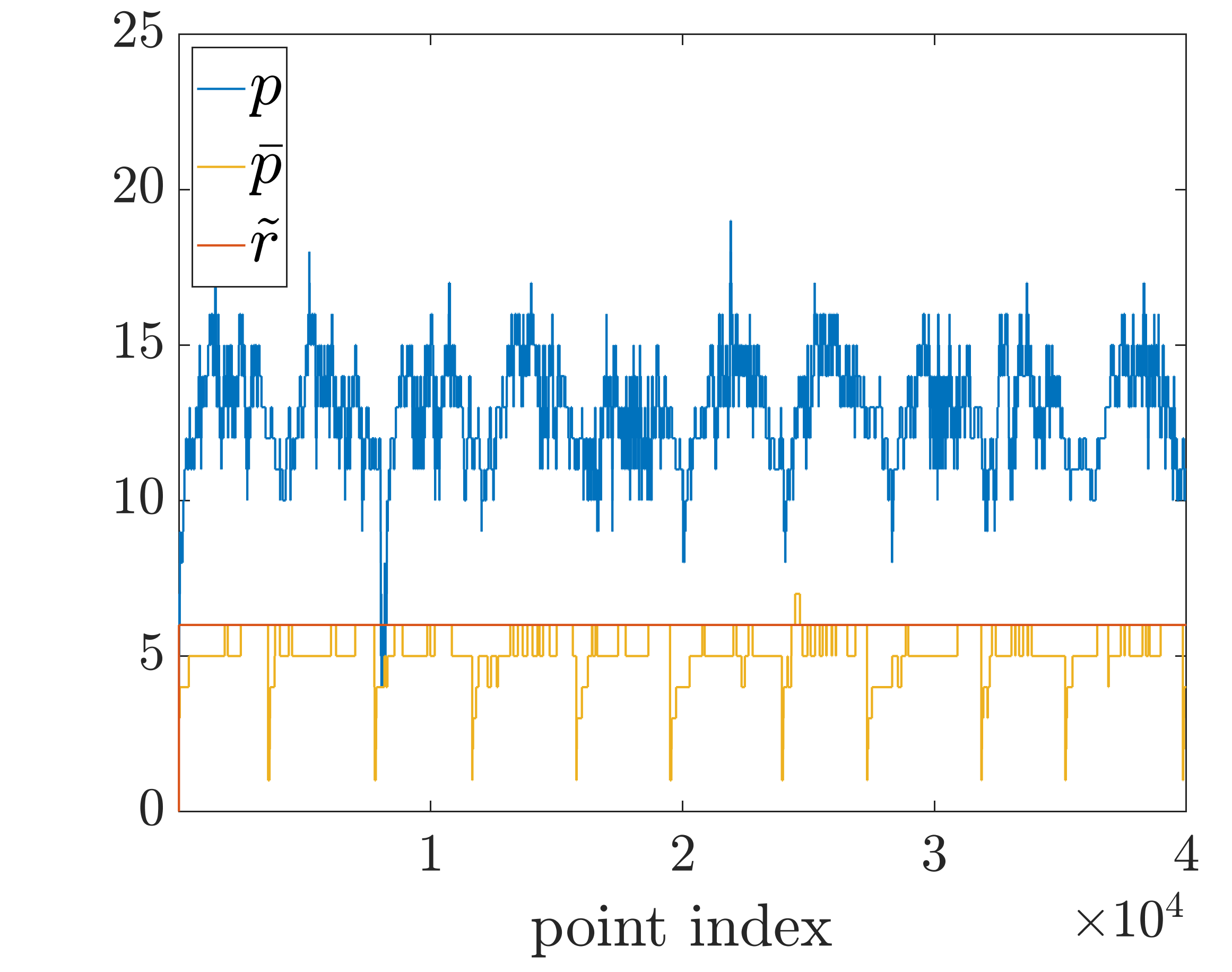}\label{fig:sprand-pred-stats}}
\hfill
\subfloat[\centering execution time breakdown]{\includegraphics[width=0.32\linewidth]{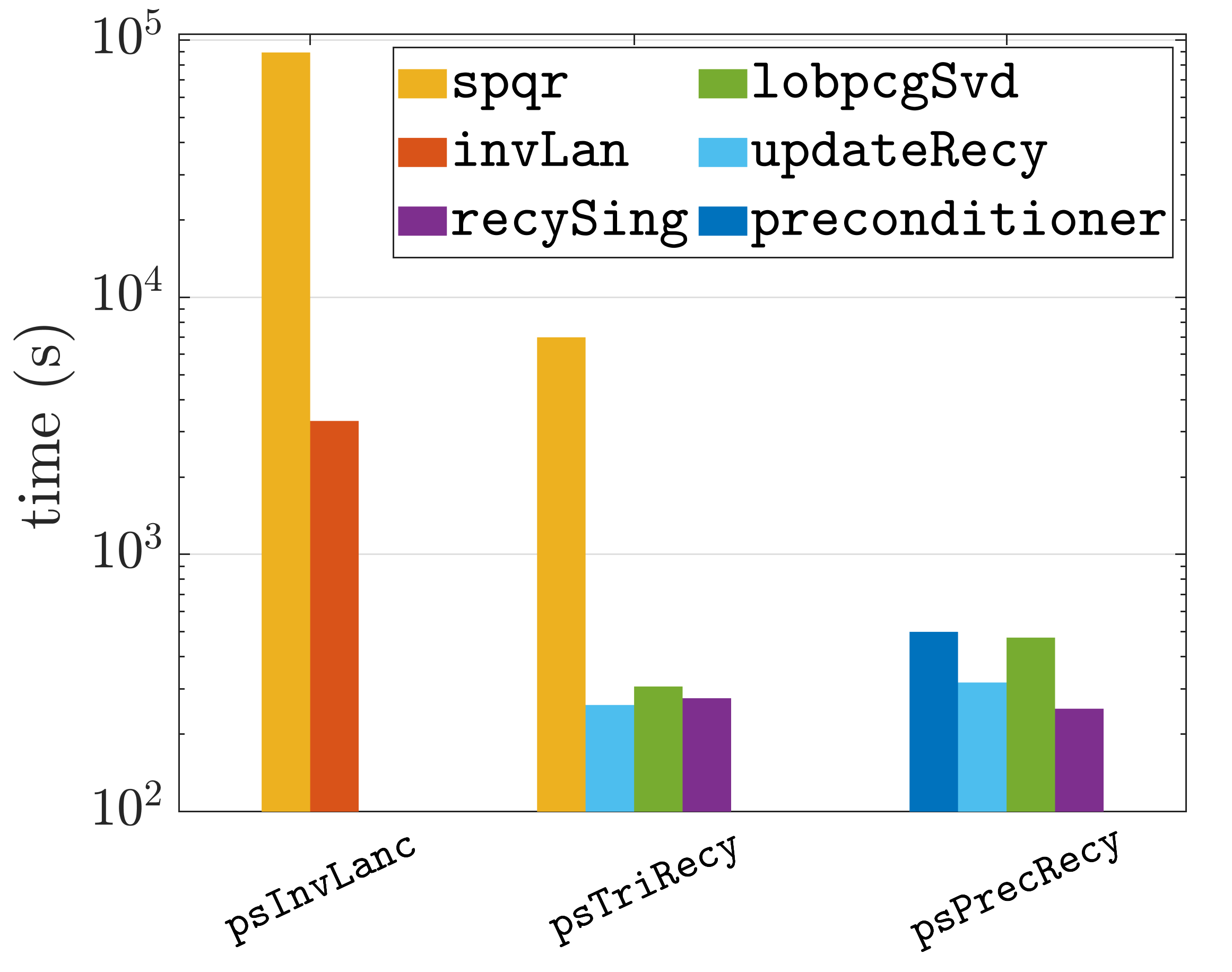}\label{fig:sprand-time}}
\caption{Results for sparserandom.}
\label{fig:sprand-results}
\end{figure}

\section{Outlook}\label{sec:out}
The success of the recycling strategies in accelerating the computation of pseudospectra encourages us to explore their applications and further enhancements.

\subsection{Extension to linear operators in Hilbert spaces}
On the one hand, the fast algorithms proposed in this paper complement the operator-oriented algorithm for computing the pseudospectra of linear operators \cite{den2}. On the other hand, nothing prevents us from applying similar recycling strategies to the operator-oriented algorithm for further speedup.

For example, the singular subspace recycling idea extends naturally to linear operators on Hilbert spaces. Let $\mH$ be a Hilbert space and $\mL:\mathcal D(\mL)\subset \mH \to \mH$ be a closed densely defined linear operator and define
\begin{align}
\mN(z) = (z\mI-\mL)^*(z\mI-\mL) = \mL^*\mL - z\mL^* - z^*\mL + |z|^2\mI,
\end{align}
where $\mI$ is the identity operator on $\mH$. The subspace $\mathcal X \subset \mathcal D(\mN(z))$ spanned by the eigenfunctions of $\mN(z)$ associated with its smallest eigenvalues can be computed using the operator Lanczos method described in \cite{den2}. We recycle such subspaces at select grid points by forming a recycling subspace $\mathcal V = \operatorname{span}\{V\}\subset\mathcal D(\mN(z))$, where $V = [v_1, \ldots, v_{pr}]$ is an orthonormal basis for $\mathcal V$ in $\mH$. The fast Rayleigh--Ritz-SVD procedure in \cref{sec:recycsing} is then applicable to $\mN(z)$ and $\mathcal V$ with \cref{H} replaced by
\begin{align*}
H_1 = \left\langle \mL V, \mL V \right\rangle, \quad
H_2 = \left\langle \mL V, V \right\rangle, \quad
H_3 = \left\langle V, V \right\rangle = I_{pr}.
\end{align*}
Here, $\left\langle \cdot, \cdot \right\rangle$ denotes the inner product in $\mH$. The recycling is followed by an adaptive update of the recycling subspace similar to that described in \cref{sec:upSingSp}, which carries over to the Hilbert space setting.

\subsection{Two-sided projection subspace recycling}
In the projection subspace recycling, we only use the projection subspace associated with the smallest eigenvalues of $(N(z), K)$. One may also recycle the subspace of the largest eigenvalues. Using the notation in \cref{twolevelCond}, the two-sided projection subspace is given by
\begin{align}
\bar{E} = [\bar{e}_1, \ldots, \bar{e}_{\bar{r}}, \bar{e}_{n-\bar{r}+1}, \ldots, \bar{e}_{n}].
\end{align}
If the two-level preconditioner in \cref{twolevel} is replaced by the balanced form
\begin{align}
K^{-1}_{Z} = (I_n - N(z)U)^T K^{-1} (I_n - N(z)U) + U,
\end{align}
where $U = Z(Z^*N(z)Z)^{-1}Z^*$, then under assumptions analogous to those in \cref{twolevelCond} we have
\begin{align}
\kappa_{r}^{\mathrm{eff}}(K_{\bar{E}}^{-1}N(z))
= \frac{\bar{\lambda}_{n-\bar{r}}}{\bar{\lambda}_{\bar{r}+r+1}}. \label{cond1}
\end{align}
Thus, two-sided projection subspace recycling can further reduce the effective condition number by replacing $\bar{\lambda}_{n}$ in \cref{cond} with the smaller quantity $\bar{\lambda}_{n-\bar{r}}$. This is particularly useful when the upper-bound criterion in \cref{tauRecycle} tends to trigger frequent reconstructions of the base preconditioner during the grid traversal.

\subsection{Recycling the domain decomposition preconditioner}
The preconditioner recycling strategy developed in the current study is also applicable to two-level domain decomposition methods. A two-level Schwarz preconditioner combines a local base solver with a coarse space, and robust algebraic constructions of such coarse spaces have been developed for sparse matrices and normal equations \cite{ald1,ald2,ald3}. For very large sparse problems, the incomplete factorization used in this paper may be replaced by an algebraic Schwarz preconditioner, offering a potentially more robust alternative. In this setting, one may keep the Schwarz preconditioner fixed over multiple nearby grid points and recycle only the coarse space, which plays the same role as the projection subspace in our two-level preconditioner.

\subsection{Other grid traversal strategies}
The batch-zigzag path is certainly not the only strategy to traverse the grid. Other locality-preserving traversal strategies could serve equally well. For instance, a Hilbert traversal \cite{bad} provides strong locality preservation and reduces the directional bias of row- or column-wise sweeps, whereas the Morton (also known as Z-order) traversal recursively visits local blocks and is easy to implement \cite{bad}. A systematic study of the traversal patterns could help us identify the optimal strategy, if one exists.

\clearpage
\section*{Supplementary Material}
\setcounter{section}{0}
\setcounter{subsection}{0}
\setcounter{equation}{0}
\setcounter{algorithm}{0}
\setcounter{figure}{0}
\setcounter{table}{0}
\renewcommand{\thesection}{SM\arabic{section}}
\renewcommand{\thesubsection}{\thesection.\arabic{subsection}}
\renewcommand{\theequation}{\thesection.\arabic{equation}}
\renewcommand{\thealgorithm}{\thesection.\arabic{algorithm}}
\renewcommand{\thefigure}{SM\arabic{figure}}
\renewcommand{\thetable}{SM\arabic{table}}
\makeatletter
\@addtoreset{algorithm}{section}
\makeatother

\section{LOBPCG algorithms}\label{sec:lobpcg}
LOBPCG is widely used for computing the extremal eigenvalues and corresponding eigenvectors of generalized eigenvalue problems (GEPs) involving a pair of Hermitian positive-definite matrices. LOBPCG is often preferable to the Lanczos method when a good preconditioner is available, since it can significantly reduce the number of iterations and thereby achieve convergence rates comparable to those of the inverse Lanczos method. In this supplementary section, we first review the standard LOBPCG method, upon which we develop an LOBPCG method for singular value decomposition. This method is referred to as LOBPCG-SVD in the main text.

\subsection{Standard LOBPCG}
Consider the GEP for Hermitian matrices $A,B\in\mathbb{C}^{n\times n}$,
\begin{align*}
AX = BX \Lambda,
\end{align*}
where $\Lambda = \operatorname{diag}\{\lambda_1, \dots, \lambda_r\}$ contains the $r$ smallest eigenvalues and $X\in\mathbb{C}^{n\times r}$ consists of the corresponding eigenvectors. Given a preconditioner $K\in\mathbb{C}^{n\times n}$, the LOBPCG algorithm constructs the search basis
\begin{align*}
S^{(j)} = [X^{(j)}\ W^{(j)}\ P^{(j)}] \in \mathbb{C}^{n\times 3r},
\end{align*}
at the $j$th iteration for $j \geq 2$. Here, $X^{(j)}\in\mathbb{C}^{n\times r}$ contains the current approximate eigenvectors. The matrix $W^{(j)}\in\mathbb{C}^{n\times r}$ consists of the preconditioned residuals
\begin{align*}
W^{(j)} = K^{-1}(AX^{(j)} - BX^{(j)}\Theta^{(j)}), \label{residual}
\end{align*}
where $\Theta^{(j)} = \operatorname{diag}\{\theta_1^{(j)}, \dots, \theta_r^{(j)}\}$ is the diagonal matrix of Rayleigh quotients associated with $X^{(j)}$. The matrix $P^{(j)}\in\mathbb{C}^{n\times r}$ contains the conjugate directions from previous iterations. For the first iteration,
\begin{align*}
S^{(1)} = [X^{(1)}\ W^{(1)}] \in \mathbb{C}^{n\times 2r}.
\end{align*}
The approximate eigenpairs are updated by first performing the Rayleigh--Ritz procedure
\begin{align}
(S^{(j)*}AS^{(j)})G^{(j+1)} = (S^{(j)*}BS^{(j)})G^{(j+1)} \Theta^{(j+1)}, \label{rr}
\end{align}
where the leading $r\times r$ block of $\Theta^{(j+1)}$ provides the updated approximate eigenvalues. For efficient and stable updates of $X^{(j)}$ and $P^{(j)}$, the technique of Hetmaniuk and Lehoucq \cite{het,due} is applied to $G^{(j+1)}$ to obtain matrices $G_X^{(j+1)}\in\mathbb{C}^{3r\times r}$ and $G_P^{(j+1)}\in\mathbb{C}^{3r\times r}$. Finally, the approximate eigenvectors and conjugate directions are updated as
\begin{align*}
X^{(j+1)} = S^{(j)}G_X^{(j+1)}, \quad
P^{(j+1)} = S^{(j)}G_P^{(j+1)}.
\end{align*}

\subsection{LOBPCG for SVD}\label{sec:lobpcgsvd}

Consider a general matrix $C \in \mathbb{C}^{m \times n}$. Let $\Sigma = \operatorname{diag}\{\sigma_1, \dots, \sigma_r\}$ be the diagonal matrix containing the $r$ smallest singular values of $C$, and let $Y \in \mathbb{C}^{m \times r}$ and $X \in \mathbb{C}^{n \times r}$ be the matrices consisting of the corresponding left and right singular vectors, respectively. These matrices satisfy
\begin{align}
CX = Y\Sigma, \quad C^*Y = X\Sigma. \label{svd}
\end{align}
Since the eigenvalues of $C^*C$ are the squares of the singular values of $C$, and the eigenvectors of $C^*C$ coincide with the right singular vectors of $C$, we adapt the standard LOBPCG algorithm, originally formulated for symmetric GEPs, to the matrix pair $(C^*C, I_n)$ for computing the factorization in \cref{svd}. The Rayleigh--Ritz problem \cref{rr} then becomes
\begin{align}
(S^{(j)*}C^*CS^{(j)})G^{(j+1)} = (S^{(j)*}S^{(j)})G^{(j+1)} (\Xi^{(j+1)})^2, \label{RRSvdNorm1}
\end{align}
where $\Xi^{(j+1)} = \operatorname{diag}\{\xi^{(j+1)}_1, \dots, \xi^{(j+1)}_r\}$ contains the Ritz singular values. Solving \cref{RRSvdNorm1} directly would reduce the accuracy of the computed singular values and singular vectors, since $\operatorname{cond}(C^*C) = \operatorname{cond}(C)^2$. Furthermore, when $\operatorname{cond}(C) > \ema^{-1/2}$, the Rayleigh--Ritz problem \cref{RRSvdNorm1} becomes so ill-conditioned that the computed eigenpairs may lose all accuracy, or the eigensolver may even break down.

To address this issue, we enforce orthogonality on $S^{(j)}$, i.e.,
\begin{align*}
S^{(j)*}S^{(j)} = I_{3r},
\end{align*}
and compute the thin QR factorization of $CS^{(j)}$:
\begin{align}
CS^{(j)} = Q^{(j)}R^{(j)}, \label{QRAS}
\end{align}
Because of the orthogonality of $Q^{(j)}$, \cref{RRSvdNorm1} reduces to
\begin{align}
(R^{(j)*}R^{(j)})G^{(j+1)} = G^{(j+1)} (\Xi^{(j+1)})^2. \label{RRSvdNorm2}
\end{align}
Instead of solving \cref{RRSvdNorm2}, we compute the SVD of $R^{(j)}$ directly:
\begin{align}
R^{(j)} = J^{(j+1)}\Xi^{(j+1)}(G^{(j+1)})^*, \label{RRSvdNorm}
\end{align}
where $\Xi^{(j+1)}$ and $G^{(j+1)}$ represent the same quantities as in \cref{RRSvdNorm1}, but are obtained here with greater numerical stability. We refer to the QR factorization of \cref{QRAS} and SVD factorization of \cref{RRSvdNorm} collectively as the \emph{Rayleigh--Ritz-SVD procedure}.

In each iteration, we estimate the number of singular values and singular vectors that have converged by finding $r_c$ such that, for $k = 1, \ldots, r_c$, the residuals of the Ritz singular pairs satisfy
\begin{align}
\lVert C^*C X_{k{:}k} - (\xi^{(j+1)}_k)^2 X_{k{:}k}\rVert \leq \max \{\eta \lVert C \rVert \xi^{(j+1)}_k ,\ \lVert C \rVert^2\ema\}, \label{tolSVD}
\end{align}
where $\eta$ is a user-specified tolerance in the backward-stability criterion \cite{bag}, and the term $\lVert C \rVert^2\ema$ prevents unnecessary iterations.

The remaining components of the LOBPCG algorithm for SVD (LOBPCG-SVD) parallel those for the GEP. The complete algorithm is summarized in \cref{alg:lobpcgSvd}, where \texttt{hl} denotes the subroutine implementing the Hetmaniuk--Lehoucq technique. In line 10 of \cref{alg:lobpcgSvd}, the products $CX^{(j)}$ and $CP^{(j)}$ can be obtained either by explicit computation or by implicit updates \cite{due}. For efficiency, we adopt the latter approach.

\begin{algorithm}[tbhp]
  \renewcommand{\algorithmicrequire}{\textbf{Input:}}
  \renewcommand{\algorithmicensure}{\textbf{Output:}}
  \algrenewcommand{\algorithmicrepeat}{\textbf{do}\ }
  \algrenewcommand{\algorithmicuntil}{\textbf{while}}
  \caption{The LOBPCG algorithm for SVD}\label{alg:lobpcgSvd}
  \begin{spacing}{1.05}
  \begin{algorithmic}[1]
    \Require $C \in \mathbb{C}^{m \times n}$, preconditioner $K$, initial approximate right singular vectors $X^{(0)} \in \mathbb{C}^{n \times r}$ that satisfy $(X^{(0)})^*X^{(0)} = I_r$, and desired number $r_v$ of converged right singular vectors ($r_v \leq r$).
    \Ensure Converged right singular vectors $X \in \mathbb{C}^{n \times r}$ and diagonal matrix $\Sigma \in \mathbb{C}^{r \times r}$ containing the converged singular values.
    \AlgoOutputRule
    \Function{$[X,\Sigma] = \tt{lobpcgSvd}$}{$C, K, X^{(0)}, r_v$}
      \State $[\sim, R^{(0)}] = {\tt{qr}}(CX^{(0)})$, $[\sim, \Xi^{(1)}, G^{(1)}] = {\tt{svd}}(R^{(0)})$
      \State $X^{(1)} = X^{(0)}G^{(1)}$
      \State $W^{(1)} = C^*(CX^{(1)}) - X^{(1)}(\Xi^{(1)})^2$
      \State $P^{(1)} = [~]$
      \For{$j = 1,2,\dots$}
        \State $W^{(j)} \gets K^{-1}W^{(j)}$
        \State $W^{(j)} \gets W^{(j)} - [X^{(j)}\ P^{(j)}]([X^{(j)}\ P^{(j)}]^*W^{(j)})$
        \State $[W^{(j)}, \sim] = {\tt{qr}}(W^{(j)})$
        \State $S^{(j)} = [X^{(j)}\ W^{(j)}\ P^{(j)}]$
        \State $[\sim, R^{(j)}] = {\tt{qr}}([CX^{(j)}\ CW^{(j)}\ CP^{(j)}])$
        \State $[\sim, \Xi^{(j+1)}, G^{(j+1)}] = {\tt{svd}}(R^{(j)})$, $\Xi^{(j+1)} \gets \Xi^{(j+1)}_{1{:}r, 1{:}r}$
        \State $[G_X^{(j+1)}, G_P^{(j+1)}] = {\tt{hl}}(G^{(j+1)})$
        \State $X^{(j+1)} = S^{(j)}G_X^{(j+1)}$
        \State $P^{(j+1)} = S^{(j)}G_P^{(j+1)}$
        \State $W^{(j+1)} = C^*(CX^{(j+1)}) - X^{(j+1)}(\Xi^{(j+1)})^2$
        \State Determine $r_c$ so that \cref{tolSVD} holds for $k = 1, \ldots, r_c$.
        \If{$r_c \geq r_v$}
          \State Return $X = X^{(j+1)}$ and $\Sigma = \Xi^{(j+1)}$
        \EndIf
      \EndFor
    \EndFunction
  \end{algorithmic}
  \end{spacing}
\end{algorithm}

\section{Preliminary reductions for dense and sparse matrices}\label{sec:prelred}
In this supplementary appendix, we describe how the reduced form
\begin{align}
M(z) = M - zS
\end{align}
used in \cref{sec:prelimred} of the main text is obtained for dense and sparse matrices. Throughout, for a given matrix $C\in \mathbb{C}^{m\times n}$ we assume $m \ge n$. In the algorithms of the main text, the routine $[M, S] = $ \texttt{preliminary}$(C)$ refers to the appropriate reduction described below that returns matrices $M$ and $S$ such that $C-zI$ and $M-zS$ have the same singular values.

\subsection{Dense square matrices}
For square matrices, we first computes a Schur decomposition
\begin{align}
C = Q R Q^*,
\end{align}
where $Q$ is unitary and $R$ is upper triangular. Thus,
\begin{align}
C-zI_n = Q(R-zI_n)Q^*,
\end{align}
from which it follows that $C-zI_n$ and $R-zI_n$ have the same singular values for all $z$. Hence,
\begin{align}
M = R, \qquad S = I_n,
\end{align}
and the thin QR factorization of \cref{triGen} is skipped as $Q(z) = I$ and $R(z)=R-zI_n$. See \cite[\S 39]{tre2} for more details.

\subsection{Dense rectangular matrices with \texorpdfstring{$m \ge 2n$}{m >= 2n}}
Let
\begin{align}
I = \begin{bmatrix} I_n \\ 0 \end{bmatrix} \in \mathbb{C}^{m \times n}
\end{align}
and partition $C$ conformably as
\begin{align}
C = \begin{bmatrix} C_u \\ C_l \end{bmatrix},
\end{align}
where $C_u \in \mathbb{C}^{n \times n}$ and $C_l \in \mathbb{C}^{(m-n)\times n}$. Since only the leading $n \times n$ block of $C(z) = C-zI$ depends on $z$, we first QR factorize the lower part of $C$ by
\begin{align}
C_l = Q_l \begin{bmatrix} R_l \\ 0 \end{bmatrix},
\end{align}
where $Q_l \in \mathbb{C}^{(m-n)\times(m-n)}$ is unitary and $R_l \in \mathbb{C}^{n\times n}$ is upper triangular. Pre-applying the unitary transformation $P = \operatorname{diag}(I_n,Q_l^*)$ to $C(z)$ gives
\begin{align}
P(C-zI)
=
\begin{bmatrix}
C_u-zI_n\\
R_l\\
0
\end{bmatrix}.
\end{align}
After removing the zero rows in the bottom, we obtain a $2n\times n$ matrix pencil
\begin{align}
M(z)=
\begin{bmatrix}
C_u\\
R_l
\end{bmatrix}
- z
\begin{bmatrix}
I_n\\
0
\end{bmatrix},
\end{align}
which is trapezoidal and has the same singular values as $C(z)$. Hence, we take
\begin{align}
M=\begin{bmatrix} C_u\\ R_l \end{bmatrix},
\qquad
S=\begin{bmatrix} I_n\\ 0 \end{bmatrix}.
\end{align}
See \cite[\S 4.1]{tre2} for more details.

% For each $z$, we further performs QR factorization of $M(z)$ and then discard the zero rows in the resulting upper triangular factor, obtaining the $n\times n$ triangular matrix whose smallest singular value equals that of $C(z)$. As shown in \cite{wri2}, this further QR reduction costs $\mO(n^3)$ flops for each $z$. 

\subsection{Dense rectangular matrices with \texorpdfstring{$n < m < 2n$}{n <= m < 2n}}
When $n < m < 2n$, the $z$-independent block of $C(z)$ is smaller than the $z$-dependent one. In this case, we work with a different partition
\begin{align}
I = \begin{bmatrix} I_u \\ I_l \end{bmatrix},
\qquad
C = \begin{bmatrix} C_u \\ C_l \end{bmatrix},
\end{align}
where $I_u,C_u \in \mathbb{C}^{(m-n)\times n}$ and $I_l,C_l \in \mathbb{C}^{n\times n}$. We then compute the generalized Schur decomposition of the square pencil $(I_l,C_l)$ as
\begin{align}
Q(zI_l-C_l)Z = zT_l-S_l,
\end{align}
where $Q,Z \in \mathbb{C}^{n\times n}$ are unitary and $S_l,T_l \in \mathbb{C}^{n\times n}$ are upper triangular. Let $\tilde{I}$ be the ${(m-n)\times (m-n)}$ identity matrix and $P = \operatorname{diag}(\tilde{I}, Q)$. Pre-multiplying $P$ and post-multiplying $Z$ to $C(z)$ gives
\begin{align}
P(zI-C)Z=
z
\begin{bmatrix}
I_uZ\\
T_l
\end{bmatrix}
-
\begin{bmatrix}
C_uZ\\
S_l
\end{bmatrix},
\end{align}
Thus, $M-zS$ has the same singular values as $C(z)$ for any $z$, where
\begin{align}
M=
\begin{bmatrix}
C_uZ\\
S_l
\end{bmatrix},
\qquad
S=
\begin{bmatrix}
I_uZ\\
T_l
\end{bmatrix}.
\end{align}
See \cite[\S 4.2]{tre2} for more details.

% In addition, $M-zS$ is again trapezoidal, and one can compute the triangular factor in \cref{triGen} by banded QR factorization at a cost of $\mO((m-n)n^2)$ flops per grid point. This is usually cheaper than the $\mO(n^3)$ cost in the case of $m \ge 2n$.

\subsection{Sparse matrices}
For sparse matrices, we do not attempt to perform a dense preliminary reduction. Instead, we compute a fill-reducing column permutation $\Pi$ by a symbolic analysis of the sparsity pattern of $N(z) = (C-zI)^*(C-zI)$. Common choices for $\Pi$ include column approximate minimum degree orderings \cite{dav1,dav2}, approximate minimum degree orderings applied to the normal-equations pattern \cite{ame1,ame2}, and nested-dissection-type orderings, for instance those produced by graph partitioning packages such as \textsc{METIS} \cite{kar}. These orderings are routinely used in modern sparse QR packages; see, e.g., \cite{dav3} for details.

Since the nonzero pattern of $N(z)$ is independent of $z$, $\Pi$ can be reused for all grid points. We then apply the same permutation to both $C$ and $I$ so that
\begin{align}
(C-zI)\Pi = M-zS,
\end{align}
where 
\begin{align}
M = C\Pi, \qquad S = I\Pi.
\end{align}
% At each grid point, one may then compute a sparse QR factorization of $M-zS$ and use the resulting triangular factor as in \cref{triGen}.
\bibliographystyle{siamplain}
\bibliography{references}
\end{document}